\documentclass[11pt]{amsart}
\usepackage{amsmath,amsfonts,amsthm,amssymb,amscd}
\def\classification#1{\def\@class{#1}}
\classification{\null}

\DeclareFontFamily{OT1}{rsfs}{}
\DeclareFontShape{OT1}{rsfs}{n}{it}{<-> rsfs10}{}
\DeclareMathAlphabet{\mathscr}{OT1}{rsfs}{n}{it}

\DeclareMathOperator{\Ad}{Ad}

\DeclareMathOperator{\mo}{\,mod}

\DeclareMathOperator{\GL}{GL}
\DeclareMathOperator{\diam}{diam}
\DeclareMathOperator{\pol}{pol}
\DeclareMathOperator{\SL}{SL}
\DeclareMathOperator{\sL}{sl}
\DeclareMathOperator{\PGL}{PGL}
\DeclareMathOperator{\PSL}{PSL}

\DeclareMathOperator{\vdeg}{\overrightarrow{\text{deg}}}

\DeclareMathOperator{\Sp}{Sp}
\DeclareMathOperator{\SO}{SO}
\DeclareMathOperator{\SU}{SU}

\DeclareMathOperator{\charac}{char}

\DeclareMathOperator{\tr}{tr}

\DeclareMathOperator{\Cl}{Cl}
\DeclareMathOperator{\disc}{Disc}

\newtheorem{prop}{Proposition}[section]
\newtheorem{thm}[prop]{Theorem}
\newtheorem*{main}{Main Theorem}

\newtheorem*{conj}{Conjecture}
\newtheorem{cor}[prop]{Corollary}
\newtheorem{lem}[prop]{Lemma}

\newenvironment{Rem}{{\bf Remark.}}{}
\numberwithin{equation}{section}
\title{Growth in $\SL_3(\mathbb{Z}/p\mathbb{Z})$}
\author{H. A. Helfgott}
\address{H. A. Helfgott, School of Mathematics, University of Bristol, Bristol, BS8 1TW, United Kingdom}
\thanks{The author was supported in part by EPSRC grant EP-E054919/1 and
NSF grant DMS-0635607.} 
\begin{document}
\begin{abstract}
Let $G = \SL_3(\mathbb{Z}/p\mathbb{Z})$, $p$ a prime. Let $A$ be a set of
generators of $G$. Then $A$ grows under the group operation.

To be precise: denote by $|S|$ the number of elements of a finite set $S$.
Assume $|A| < |G|^{1-\epsilon}$ for some $\epsilon>0$. Then
 $|A\cdot A\cdot A|>|A|^{1+\delta}$, where $\delta>0$ depends only on 
$\epsilon$. 

We will also study subsets $A\subset G$ that do not generate $G$. Other results on growth and generation follow.
\end{abstract}
\maketitle

\tableofcontents
\section{Introduction}
\subsection{Growth in groups and graphs}
``Growth'' can mean one of many things.
\begin{enumerate}
\item {\em Growth in graphs.}\label{it:agar} Let $\Gamma$ be a graph.
How many vertices can be reached from a given vertex in a given
number of steps?
\item {\em Growth in infinite groups.}\label{it:woro}
 Let $A$ be a set of generators of
an infinite group $G$. Let $B(t)$ be the number of elements that can be
expressed as products of at most $t$ elements of $A$. How does $B(t)$
grow as $t\to \infty$?
\item {\em Random walks in groups.}\label{it:adar} Let $A$ be a set of generators of
a finite group $G$. Start with $x=1$, and, at each step, multiply $x$ by a 
random element of $A$. After how many steps is $x$ close to being 
equidistributed in $G$?
\item {\em More on growth in graphs: the spectral gap.}\label{it:ahar}
 Let $\Gamma$ be a graph. Consider
its adjacency matrix. What lower bounds can one give for the difference
between its two largest eigenvalues?
\item {\em Growth in arithmetic combinatorics.}\label{it:ortol} Let $G$ be an abelian group.
Let $A\subset G$. How large is $A+A=\{x+y:x,y\in A\}$ compared to $A$, and why? In general,
let $G$ be a group. Let $A\subset G$. How large\footnote{In the non-abelian case, there are technical reasons why it makes more sense to consider 
$A\cdot A\cdot A=\{x\cdot y \cdot z : x,y,z\in A\}$ 
rather than $A\cdot A = \{x\cdot y : x,y\in A\}$. The product $A\cdot A$ could
be small ``by accident''.} is $A\cdot A\cdot A$
compared to $A$, and why?
\end{enumerate}

Question (\ref{it:ortol}) has been extensively studied in the abelian setting.
Some time ago, I started studying it for non-abelian groups, and proved
\cite{He} that every set of generators $A$ of $G = \SL_2(\mathbb{F}_p)$
grows: $|A\cdot A\cdot A| > |A|^{1+\delta}$, $\delta>0$, provided
that $|A|<|G|^{1-\epsilon}$, $\epsilon>0$. (Here $|S|$ is the number of elements
of a set $S$.) This answered question (\ref{it:agar}) (on growth in graphs)
immediately in the
case of the Cayley graph of $\SL_2(\mathbb{F}_p)$; the bounds obtained
were strong enough to constitute the first proved case of a standard
 conjecture (Babai's). 
Questions (\ref{it:adar}) and (\ref{it:ahar}) (on random walks and
spectral gaps)
are closely related to each other,
and somewhat more indirectly to (\ref{it:agar}) and (\ref{it:ortol});
the result in \cite{He} gave non-trivial bounds for
(\ref{it:adar}) and (\ref{it:ahar}). These bounds were greatly improved
by Bourgain and Gamburd (\cite{BG}), who showed how to 
use a technique of Sarnak and Xue's \cite{SX} to derive from the results in \cite{He}
bounds for (\ref{it:adar}) and (\ref{it:ahar}) that are qualitatively
optimal (sufficient to amount to an {\em expander graph
property} for all sets of generators $A$ of $G$ such that $(G,A)$ has
the {\em large girth} property).

\subsection{Main result}

It remained to be seen whether the result in \cite{He} on growth in
$\SL_2(\mathbb{F}_p)$ could be generalised to other groups. Much of the
work in \cite{He} was specific to $\SL_2(\mathbb{F}_p)$. In
\cite{BG2}, the result was generalised (in a suitably strong form) to
$\SU_2(\mathbb{C})$; there is also a recent generalisation by O. Dinai
\cite{Din} to $\SL_2(\mathbb{F}_q)$, as well as results \cite{Bo} on
$\SL_2(\mathbb{Z}/d\mathbb{Z})$. From the point of view of the Lie
algebra, all of these groups are very closely related to
$\SL_2(\mathbb{F}_p)$. Thus, the matter of the extent to which the methods
in \cite{He} were truly flexible remained open.

The point of the present paper is to prove growth for $\SL_3(\mathbb{Z}/p
\mathbb{Z})$. Part of the proof (\S \ref{sec:torcon})
is ultimately derived from that in \cite{He},
and is likely to be valid for all semisimple groups of Lie type; part of
the proof is essentially new. 

\begin{main}
Let $G = \SL_3$. Let $K = \mathbb{Z}/p\mathbb{Z}$, $p$ a prime.
Let $A\subset G(K)$ be a set of generators of $G(K)$.

Suppose $|A|< |G(K)|^{1-\epsilon}$, $\epsilon>0$. Then
\begin{equation}\label{eq:atata}
|A\cdot A\cdot A|\gg |A|^{1 + \delta},\end{equation}
where $\delta>0$ and the implied constant depend only on $\epsilon$. 
\end{main}
We could, as in \cite{He}, write {\em let $A$ be a subset of $G(K)$
not contained in a proper subgroup of $G(K)$} instead of
{\em let $A$ be a set of generators of $G(K)$}; the two statements are
equivalent.

The condition that $A$ generate $G(K)$ is easy to satisfy in
applications (see, e.g., \cite{BG}, where the analogous result
(\cite{He})
on $\SL_2(\mathbb{F}_p)$ was applied).

Quite separately, it can be argued that the condition that $A$
generate $G(K)$ is a natural one. If $A$ does not generate $G(K)$,
what we have is no longer a statement about $G(K)$, but, rather, a
statement
about the group $\langle A \rangle$ generated by $A$; the set $A$
cannot know
that elements outside $\langle A\rangle$ exist.

We will, nevertheless, study all subsets $A$ of $\SL_3(\mathbb{Z}/p\mathbb{Z})$,
whether they generate the group or not.
\begin{thm}\label{thm:qartay}
Let $G = \SL_3$. Let $K = \mathbb{Z}/p\mathbb{Z}$, $p$ a prime.
Let $A\subset G(K)$.

Then, for every $\epsilon>0$, either
\begin{equation}\label{eq:caspond}
|A \cdot A\cdot A|\gg |A|^{1+\delta},\end{equation} where $\delta>0$ and
the implied constant depend only on $\epsilon$, or 
there are subgroups $H_1\triangleleft H_2 \triangleleft \langle A\rangle$
such that 
\begin{enumerate}
\item $H_2/H_1$ is nilpotent,
\item $A_k$ contains $H_1$, where $k$ depends only on $\epsilon$, and
\item $A$ is contained in the union of $\leq |A|^{\epsilon}$ cosets of $H_2$.
\end{enumerate}
\end{thm}
It is tempting to guess that a statement of this sort should be true in general for
subsets $A$ of arbitrary groups $G$.  As pointed out by Pyber \cite{P},
the constants $\delta$ and $k$ would then have to depend on $n$, where $n$ is
the smallest integer such that $G$ is isomorphic to a subgroup of
$\SL_n(\mathbb{F}_{p^{\alpha}})$ for some prime power $p^{\alpha}$. 
(See the remarks in \S \ref{subs:gorno}.)



\subsection{Consequences}
\subsubsection{Diameters}
By a result of Gowers, Nikolov and Pyber\footnote{Gowers \cite{Gow} proved a statement 
from
which (\ref{eq:muttan}) quickly follows, as was pointed out by
Nikolov and Pyber; see \cite{NP}. The results in \cite{Gow} and \cite{NP}
are of a general nature; with the aid of standard lower bounds on the
dimensions of complex representations of $\SL_n$, the special cases
$\SL_2$ and $\PSL_n$ were worked out in \cite{Gow} and \cite{NP},
respectively. More general statements can be found in \cite{BNP}.
A weaker version of (\ref{eq:muttan}) for $n=2$ was proven in
\cite[Key proposition, part (b)]{He}.}
\cite[Cor.\ 1 and Prop.\ 2]{NP},
\begin{equation}\label{eq:muttan}
A\cdot A\cdot A = \SL_n(K)
\end{equation}
for $A\subset G$, $|A|>2 |G|^{1-\frac{1}{3 (n+1)}}$, where $G = \SL_n(K)$
and $K = \mathbb{Z}/p\mathbb{Z}$. 

Together with (\ref{eq:muttan}), the main theorem implies results on
diameters.
The {\em diameter} of a graph $\Gamma$ is
\[\max_{v_1,v_2\in V} (\text{shortest distance between $v_1$ and $v_2$}),\]
where $V$ is the vertex set of $\Gamma$. We are especially interested
in the diameters of {\em Cayley graphs}. The {\em Cayley graph}
$\Gamma(G,A)$
of a pair $(G,A)$ (where $G$ is a group and $A\subset G$)
is defined to be the graph that has $G$ as its set of vertices and
$\{(g,a g): g\in G, a\in A\}$ as its set of edges. It is easy to see
that the diameter $\diam(\Gamma(G,A))$ of a Cayley graph $\Gamma(G,A)$
is the least integer $k$ such that
\[G = \{I\} \cup A \cup (A\cdot A) \cup \dotsb \cup 
(\mathop{\underbrace{A \cdot A \dotsb A}}_{\text{$k$ times}}).\]
If $A$ is a set of generators of $G$, then, by definition, every element
of $G$ can be expressed as a product of elements of $A \cup A^{-1}$; when
$G$ is finite, this implies that every element of $G$ can be
expressed as a product of elements of $A$, i.e., the diameter 
$\diam(\Gamma(G,A))$ of the Cayley graph $\Gamma(G,A)$ is finite. 
The question remains: how large can the diameter $\diam(\Gamma(G,A))$ be in
terms of $G$ and $A$?

The following statement is known as {\em Babai's conjecture}.
\begin{conj}[\cite{BS}] For every non-abelian finite 
simple group $G$ and any set of generators $A$ of $G$,
\begin{equation}\label{eq:udo}
\diam(\Gamma(G,A)) \ll (\log |G|)^c,\end{equation}
where $c$ is some absolute constant and $|G|$ is the number of elements of $G$.
\end{conj}

Until recently, there was no infinite family of groups $G$ for which the
conjecture was known for all $A$. In \cite{He}, I proved Babai's conjecture
for $G = \SL_2(\mathbb{Z}/p\mathbb{Z})$ and all $A$. As we shall see in
\S \ref{subs:broker}, the conjecture for $G=\SL_3(\mathbb{Z}/p\mathbb{Z})$
follows easily from the main theorem and (\ref{eq:muttan}).


\begin{cor}[to the main theorem and (\ref{eq:muttan})]\label{cor:gorot}
Let $p$ be a prime. Let $G = \SL_3(\mathbb{Z}/p\mathbb{Z})$. Let $A$ be
a set of generators of $G$. Then
\begin{equation}\label{eq:selfex}
\diam(\Gamma(G,A)) \ll (\log |G|)^c,
\end{equation}
where $c$ and the implied constant are absolute.
\end{cor}

It is clear that the corollary, as stated,
 implies that (\ref{eq:selfex}) holds for
$G = \PSL_3(\mathbb{Z}/p\mathbb{Z})$ as well. (I bother to say this because
$\PSL_3(\mathbb{Z}/p\mathbb{Z})$ is always simple, while
$\SL_3(\mathbb{Z}/p\mathbb{Z})$ is not simple for some $p$.)

If $A \subset G =  \SL_3(\mathbb{Z}/p\mathbb{Z})$ is such that
$\Gamma(G,A)$ has {\em girth} $\gg \log |G|$ (i.e., if it has no
non-trivial cycles of length less than a constant times $\log |G|$), it
is easy to see that the main theorem implies that the
diameter of $\Gamma(G,A)$ is in fact $\ll \log |G|$, not simply
$\ll (\log |G|)^c$ (see \S \ref{sec:genfin}). As we are about to discuss, it
is likely that even stronger statements can be made in this situation. 
 
\subsubsection{Spectral gaps and expander graphs}
Soon after \cite{He}, Bourgain and Gamburd (\cite{BG}) showed that, for 
$G = \SL_2(\mathbb{Z}/p\mathbb{Z})$ and $A$ any set of generators such that
the girth of $\Gamma(G,A)$ is $\gg \log |G|$, 
the adjacency matrix of the Cayley graph $\Gamma(G,A)$
has a {\em spectral gap} of size $\epsilon>0$, i.e., the difference between its
largest and second largest eigenvalues is bounded below by a constant.
(This implies that the endpoint of a random walk on $\Gamma(G,A)$ 
of length $C\cdot \log |G|$, $C$ large, is close to being equidistributed.)

The starting point was the Key
Proposition in \cite{He}, viz., the statement
$|A\cdot A\cdot A|\geq |A|^{1+\epsilon}$ for 
$A\subset \SL_2(\mathbb{Z}/p\mathbb{Z})$; Bourgain
and Gamburd succeeded in extracting a spectral gap $\epsilon>0$
therefrom thanks to their use of a technique of Sarnak and Xue \cite{SX}.
(In \cite{SX}, as in the work of Gowers et al., the main ingredient is
the fact that $\SL_2(\mathbb{Z}/p\mathbb{Z})$ (or
$\SL_n(\mathbb{F}_q)$, for that matter) has no small-dimensional
complex representations.) 

It is very likely that it will be possible to adapt
Bourgain and Gamburd's procedure so as to prove a
spectral gap $\lambda_1 - \lambda_2>\epsilon$, $\epsilon>0$
for $(\SL_3(\mathbb{Z}/p\mathbb{Z}),A)$ with large girth
 starting from the main theorem in
the present paper. However, this is not immediate: what is needed,
other than a straightforward translation of \cite{BG} into $\SL_3$,
is a bound ruling out the possibility that the random walks on a Cayley
graph of $\SL_3(\mathbb{Z}/p\mathbb{Z})$ with large girth be highly
concentrated on a subgroup early on. 

\subsection{Outline} Some basic background information
will be given in \S
\ref{sec:torul}. Sections \ref{sec:grosp} and 
 \ref{sec:orwise} will be devoted to preparatory results in arithmetic
combinatorics and growth in algebraic groups, respectively. The behaviour
of a (hypothetical) non-growing set $A$ in relation to maximal tori 
will be treated in \S \ref{sec:torcon};  we will also examine
the number of conjugacy classes occupied by such a set.
The main result will finally be proven -- for most of the possible range
of $|A|$ -- in \S \ref{sec:armon}. Part of the range will be
treated in \S \ref{sec:grome}; its treatment will
necessitate some detailed work involving the subgroup structure of $\SL_3$
(\S \ref{sec:pogor}). 

Section \ref{sec:orwise} treats algebraic groups in
general.
Most of the work in \S \ref{sec:torcon} will be done for $\SL_n$.
Sections \ref{sec:armon} to \ref{sec:grome} are in part specific to
$\SL_3$, though many of the results in them are stated and proved in greater
generality.

\subsubsection{Plan of proof}
Let $G = \SL_3$, $K = \mathbb{Z}/p\mathbb{Z}$. Suppose there is
a subset $A\subset G(K)$ violating the main theorem, i.e., 
a set $A$ such that (a) $A$ is substantially smaller than $G$ 
($|A|<|G|^{1-\epsilon}$, $\epsilon>0$) and (b) $A$ fails to grow
($|A\cdot A\cdot A| \ll |A|^{1+\delta}$, $\delta$ positive and very small). Then, as we
shall show in \S \ref{sec:torcon}, the set $A$ must be in some sense
very regular. For example, the number of conjugacy classes $\Cl_G(g)$
occupied by elements $g$ of $A$ will have to be almost precisely what
one would expect out of dimensional reasons.

Perhaps more surprisingly, $A$ will have to have a large intersection
with some maximal torus $T$; in other words, $A$ has many simultaneously
diagonalisable elements. Our aim will be to use $A$ to construct (\S
\ref{sec:armon}) a set of tuples of elements of $\mathbb{Z}/p\mathbb{Z}
\times \mathbb{Z}/p\mathbb{Z}$ satisfying too many linear relations
too often. This will stand in contradiction to a bound on linearity
(Cor.\ \ref{cor:espada}) that follows from a sum-product theorem
(\S \ref{subs:sumpro}--\ref{subs:schw}).

The above argument has a blind spot ($p^{4-\epsilon} < |A| < p^{4+\epsilon}$)
resulting from the fact that sum-product theorems for 
$\mathbb{Z}/p\mathbb{Z} \times \mathbb{Z}/p\mathbb{Z}$ do have exceptions --
all of size about $p$. For sets $A$ of size in the blind spot, it
becomes necessary to pass to a maximal parabolic subgroup and then use
the fact that we already know that the main theorem holds for $\SL_2$. If
the intersection $A^{-1} A \cap M$ with a maximal parabolic subgroup $M$
fails to generate a quotient of $M$ isomorphic to 
$\SL_2(\mathbb{Z}/p\mathbb{Z})$, then $A^{-1} A\cap M$ must (in essence)
lie in a Borel subgroup. We will see how sets grow in Borel subgroups by
means of a general result (Prop.\ \ref{prop:guggen}) of which the sum-product
theorem is but a shadow (Lem.\ \ref{lem:sumprod}).
\subsubsection{Tools}
The tools used are elementary in nature -- in contrast to the analytical tools
sometimes used to study arithmetic groups.

The reader may wonder why the main theorem is a statement on $A\cdot A\cdot
A$,
as opposed to one on $A\cdot A$ or on the product of $A$ with itself 
ten times. The statement $|A\cdot A|> |A|^{1+\delta}$ is not always true:
let $A = H \cup \{g\}$, where $H$ is a (non-normal) subgroup of $G$ and $g\notin H$,
for example. As for a statement on ten or twenty copies of $A$: we shall, 
in fact, be proving such a statement; a result essentially due to Ruzsa
(Lemma \ref{lem:furcht})
then tells us that, if $|A \cdot A^{-1} \cdot A \cdot A \cdot A \cdot A|> 
|A|^{1+\delta}$ (say), then
$|A \cdot A \cdot A| > |A|^{1+\delta'}$ (with $\delta'>0$ depending
only on $\delta>0$). 

Additive combinatorics appears again in the guise of the 
Balog-Szem\'eredi-Gowers theorem. This is a very useful result, if somewhat
rigid in its requirements; 
Bourgain showed in \cite{BG2} how to remove it from the proof in
\cite{He}, and it is likely that it will have to be replaced in the proof
given here as well when the proof is generalised to $\SL_n$, $n>3$. 

There is a rich literature on growth in infinite groups, based on the works of 
Gromov, Tits et al. There seems to be now at least one point of intersection
with it: the escape argument of \cite{EMO} will be used time and again in the
course of this paper. In essence, it tells us that we may avoid any
non-generic situation, such as, for example, that of matrices with repeated 
eigenvalues.

A truly crucial role is played by a sum-product theorem (first proven
over finite fields by Bourgain, Katz and Tao \cite{BKT} and Konyagin 
\cite{Ko}). The result we need will be derived here from a more
general statement (Prop.\ \ref{prop:guggen}) on growth in groups under 
commuting actions without fixed points.

\subsection{Acknowledgements}
Starting on September 2007, I was supported by the EPSRC grant 
EP-E054919/1. My stay at the Institute for Advanced Study (Princeton)
was supported by funds from the NSF grant DMS-0635607.
Thanks are also due to the Tata Institute (Mumbai), the Institute for
Mathematical Sciences (Chennai),  Universit\'e Paris-Sud 11
(Orsay), \'Ecole Polytechnique (Paris), the R\'enyi institute (Budapest) and Universidad de la Habana, for
their hospitality and their support during my visits.

Nick Gill's assistance was invaluable;
he is responsible for several careful readings and many helpful comments. 
Emmanuel Breuillard answered several of my questions, starting well before
anything was written down. Thanks are
also due to J.\ Bourgain, Y.\ Benoist, K.\ Buzzard, B.\ Conrad, 
O.\ Dinai, T.\ Ekedahl, G.\ Harcos, R.\ Hill, V.\ Meldrew, J.\ Pila, A.\ Silberstein, A.\ Skorobogatov, T.\ Szamuely and T.\ Wooley,
 for their help, and 
to the entire {\em groupe de travail} at the \'Ecole Polytechnique and
Chevaleret (Paris VI/VII), for hearing me out.

\section{Notation and preliminaries}\label{sec:torul}
\subsection{General notation}
As is customary, we denote by $\mathbb{F}_{p^{\alpha}}$ the finite field
of order $p^{\alpha}$.
Given a set $A$, we write $|A|$ for its number of elements. By $A+B$
(resp. $A\cdot B$), we shall always mean
$\{x+y : x\in A, y\in B\}$ (resp. $\{x\cdot y: x\in A, y\in B\}$),
By $A+\xi$ and $\xi\cdot A$ we
mean $\{x+\xi : x\in A\}$ and $\{\xi \cdot x: x\in A\}$, respectively.

Given a positive integer $r$ and a
 subset $A$ of a group $G$, we define
$A_r$ to be the set of all products of at most $r$ elements of $A \cup A^{-1}$:
\begin{equation}\label{eq:defsubl}
A_r = \{g_1 \cdot g_2 \dotsb g_r : g_i\in A \cup A^{-1} \cup \{1 \}\} .
\end{equation}

For us, $A^r$ means
$\{x^r : x\in A\}$; in general, if $f$ is a function on $A$,
we take $f(A)$ to mean $\{f(x) : x\in A\}$.
If $\Upsilon$ is a set of maps from $X$ to $Z$, and
$A$ and $Y$ are subsets of $X$ and $\Upsilon$, respectively, then
\[Y(A) = \{y(a) : y\in Y,\; a\in A\}.\]
We write $Y(a)$ for $Y(\{a\})$ and $y(A)$ for $\{y\} (A) =
\{y(a) : a\in A\}$.
\subsection{Boundedness}\label{subs:goroto}

We say ``$a\ll b$, where the implied constant is absolute'' or
``$a=O(b)$, where the implied constant is absolute'' 
when we mean that the non-negative real number $a$ (or the absolute value of the arbitrary
real number $a$) is at most
the real number $b$ multiplied by an absolute constant.
We write
$a \ll_{c_1,c_2,\dotsc, c_n} b$ or
$a = O_{c_1,c_2,\dotsc, c_n}(b)$ when we mean that
the non-negative real number $a$ (or the absolute value of the arbitrary
real number $a$) is at most the real number $b$ multiplied by a constant
depending only on $c_1,c_2,\dotsc,c_n$. We write 
$a \gg_{c_1,c_2,\dotsc, c_n} b$ to mean that $a$ is larger than
a positive constant depending only on $c_1,c_2,\dotsc,c_n$.

 In particular,
$a \ll_{c_1,c_2,\dotsc, c_n} 1$ (or $a = O_{c_1,c_2,\dotsc, c_n}(1)$) will
mean that $a$ is bounded in terms of $c_1,c_2,\dotsc, c_n$ alone.
We will use this latter notation even when $a$ is not a real number,
provided that we have defined what it means for $a$
to be bounded (in terms of other variables).

For example, when we say that a vector
\[\vec{d} = (d_0,d_1,d_2,\dotsc, d_n, 0,0,\dotsc)\;\;\;\;\;\;\;
\text{($d_i$ non-negative)}\]
is bounded in terms of a quantity $\ell$ alone, we mean that
both $n$ and $d_0,d_1,\dotsc, d_n$ are bounded in terms of $\ell$ alone.
We can then write this as follows: $\vec{d} \ll_{\ell} 1$. The quantity $\ell$
may
itself be a vector: we may write, for example, $\vec{d} \ll_{\vec{d}'} 1$ 
-- meaning that $n$ and $d_0,d_1,\dotsc, d_n$ are bounded in terms of
a vector $\vec{d}'$ alone -- or, for that matter, $a \ll_{\vec{d}} 1$ -- meaning that
a number $a$ is bounded in terms of $\vec{d}$ alone.

\subsection{Arithmetic combinatorics}

We start with a very simple and standard lemma.
\begin{lem}\label{lem:rastropor}
Let $G$ be a finite group. Let $A\subset G$. Suppose $|A|>\frac{1}{2} |G|$.
Then $A\cdot A = G$.
\end{lem}
\begin{proof}
Suppose there is a $g\in G$ not in $A\cdot A$. Then, for every $x\in G$,
either $x$ or $g x^{-1}$ is not in $A$. As $x$ goes over all elements
of $G$, we see that no more than one out of every two elements of $G$
can lie in $A$. In other words, $|A|\leq \frac{1}{2} |G|$. Contradiction.
\end{proof}

The following result
 is based on ideas of Ruzsa's, and, in particular, on his
triangle inequality (\cite[Lem.\ 2.1]{He}).
\begin{lem}[Tripling lemma]\label{lem:furcht}
Let $k>2$ be an integer. Let $A$ be a finite subset of a group $G$.
Suppose that
\[|A_k| \geq C |A|.\]
for some $C\geq 1$. Then
\[|A\cdot A \cdot A| \geq C^\delta |A|\]
where $\delta>0$ depends only on $k$.
\end{lem}
The dependence of $\delta$ on $k$ is, in fact, inverse linear ($1/\delta= O(k)$).
\begin{proof}
See \cite[Lem.\ 3.4]{T} or \cite[Lem.\ 2.2]{He}.
\end{proof}
In the present
paper, we shall almost always use the tripling lemma in the following form:
if $|A_k| \geq c |A|^{1 + \epsilon}$ with $c,\epsilon>0$, then
$|A\cdot A\cdot A| \gg_{c,\epsilon,k} |A|^{1 + \epsilon'}$, where $\epsilon'>0$
depends only on $c$, $\epsilon$ and $k$. This is simply a special case of
the lemma: set $C = c |A|^{\epsilon}$. (The proof in \cite[Lem.\ 2.2]{He}
is stated for this special case, but works in general.)

The Balog-Szemer\'edi-Gowers theorem is known in several different forms.
We derive the one we need from one of the most common formulations.
We make no effort to optimise the constants involved.
\begin{prop}[Balog-Szemer\'edi-Gowers]\label{prop:bsg}
Let $A_1$, $A_2$,\dots, $A_n$ be finite subsets of an abelian group $Z$. Let
$m = \min_j |A_j|$ and $M = \max_j |A_j|$. Let
$S\subset A_1\times A_2 \times \dotsb \times A_n$ be such that
\begin{equation}\label{eq:coron}|S|\geq c M^n \;\;\;\;\text{and}\;\;\;\;
 \left|
\left\{\sum_{1\leq j\leq n} a_j : (a_1,a_2,\dotsc,a_n)\in S\right\}\right| \leq
\frac{1}{c} m\end{equation}
for some constant $c\in (0,1)$.

Then there is a subset $A' \subset A_1$ such that
\[|A'| \gg c |A|\;\;\;\;\text{and}\;\;\;\;|A' + A'|\ll \frac{1}{c^C} |A'|,\]
where $C>0$ and the implied constants are absolute.
\end{prop}
Note that condition (\ref{eq:coron}) can hold only if
$\min |A_j| \geq c \max |A_j|$. 
\begin{proof}
Choose the tuple $(a_3,a_4,\dotsc,a_n)\in A_3\times A_4\times \dotsb \times
A_n$ such that the number of elements of the set
\[G_{a_3,a_4,\dotsc,a_n} = \{(a_1,a_2)\in A_1\times A_2:
(a_1,a_2,a_3,\dotsc,a_n)\in S\}\]
is maximal. We apply the Balog-Szemer\'edi-Gowers theorem as given in
\cite[Thm.\ 2.29]{TV} with $A = A_1$, $B=A_2$ and $G =
G_{a_3,a_4,\dotsc,a_n}$, and obtain that there are sets $A'\subset A_1$,
 $B'\subset A_2$ with $|A'| \gg c |A_1|$, $|B'|\gg c |A_2|$ and
\[|A' + B'| \ll c^{-7} |A_1|^{1/2} |A_2|^{1/2} \ll c^{-8} |A'|^{1/2}
|B'|^{1/2} \ll c^{-9} |A'|.\]
We apply the Pl\"unnecke-Ruzsa estimates \cite[Cor.\ 6.29]{TV}
and obtain that $|A' + A'|\ll c^{-18} |A'|$.
\end{proof}
There are non-commutative versions of Balog-Szemer\'edi-Gowers
(see \cite{T}); we shall not need them, however.

\subsection{Groups and generation}
By $\langle g\rangle$ we mean the group generated by an
element $g$ of a group $G$. By $\langle A\rangle$ we mean the group
generated by a subset $A$ of a group $G$. By $H<G$ we mean that 
$H$ is a subgroup (proper or not) of the group $G$.

We write $\Cl_G(g)$ for the conjugacy class of an element $g\in G$
in $G$.

\subsection{Varieties} Let us speak concretely.
An (affine) {\em variety} $V$ is given by a finite set of polynomial
 equations $F(x_1,x_2,\dotsc,x_n)=0$
in $n$ variables with coefficients in a field. (We will usually work
in an affine space (denoted by $\mathbb{A}^n$), rather than in
 projective space $\mathbb{P}^n$; if we work in projective space,
our polynomials $F$ must all be homogeneous.)
If the coefficients all
lie in a field $K$, we say that $V$ is {\em defined over} $K$, or simply
write $V/K$. If $L$ is another field -- containing, contained in,
or equal to $K$ -- then we write $V(L)$ for the set of {\em $L$-valued points
of $V$}, i.e., the set of solutions in $L^n$ to our set of equations.

A {\em subvariety} $W\subset V$ is a variety that can be defined by
a set of equations that contains a set of equations defining $V$.
By a {\em proper subvariety} $W\subsetneq V$ we mean simply a subvariety with
$W\ne V$. (There is a very different algebraic-geometrical notion of
{\em properness}; we shall not use it.) Clearly, if two varieties $V$, $W$
defined over $K$ satisfy $W\subset V$, then $W(L)\subset V(L)$ for every
extension $L$ of $K$.

A {\em Zariski-open set} $\Sigma$ in a variety $V$ is 
the complement of a variety $W\subset V$; its set of points 
$\Sigma(L)$ is defined to be $V(L)\setminus W(L)$. A Zariski-open set is
not, in general, a variety.

All or nearly all the algebraic geometry we need can be found in
\cite{Da}, for instance.

\subsubsection{Algebraic groups}

If we speak of an (affine)
{\em algebraic group} defined over a field $K$, we mean
an affine variety $G/K$ with a group law such that the multiplication map
$\mu:G\times G\mapsto G$ and the inverse map $\iota:G\to G$ are regular
and defined over $K$. (Between affine varieties, a {\em regular map}
is simply a map given by polynomials.)
Thus, strictly speaking, an algebraic group $G$ is not a group; rather,
its set of points $G(L)$ will be a group for every field $L$ containing $K$.
The set of points $G(L)$ for $L$ contained in $K$ may also be a group, if
it is closed under the group operation.

The following are typical examples. We may speak of the algebraic group
$G = \SL_n$ (or, for that matter,
 $G = \SO_n$ or $G = \Sp_{2 n}$). This is a variety defined
over $\mathbb{Z}$, and thus over an arbitrary field: it is given by the
equation $\det(g) = 1$ in the $n^2$ variables
$g_{i j}$, $1\leq i,j\leq n$. (Note that the determinant is a polynomial.)
The multiplication map from $\SL_n\times \SL_n$ to $\SL_n$ is given by matrix
multiplication. For any field $K$, the set $G(K)$ is the set $\SL_n(K)$
of all $n$-by-$n$ matrices with entries in $K$ and determinant $1$; this
set is a group under the group law just given, i.e., matrix multiplication.
A {\em maximal torus} $T$ in $G=\SL_n$ is a group consisting of all
diagonal matrices for some choice of basis, i.e., a group that can be
made into the group of diagonal matrices by conjugation. If $T$ can
be thus diagonalised by conjugation by a matrix in $G(K)$, then $T$ is
defined over $K$; otherwise, $T$ is defined over $\overline{K}$ but not
over $K$. Even in the latter case, we may still speak of the group $T(K)$.
For example, if $K = \mathbb{R}$, and we consider the matrices
\[\left(\begin{matrix}\cos(\theta) & \sin(\theta)\\
-\sin(\theta) & \cos(\theta)\end{matrix}\right),\;\;\; \theta\in \mathbb{R},\]
we can see that they are the points over $\mathbb{R}$ of a maximal torus $T$,
in that they can all be diagonalised simultaneously; this torus $T$ is
defined over $\mathbb{C}$, but cannot be defined over $\mathbb{R}$.

In general, algebraic groups behave a great deal like Lie groups, even over
finite fields; in particular, they have maximal tori, roots, etc.
Every (affine) algebraic group is a closed algebraic subgroup of
$\GL_n$ for some $n\geq 1$ (\cite[\S 8.6]{Hum}).
For an introduction to algebraic groups, see \cite{Bor} or \cite{Hum}.

\subsubsection{Degree and dimension}\label{subs:convu}

The {\em dimension} $\dim(X)$ of an irreducible variety $X$ is the length $k$ of the
longest chain $\{x\} = X_0 \subset X_1 \subset \dotsb \subset X_k = X$
of irreducible subvarieties of $X$; this corresponds to the intuitive notion
of dimension. If the irreducible components of a variety $V$ all have
the same dimension, we say $V$ is {\em pure dimensional}, and define
the dimension $\dim(V)$ of $V$ to be that of any of its irreducible components.
(An {\em irreducible component} of a variety $V$ is an irreducible subvariety
of $V$ not contained in any other irreducible subvariety of $V$.)

The {\em degree} $\deg(V)$ of a pure-dimensional variety $V$ of dimension $r$
in $n$-dimensional
affine or projective space is its number of intersection points
with a generic linear variety of dimension $n-r$. (Thus, for example,
the degree of an irreducible plane curve is its number of intersection
points with a generic line.)

Let us first see what the degree of a variety has to do with the familiar
notion of the degree of a polynomial. Let $F$ be an irreducible
 polynomial in $n$ variables with coefficients in a field $K$.
Then the equation
\[F(x_1,x_2,\dotsc,x_n) = 0\]
defines an irreducible variety $V$ of codimension $1$, i.e., of 
dimension $n-1$ in
$n$-dimensional affine space $\mathbb{A}^n$. (The irreducibility of $V$
turns out to be an easy consequence of the irreducibility of $F$ and the fact
that $K\lbrack x_1,x_2,\dotsc, x_n\rbrack$ is a 
{\em unique factorisation domain}.)

 Now, it is not hard to see
that the degree of $V$ will be equal to the degree of $F$: if
we let $x_1 = a_1 + b_1 t, x_2 = a_2 + b_2 t,\dotsc , x_n = a_n  + b_n t$
 for some constants
$a_1, a_2, \dotsc , a_n \in \overline{K}$, 
$b_1, b_2,\dotsc , b_n \in \overline{K}^*$,
the equation $F(a_1 + b_1 t, a_2 + b_2 t,\dotsc, a_n + b_n t) = 0$ will be an equation
on $t$ of degree at most $\deg(F)$, and, for $a_1,a_2,\dotsc,a_n$,
$b_1,b_2,\dotsc, b_n$ sufficiently
``generic'', of degree exactly $\deg(F)$. That equation on $t$ will hence
have $\deg(F)$ roots (all distinct for $a_i$, $b_i$ sufficiently
generic). In other words, $V$ and the line given by
$x_1 = a_1 + b_1 t$, $x_2 = a_2 + b_2 t$, \dots , $x_n = a_n + b_n t$ 
have $\deg(F)$
intersection points. (It should be clear now that we mean intersection
points whose coordinates lie in the algebraic closure $\overline{K}$, and
not necessarily in $K$.) We have thus sketched how to show
 that $\deg(V) = \deg(F)$.

All of the above can be made precise, in that all the statements above
remain true when ``generic'' is given what we shall see as
 its precise meaning: namely, ``outside
a variety of positive codimension''. Thus, for example, the degree
of $F(a_1 + b_1 t, a_2 + b_2 t,\dotsc, a_n + b_n t) = 0$ 
is exactly $\deg(F)$ provided
that $(a_1,a_2,\dotsc,a_n,b_1,b_2,\dotsc,b_n)$ 
lies outside a variety of codimension $1$
in $\mathbb{A}^{2 n}$, viz., the variety given by
the equation \[\text{leading coefficient}
= 0.\] Similarly,
the roots $t_1, t_2, \dotsc$
of $f_{\vec{a},\vec{b}}(t) = 
F(a_1 + b_1 t, a_2 + b_2 t,\dotsc, a_n + b_n t) = 0$ are all distinct if
$(a_1,a_2,\dotsc,a_n,b_1,b_2,\dotsc,b_n)$ lies outside the variety given by
the equation \[\text{discriminant}(f_{\vec{a},\vec{b}}) = 0,\]
or, alternatively, if the line given by $(a_1+b_1 t,\dotsc,a_n + b_n t)$
is not tangent to the surface $F(x_1,x_2,\dotsc,x_n)=0$ at any point.
(Showing that the discriminant is not identically $0$ may not be immediately
obvious.)

The degree of a variety
is a yardstick of complexity that behaves well under intersections.
We shall need the following general version of Bezout's
theorem.
\begin{lem}[Bezout's theorem, generalised]\label{lem:bezout}
Let $X_1,X_2,\dotsc,X_k$ be pure-dimensional
varieties in $\mathbb{P}^n$, and let
$Z_1,Z_2,\dotsc,Z_l$ be the
irreducible components of the intersection $X_1\cap X_2\cap \dotsb \cap
X_k$. Then
\[\sum_{j=1}^l \deg(Z_j) \leq \prod_{i=1}^k \deg(X_i),
\]
where $X_{1},X_{2},\dotsc,X_{k}$ are the irreducible components of
$X_j$.
\end{lem}
As is stated in \cite{Da},
the form of the statement goes back to Fulton and MacPherson.
\begin{proof}
See \cite{Da}, p.\ 251.
\end{proof}

It remains to see how to define the dimension and the degree of a variety $V$
when $V$ is not irreducible. We simply define $\dim V$ to be the dimension
of the irreducible subvariety of $V$ of largest dimension. As for the
degree, it will be best to see it as a vector: we define the {\em degree}
$\vdeg(V)$ of an arbitrary variety $V$ to be
\[(d_0,d_1,\dotsc,d_k,0,0,0,\dotsc),\]
where $k = \dim(V)$ and $d_j$ is the degree of the union of the irreducible
components of $V$ of dimension $j$.

It is easy to see that
Bezout's theorem implies that, for any varieties $V_1,V_2,\dotsc,V_k$,
the degree $\vdeg(W)$ of the intersection $W = V_1 \cap V_2 \cap \dotsb
\cap V_k$ is bounded in terms of $\vdeg(V_1)$, $\vdeg(V_2)$,\dots ,
$\vdeg(V_k)$ alone. (See \S \ref{subs:goroto} for an explanation of what
we mean by $\vdeg(W)$ being bounded in terms of such and such; we mean
that both $\dim(W)$ and the degree $d_j$ of the union of the
irreducible components of $V$ of dimension $j$ are bounded in terms of such
and such.)
 We can also see easily (without the use of
Bezout's theorem) that the degree of the variety
$V = V_1 \cup V_2 \cup \dotsb \cup V_k$ is bounded in terms of $\vdeg(V_1)$,
$\vdeg(V_2)$, \dots, $\vdeg(V_k)$; in fact, if the $V_i$'s have no
components in common, we will have $\vdeg(V) = \sum_i \vdeg(V_i)$.

A very concrete consequence of what we have said so far is the following:
if a variety $V$ is defined by equations
\[\begin{aligned}
F_1(x_1,x_2,\dotsc,x_n) &= 0\\
F_2(x_1,x_2,\dotsc,x_n) &= 0\\
\dotsc \dotsc\\
F_k(x_1,x_2,\dotsc,x_n) & = 0,\end{aligned}\]
then its degree $\vdeg(V)$ is bounded in terms of
$n$ and $\deg(F_1)$, $\deg(F_2)$, \dots, $\deg(F_n)$ alone.

If a regular map $\phi:V\mapsto W$ between two varieties 
$V\subset \mathbb{A}^m$, $W\subset \mathbb{A}^n$
is defined by polynomials $\phi_1, \phi_2,\dotsc , \phi_n$ on the variables
$x_1, x_2,\dotsc , x_m$, we define $\deg_{\pol}(\phi)$ to be 
$\max_j \deg(\phi_j)$.
(If several representations of $\phi$ by polynomials $\phi_1, \phi_2, \dotsc,
\phi_n$ are possible, we choose -- for the purposes of defining $\deg_{\pol}$ --
the one that gives us the least value of $\deg_{\pol}$.) What we have just seen
amounts to stating that, if a subvariety $V'$ of $V$ is given by
$\phi(x) = y$ for some $y\in W(\overline{K})$, then 
$\vdeg(V')$ can be bounded in terms of $\deg_{\pol}(\phi)$ and $n$
(where $W\subset \mathbb{A}^n$).

\subsubsection{Fibres and counting}\label{subs:fibcou}

Let $V$ be a subvariety of $X\times Y$, where $X$ and $Y$ are varieties.
The {\em fibre} $V_{x=x_0}$ (or $V_{y=y_0}$) is the subvariety of $Y$
(or $X$) consisting of the points $y$ such that $(x_0,y)$ lies on $V$
(or of the points $x$ such that $(x,y_0)$ lies on $V$). It is
an immediate consequence of Bezout's theorem that $\vdeg(V_{x=x_0})$
and $\vdeg(V_{y=y_0})$ are bounded in terms of $\vdeg(V)$ alone.

Let $V$ be a proper subvariety of $X\times Y$,
where $X$ and $Y$ are varieties.
Then there is a proper
subvariety $W$ of $X$ such that, for every $x_0$ lying on $X\setminus W$,
the fibre $V_{x=x_0}$ is a proper subvariety of $Y$; moreover, 
$\vdeg(W)$ is bounded in terms of $\vdeg(V)$ alone. This is easy to show:
since $V$ is a proper subvariety of $X\times Y$, there is a point
$(x_0,y_0)$ of $X\times Y$ not on $V$; then the fibre $V_{y=y_0}$
is a proper subvariety of $X$, and, for every $x_0'$ lying on $X\setminus
V_{y=y_0}$, the fibre $V_{x=x_0'}$ does not contain the point $y_0$,
and hence is a proper subvariety of $Y$. Set, then, $W = V_{y=y_0}$.

By the same argument, there is also a proper subvariety $W'$ of $Y$ such
that, for every $y_0$ lying on $Y\setminus W'$, the fibre $V_{y=y_0}$
is a proper subvariety of $X$.

Let $K$ be a finite field.
Let $V/\overline{K}$ be a subvariety of $\mathbb{A}^n$
such that all of its irreducible components have dimension $\leq m$.
Then \begin{equation}\label{eq:otoronco}
|V(K)| \ll_{\vdeg(V),n} |K|^m.\end{equation}
This crude bound can be proven as follows.

We will proceed by induction on $n$.
We can assume without loss of generality that $V$ is irreducible of
dimension $m$. (The number of components of a variety $V$ is 
$\ll_{\vdeg(V)} 1$.)
See $\mathbb{A}^n$ as the product of affine varieties 
$\mathbb{A}^1\times \mathbb{A}^{n-1}$.
Suppose that there is a point $t\in \mathbb{A}^1$ such that the fibre 
$V_{x_1 = t}$
has components of dimension $m$. Then $V = \{t\}\times V_{x_1 = t}$, as
otherwise $V$ would have dimension $>m$ (by the definition of dimension).
We then obtain (\ref{eq:otoronco}) by the inductive assumption for $n-1$.

Suppose now that there is no point $t\in \mathbb{A}^1$ such that
the fibre $V_{x_1 = t}$ has components of dimension $m$. By
the inductive assumption, (\ref{eq:otoronco}) holds for $n-1$, and so,
in particular,
$|V_{x_1=t}(K)| \ll_{\vdeg(V_{x_1=t}),n-1} |K|^{m-1}$ for every $t$. Since there
are $|K|$ possible values of $t$, and since $\vdeg(V_{x_1=t}) \ll_1
\vdeg(V)$, we conclude that
\[|V(K)| \ll_{\vdeg(V),n} |K| \cdot |K|^{m-1} = |K|^m,
\]
as we wished to show.

Bounds much more precise than (\ref{eq:otoronco}) are known: take, for
instance, the Lang-Weil theorem \cite{LW}. (We shall not need the later
and very deep results of Deligne and others.)

\subsubsection{Abuse of language}
Given a variety $V$ defined over a field $K$, and a subvariety $W/\overline{K}$
defined over the algebraic completion $\overline{K}$ of $K$, we will write
$W(K)$ for $W(\overline{K})\cap V(K)$. (We will even speak of the points of $W$ over
$K$, meaning $W(K):= W(\overline{K})\cap V(K)$.) 

\subsubsection{Independence}\label{subs:indep}

Let $V_1, V_2,\dotsc , V_k$ be linear subspaces of an affine space 
$\mathbb{A}^n$; let them be defined over a field $K$. We say that
$V_1, V_2,\dotsc ,V_k$ are {\em linearly independent} if there is no
choice of points $v_1\in V_1(\overline{K})$, $v_2\in V_2(\overline{K})$,
\dots , $v_k\in V_k(\overline{K})$, not all of them $0$,
such that $v_1 + v_2 + \dotsb + v_k = 0$.

\section{Growth in rings and Borel subgroups}\label{sec:grosp}
\subsection{Growth under commuting actions}\label{subs:kaye}
Ever since the sum-product theorem was proven by Bourgain, Katz and Tao
(\cite{BKT}), it has been subject to a series of refinements and variations. 
Of these, one of the most interesting is a result of Glibichuk and
Konyagin (\cite{GK}, Lemma 3.2--Corollary 3.5), 
both because it applies to pairs of sets
of completely arbitrary sizes, and because of its rather simple proof.

It will become apparent that the natural setting of ``sum-product
theorems'' is a much broader one than the one in 
\cite{BKT}, \cite{GK} or the related
literature. It is not really a result about subsets of the field 
$\mathbb{Z}/p\mathbb{Z}$, but, rather, a result about groups (abelian or
non-abelian) and commuting automorphisms thereof. We shall show that
the sum-product theorem
over $\mathbb{Z}/p\mathbb{Z}$ (say) is a consequence of a special case
of the general result below.
Before that, we shall also see how this general result has useful
implications on the action of maximal tori in $\SL_n(K)$ on unipotent
subgroups. 

\begin{prop}\label{prop:guggen}
Let $G$ be a group and $\Upsilon$ an abelian group of automorphisms of $G$.
Let $Y\subset \Upsilon$ be a non-empty set such that
\begin{equation}\label{eq:conocon}
\text{if $y(g) = g$ for $y\in Y^{-1} Y$, $g\in G$, then either
$y = e$ or $g=e$.}
\end{equation}

 Then, for any non-empty
$A\subset G$ and any
 $Y_0 \subset
\Upsilon$, $A_0\subset G$,
either
\begin{equation}\label{eq:arma}|A \cdot Y(a_1)| \geq |A| \cdot |Y| \end{equation}
or
\begin{equation}\label{eq:erme}
|\{y_2(a) \cdot
y(y_2(a_0)) \cdot y(a_2^{-1} \cdot a_1) \cdot y(y_1(a_0^{-1})) \cdot y_1(a^{-1}) :
a\in A, y\in Y\}|\geq |A|\cdot |Y| .
\end{equation}
for some $a_0 \in A_0$,
 $a_1,a_2\in A$, $y_1,y_2\in Y$, or
\begin{equation}\label{eq:irmi}
|\{y_2(a) \cdot y_0(y(a_2^{-1} a_1)) \cdot y_1(a^{-1}) :a \in A, y\in Y\}|
\geq  |A|\cdot |Y| 
\end{equation}
for some $y_0 \in Y_0$, $a_1,a_2\in A$, $y_1,y_2 \in Y$, or
\begin{equation}\label{eq:ormo}
|\{y_2(a) \cdot y(a_2^{-1} a_1) \cdot y_1(a^{-1})
:a\in A, y\in Y\}|
> \frac{|A| |Y| |\mathscr{O}|}{|A| |Y| + |\mathscr{O}|} 
\geq \frac{1}{2} \min(|A| |Y|, |\mathscr{O}|),\end{equation}
where $a_1,a_2\in A$, $y_1,y_2\in Y$, and
$\mathscr{O}$ is the union of the orbits of the elements
of $A$ under the operations
$a\mapsto a_0 \cdot a$ (for all $a_0\in A_0$) and
$a\mapsto y_0(a)$ (for all $y_0\in Y_0$).
\end{prop}
It should be easy to see that the inequalities
(\ref{eq:arma})--(\ref{eq:irmi})
must all be equalities; we phrase them as inequalities simply because
we are interested in lower bounds on growth.

If we take $A_0=A$ and $Y_0 = Y \cup Y^{-1}$, 
Proposition \ref{prop:guggen} acquires a particularly simple
form:
\begin{cor}\label{cor:ogrodo}
For any group $G$ and any abelian group $\Upsilon$ of automorphisms of $G$.
Then, for any $A\subset G$ and any $Y\subset \Upsilon$ satisfying
(\ref{eq:conocon}),
\[|(Y_2(A))_6| > \frac{1}{2} \min(|A| |Y|, |
R|), 
\] 
where $R = \langle \langle Y\rangle (\langle A\rangle)\rangle$ is the set of all products of elements of
the form $y(a)$ with $a\in \langle A\rangle$ and 
$y\in \langle Y\rangle$.
\end{cor}
\begin{proof}[Proof of Corollary \ref{cor:ogrodo}]
Set $A_0 = A \cup A^{-1}$, $Y_0=Y\cup Y^{-1}$ and apply Proposition \ref{prop:guggen}. It
remains only to prove that the union $\mathscr{O}$
of the orbits of the elements of $A$ under the action of
$x\mapsto a\cdot x$ ($a\in A$) and $x\mapsto y(x)$ ($y\in Y$) is
equal to the set $R$ described in the statement. It is clear that
$\mathscr{O}\subset R$. 

To prove $R\subset \mathscr{O}$, we proceed
by induction:  let $R(n)$ be the set of all products of at most 
$n$ elements of the form $y(a)$, $a\in A \cup A^{-1}$, $y\in \langle Y\rangle$. 
Assume $R(n) \subset \mathscr{O}$. (This is certainly true for $n=0$, since
the identity element $e = a \cdot a^{-1}$ is in $\mathscr{O}$.)
We wish to prove $R(n+1) \subset \mathscr{O}$. Any $g\in R(n+1)$ can
be written in the form $y(a)\cdot h$, where $y\in \langle Y\rangle$
and $a\in A \cup A^{-1}$. Now $y(a) \cdot h = y(a \cdot y^{-1}(h))$. Because
$h\in R(n)$, and because $y$ is a homomorphism, $y^{-1}(h)$ is also in $R(n)$.
Since $R(n)\subset \mathscr{O}$, $y^{-1}(h)$ must be in $\mathscr{O}$. Then
$y(a \cdot y^{-1}(h))$ must also be in $\mathscr{O}$. Thus every element
of $R(n+1)$ is in $\mathscr{O}$.
\end{proof}

{\bf Examples.} Before we prove Proposition \ref{prop:guggen}, let us see two
of its consequences; we shall examine them in more detail later.
\begin{enumerate}
\item Let $G = \mathbb{F}_p$ (as an additive group),
$\Upsilon = \mathbb{F}_p^*$ (acting on $G$ by multiplication),
$A_0 = \{1\}$, $G_0 = e$. Then condition (\ref{eq:conocon}) is easily
seen to be satisfied: it just says that, in a field, if
$y\cdot g = g$, then either $y=1$ or $g=0$. (The same is true in
any ring without zero divisors.)
Thus we may apply Proposition \ref{prop:guggen}, and we obtain that,
 for any $A\subset \mathbb{F}_p$
and any $Y\subset \mathbb{F}_p^*$, 
\begin{equation}\label{eq:arta}|Y\cdot A + Y\cdot A - Y \cdot A - Y\cdot A +
  Y^2 - Y^2| > \frac{1}{2} \min(|A| |Y|, p) .\end{equation}
(This is the result of Glibichuk and Konyagin's mentioned before;
see \cite[\S 3]{GK}.)
We may set $Y = A$, and then a few applications of the Pl\"unnecke-Ruzsa estimates
 (\cite{TV}, Cor.\ 6.29) suffice to derive from (\ref{eq:arta}) the
conclusion that
\[|A\cdot A + A\cdot A| \geq |A| \cdot (\frac{1}{2} \min(|A|,p/|A|))^{1/6} 
\]
for every subset $A$ of $\mathbb{F}_p^*$.
An application of the Katz-Tao lemma (\cite[Lem.\ 2.53]{TV}; see also
\cite{B2}, \cite{Ga}) then suffices
to show that, for every $A\subset \mathbb{F}_p^*$ with 
$|A|<p^{1-\delta}$, $\delta>0$, we have either $|A+A|>|A|^{1+ \epsilon}$
or $|A\cdot A|>|A|^{1 + \epsilon}$, where $\epsilon>0$ depends only on
$\delta>0$. This is the well-known {\em sum-product theorem} of
Bourgain, Katz and Tao (\cite{BKT}), as extended by Konyagin.
We shall not use this theorem; 
instead, we shall use a sum-product theorem on the
ring $\mathbb{F}_p \times \mathbb{F}_p$, after proving it by proceeding
much as we just did.

\item\label{it:dordo}
 Let $G$ be the group of upper-triangular matrices in $\SL_n(K)$
with $1$'s on the diagonal. Let $\Upsilon$ be the group of diagonal matrices,
acting on $G$ by {\em conjugation} (not multiplication). Let $Y\subset
\Upsilon$ be a set of matrices such that the map $g\mapsto g_{ii} g_{jj}^{-1}$
(i.e., a {\em root of $\SL_n(K)$ relative to $\Upsilon$}, in the parlance
of groups of Lie type) is injective on $Y$ for all $1\leq i,j\leq n$ distinct.

Then (\ref{eq:conocon}) is satisfied, and
so, by Corollary \ref{cor:ogrodo},
\[|(Y_2(A))_6| \geq \frac{1}{2} \min(|A| |Y|, |R|),\]
where $R = \langle \langle Y\rangle (\langle A\rangle)\rangle$.
We shall look into this issue with more care in \S \ref{subs:bore};
see Proposition \ref{prop:avio}.
\end{enumerate}

We will now see the proof of Proposition \ref{prop:guggen}. It is quite
 close to
that of \cite[Lemma 1]{B}, whose proof is in turn based closely on the
argument in \cite[\S 3]{GK}. Our version is self-contained.
\begin{proof}[Proof of Proposition \ref{prop:guggen}]
The idea is to use a ``pivot'' $\xi$, meaning an element $\xi$ of $G$
such that the map $\phi_{\xi}$ from $A\times Y$ to $G$ given by
$(g,y)\mapsto (g\cdot y(\xi))$ is injective. If there is such a pivot,
the injectivity of $\phi_{\xi}$ gives us that $|A \cdot Y(\xi)|$ is large:
$|A\cdot Y(\xi)|\geq |A|\cdot |Y|$. Then one finishes by showing that
one can construct $\xi$ in a bounded number of steps starting from $A$ and
$Y$.
If there is no pivot, then the set of non-pivots must be rather large.
We use this fact itself to prove growth. 

Saying that $\phi_{\xi}$ is injective is the same as saying that
$\xi \notin \delta_{y_1,y_2}^{-1}(\{a_2^{-1} \cdot a_1\})$ 
for all $a_1,a_2\in A$ and all distinct $y_1,y_2\in Y$, where
$\delta_{y_1,y_2}:G\to G$ is the map $\gamma \mapsto y_2(\gamma) \cdot
(y_1(\gamma))^{-1}$. Now, if $\delta_{y_1,y_2}(\gamma_1) =
y_2(\gamma_1) \cdot (y_1(\gamma_1))^{-1}$ equals
$\delta_{y_1,y_2}(\gamma_2) =
y_2(\gamma_2) \cdot (y_1(\gamma_2))^{-1}$, then
$y_2(\gamma_1^{-1} \gamma_2) = y_1(\gamma_1^{-1} \gamma_2)$,
and so $y_1^{-1}(y_2(\gamma_1^{-1} \gamma_2)) = 
\gamma_1^{-1} \gamma_2$. Since $y_1,y_2 \in Y$ are distinct
and $\gamma_1, \gamma_2\in G$ are distinct, this contradicts
assumption (\ref{eq:conocon}). Hence $\delta_{y_1,y_2}:G\to G$
is injective for all pairs $(y_1,y_2)$ of distinct elements of $Y$.
This shall be crucial later.

We face two cases, depending on whether or not the set
\begin{equation}\label{eq:shos} S =
\mathop{\mathop{\bigcup_{a_1, a_2 \in A}}_{y_1, y_2 \in Y}}_{y_1\ne y_2}
\delta_{y_1,y_2}^{-1}(a_2^{-1} \cdot a_1)\end{equation} 
contains the orbit $\mathscr{O}$. The set $\mathscr{O}$ contains all
``easily constructible'' elements; if $\mathscr{O}$ is not contained in $S$,
we can construct an element not in $S$, i.e., a valid pivot.

{\em Case 1: $\mathscr{O}\not\subset S$.} (Read:
there is a pivot.)

The set $\mathscr{O}$ is the union of orbits of the elements of $A$ under
certain actions. Hence, if $\mathscr{O}\not\subset S$, we have that either
$A\not\subset S$ or there is an element $s$ of $S$ that is taken out of $S$
by one of the actions: that is, either $a_0\cdot s \notin S$ for
some $a_0\in A_0$ or
$y_0(s)\notin S$ for some $y_0 \in Y_0$. 
Call these three cases (a), (b) and (c). In case (a),
we let $\xi$ be any element of $A$ not in $S$; in case (b), we let
$\xi = a_0 \cdot s$; finally, in case (c), we let $\xi = y_0(s)$.

Now we are almost done. We have a map
\begin{equation}\label{eq:gogod}\phi_{\xi}:(g,y)\mapsto g \cdot y(\xi)
\end{equation}
from $A \times Y \to G$. Because $\xi \not\in S$, the map is injective.
The map has been constructed in a finite number of steps from
the elements of $A$ and $Y$,
since $\xi$ was defined that way. 

Let us work out the meaning and implications of this last statement case by
case.

{\em Case 1(a): $A\not \subset S$; $\xi$ an element of $A$ not in $S$.}
 Since $\phi_{\xi}$ is injective,
\[|A \cdot Y(\xi)| \geq |A| \cdot |Y|.\]
We have proven (\ref{eq:arma}).

{\em Case 1(b): $\xi = a_0 \cdot s$.}
 Since $\phi_{\xi}$ is injective,
\begin{equation}\label{eq:vampi}|A \cdot Y(\xi)| \geq |A| \cdot |Y|.
\end{equation}
Now we must do a little work: $\xi$ is defined in terms of $s$, and
the definition of $s$
involves the map $\delta_{y_1,y_2}^{-1}$, which we must now somehow remove.
Because $\delta_{y_1,y_2}$ is injective, (\ref{eq:vampi}) implies 
\begin{equation}\label{eq:ods}
|\delta_{y_1,y_2}(A\cdot Y(\xi))| \geq |A|\cdot |Y| .\end{equation}
Now, for
any $a\in A$, $y\in Y$,
\begin{equation}\label{eq:deliu}\begin{aligned}
\delta_{y_1,y_2}(a\cdot y(\xi)) &= y_2(a \cdot y(\xi)) \cdot
(y_1(a \cdot y(\xi)))^{-1}\\
&=  y_2(a) \cdot y_2(y(\xi)) \cdot (y_1(y(\xi)))^{-1} \cdot (y_1(a))^{-1}\\
&= y_2(a) \cdot y(y_2(\xi) (y_1(\xi))^{-1}) \cdot (y_1(a))^{-1} . 
\end{aligned}\end{equation}
(It is here that the fact that $\Upsilon$ is abelian is finally used.)
Recall that the definition of $\delta_{y_1,y_2}$ is
 $\delta_{y_1,y_2}(\xi) = y_2(\xi) (y_1(\xi))^{-1}$.

Because we are in case 1(b), there are $a_0\in A_0$, $s\in S$ such that
$\xi = a_0 \cdot s$. By the definition (\ref{eq:shos}) of $S$,
 there are $y_1,y_2 \in Y$ distinct and $a_1,a_2\in A$ such that
$\delta_{y_1,y_2}(s) = a_2^{-1} \cdot a_1$.
Then 
\begin{equation}\label{eq:clargo}\begin{aligned}
y(y_2(\xi) \cdot y_1(\xi)^{-1}) &= y(y_2(a_0) \cdot
y_2(s) \cdot (y_1(s))^{-1} \cdot (y_1(a_0))^{-1})\\
&= y(y_2(a_0)) \cdot  y(y_2(s) (y_1(s))^{-1}) \cdot y((y_1(a_0))^{-1})\\
&= y(y_2(a_0)) \cdot y(\delta_{y_1,y_2}(s)) \cdot y((y_1(a_0))^{-1})\\
&= y(y_2(a_0)) \cdot y(a_2^{-1} \cdot a_1) \cdot y((y_1(a_0))^{-1}) . \end{aligned} 
\end{equation}
Thus
\[\delta_{y_1,y_2}(a \cdot y(\xi)) = y_2(a) \cdot
y(y_2(a_0)) \cdot y(a_2^{-1} \cdot a_1) \cdot y((y_1(a_0))^{-1}) 
\cdot (y_1(a))^{-1} .
\]
We conclude that
\[|\{y_2(a) \cdot
y(y_2(a_0)) \cdot y(a_2^{-1} \cdot a_1) \cdot y(y_1(a_0^{-1})) \cdot y_1(a^{-1}) :
a\in A, y\in Y\}|\geq |A|\cdot |Y| .\]
That is, the conclusion (\ref{eq:erme}) is true.

{\em Case 1(c): $\xi = y_0(s)$.} We start as in case 1(b):
 (\ref{eq:ods}) and (\ref{eq:deliu}) still hold. 
By the definition (\ref{eq:shos}) of $S$,
 there are $y_1,y_2 \in Y$ distinct and $a_1,a_2\in A$ such that
$\delta_{y_1,y_2}(s) = a_2^{-1} \cdot a_1$.
Now, because we are in case 1(c) and not in case 1(b), 
we have $\xi = y_0(s)$ instead of
$\xi = a_0\cdot s$.
We replace (\ref{eq:clargo}) by the following calculation:
\[\begin{aligned}
y(y_2(\xi) \cdot y_1(\xi)^{-1}) &= y(y_2(y_0(s)) \cdot (y_1(y_0(s)))^{-1})
= y(y_2(y_0(s))) \cdot y(y_1(y_0(s^{-1})))\\
&= y_0(y(y_2(s))) \cdot y_0(y(y_1(s^{-1}))) =
 y_0(y(y_2(s) \cdot y_1(s^{-1})))\\
&= y_0(y(y_2(s) \cdot (y_1(s))^{-1})) = y_0(y(\delta_{y_1,y_2}(s))) = 
y_0(y(a_2^{-1} a_1)) .
\end{aligned}\]
(It is here that the fact that $\Upsilon$ is abelian
is used for the second time.)
Thus
\[ \delta_{y_1,y_2}(a\cdot y(\xi)) = y_2(a) \cdot
y_0(y(a_2^{-1} a_1)) \cdot (y_1(a))^{-1} .
\]
We conclude that
\[|\{y_2(a) \cdot y_0(y(a_2^{-1} a_1)) \cdot y_1(a^{-1}) :a \in A, y\in Y\}|
\geq  |A|\cdot |Y| .\]
In other words, (\ref{eq:irmi}) holds.

{\em Case 2: $\mathscr{O}\subset S$.} (Read: there is no pivot.)

Then $S$ must be rather large. From the definition (\ref{eq:shos}), 
it becomes clear that either $Y$ or $A$ must be rather large.
It is then no surprise that some crude techniques appropriate for large
sets shall be sufficient for our task.

Since $\delta_{y_1,y_2}$ is injective for $y_1\ne y_2$, the sets
\[R_{\xi} = \{(a_1,a_2,y_1,y_2) \in A \times A \times Y \times Y :
y_1\ne y_2,\; a_1 \cdot y_1(\xi) = a_2 \cdot y_2(\xi)\}\]
are disjoint as $\xi$ ranges in $G$. Choose $\xi_0 \in S$ such that
$|R_{\xi_0}|$ is minimal. Then
\[|R_{\xi_0}| \leq \frac{|A|^2 |Y| (|Y|-1)}{|S|} <
\frac{|A|^2 |Y|^2}{|S|} \leq
\frac{|A|^2 |Y|^2}{|\mathscr{O}|}
\]
and so
\[|\{(a_1,a_2,y_1,y_2) \in A\times A\times Y\times Y :
a_1 \cdot y_1(\xi_0) = a_2 \cdot y_2(\xi_0)\}| <
\frac{|A|^2 |Y|^2}{|\mathscr{O}|} + |A| \cdot |Y|.\]
Hence 
\begin{equation}\label{eq:dora}|A \cdot Y(\xi_0)| > \frac{|A|^2 |Y|^2}{ 
\frac{|A|^2 |Y|^2}{|\mathscr{O}|} + |A| \cdot |Y|} =
\frac{|A| |Y| |\mathscr{O}|}{|A| |Y| + |\mathscr{O}|} .
\end{equation}

As before, we must somehow remove $\delta_{y_1,y_2}^{-1}$ from $\xi_0$. 
By the injectivity of $\delta_{y_1,y_2}$, (\ref{eq:dora}) implies
\[|\delta_{y_1,y_2}(A\cdot Y(\xi_0))| >
\frac{|A| |Y| |\mathscr{O}|}{|A| |Y| + |\mathscr{O}|} .\]
 Equation
(\ref{eq:deliu}) is still valid. Since $\xi \in S$, we know that
$\delta_{y_1,y_2}(\xi_0) = a_2^{-1} a_1$ for some $a_1,a_2\in A$,
$y_1,y_2\in Y$ distinct. Thus, for $a\in A$, $y\in Y$,
\[\begin{aligned}
\delta_{y_1,y_2}(a \cdot y(\xi_0)) &= y_2(a \cdot y(\xi_0)) \cdot
(y_1(a \cdot y(\xi_0)))^{-1}\\ &= 
y_2(a) \cdot y_2(y(\xi_0)) \cdot (y_1(y(\xi_0)))^{-1} 
\cdot (y_1(a))^{-1}\\
&= 
y_2(a) \cdot y(y_2(\xi_0) \cdot (y_1(\xi_0))^{-1}) 
\cdot (y_1(a))^{-1} \\ &= y_2(a) \cdot y(\delta_{y_1,y_2}(\xi_0)) \cdot
(y_1(a))^{-1}\\ &= y_2(a) \cdot y(a_2^{-1} a_1) \cdot (y_1(a))^{-1} .
\end{aligned}\]
(It is here that the fact that $\Upsilon$ is abelian is used for the
third and last time.)
Hence
\[|\{y_2(a) \cdot y(a_2^{-1} a_1) \cdot y_1(a^{-1})
:a\in A, y\in Y\}|
> \frac{|A| |Y| |\mathscr{O}|}{|A| |Y| + |\mathscr{O}|} .\]
The inequality
$\frac{a b}{a + b} \geq \frac{1}{2} \min(a,b)$ is easy and true 
for all positive $a$, $b$. Hence we have proven (\ref{eq:ormo}).
\end{proof}

\subsection{Growth in unipotent groups under the action of the 
 diagonal}\label{subs:bore}

Proposition \ref{prop:guggen} does not require the group $G$ to be abelian.
The following is a natural application in which $G$ is non-abelian.

\begin{prop}\label{prop:avio}
Let $G$ be any semisimple group of Lie type.
Let $B$
be a Borel subgroup of $G$ defined over a field $K$.
Let $T$ be a maximal torus of $G$ contained in $B$, and let 
$U$ be the maximal unipotent subgroup of $B$. Assume that the exponential
map $\exp:\mathfrak{u}\to U$  from the Lie algebra 
$\mathfrak{u}$ of $U$ to $U$ itself is bijective.

Let $D\subset T(\overline{K})$ be a finite set 
such that, for every root $\alpha$ 
of $G$ relative to $T$, the restriction $\alpha|_D$ is injective.
Then, for any finite set $A\subset U(\overline{K})$,
\begin{equation}\label{eq:forgli}
|(A \cup D)_{20} \cap U(\overline{K})| > \frac{|A| |D| |\mathscr{O}|}{|A| |D| + 
|\mathscr{O}|},\end{equation}
where $\mathscr{O}$ is the subgroup of $U(\overline{K})$ generated by 
$\{t u t^{-1} : t \in \langle D\rangle, u\in \langle A\rangle\}$.
\end{prop}
If $\mathscr{O}$ is infinite and $A$, $D$ are finite, (\ref{eq:forgli}) reads as follows:
$|(A \cup D)_{20} \cap U(\overline{K})| 
\geq |A| |D|$. (We shall work only with finite
fields $K$, and thus $\mathscr{O}$, $A$ and $D$ will always be finite;
we only mention the case of infinite sets in passing.)

If $G$ is a subgroup of $\GL_n(K)$, $K$ a field of characteristic
$=0$ or $\geq n$,
then the exponential map $\exp:\mathfrak{u}\to U$ is invertible and,
in particular, injective. (The Taylor series for $\exp x$ and $\log x$
terminate at $x^{n-1}$, and the denominators of the coefficients of
the terms up to $x^{n-1}$ in either series are
not divisible by any primes $\geq n$.)
\begin{proof}
We will apply Proposition \ref{prop:guggen} with $G = U(\overline{K})$ and
$\Upsilon$ equal to the group $\Upsilon = 
\{y_t:t\in T\}$ of automorphisms of $U(\overline{K})$, where
\[y_t:u \mapsto t u t^{-1}.\]
The set $A$ will be as given, the set $Y$ will be $\{y_t:t\in D\}$,
and, finally, $A_0 = A \cup A^{-1}$ and $Y_0 = Y \cup Y^{-1}$. 

We need only check condition (\ref{eq:conocon}). 
Let $t$ be an element of $D^{-1} D$ other than the identity. Because
$\alpha|_D$ is injective for every root $\alpha$, we know that 
$\alpha(t)\ne 1$ for every root $\alpha$. We need to show that, if
$t g t^{-1} = g$ for some $g\in U(K)$, then $g$ is the identity.

We may write\footnote{For $G = \SL_n$, what follows amounts to the following
prosaic observation: if $t$ is a diagonal matrix with distinct eigenvalues,
and $g$ is an upper-triangular matrix with $1$'s on the diagonal, then
$t g t^{-1} = g$ can be true only if $g$ is the identity.} 
$g = \exp(\vec{v})$, where $\vec{v}$ lies on the Lie algebra $\mathfrak{u}$
of $U$. We have $t g t^{-1} = t \exp(\vec{v}) t^{-1}
= \exp(\Ad_t(\vec{v}))$. 
Because the exponential map $\exp:\mathfrak{u} \to U$
is injective, we shall have 
$t g t^{-1} = g$ if and only if $\Ad_t(\vec{v}) = \vec{v}$.

We may write $\vec{v}$ as a sum $\sum_{\alpha} \vec{v}_{\alpha}$ of
elements $\vec{v}_{\alpha}$ of the root spaces corresponding to the positive
roots $\alpha$. Then 
\[\Ad_t(\vec{v}) = \Ad_t(\sum_{\alpha} \vec{v}_\alpha) =
\sum_{\alpha} \Ad_t(\vec{v}_{\alpha}) = \sum_{\alpha} \alpha(t) \cdot
\vec{v}_\alpha .\]
Since $\alpha(t) \ne 1$ for every root $\alpha$, we conclude that 
$\Ad_t(\vec{v}) = \vec{v}$ implies $\vec{v}_\alpha = 0$ for every $\alpha$,
i.e., $\vec{v} = 0$. Hence $g = \exp(\vec{v})$ is the identity.
\end{proof}

Say that we want to apply Prop.\ \ref{prop:avio} to the study of growth
in Borel subgroups. An obvious question arises: how do we obtain
a large set of diagonal elements $D\subset T(\overline{K})$ and a large
set of unipotent elements $A\subset U(\overline{K})$? Unipotent elements
can generally be got by means of an easy pigeonhole argument, as,
for any two matrices $g_1$, $g_2$ having distinct eigenvalues and lying in
 the same conjugacy class in $B$, the quotient $g_1^{-1} g_2$ is unipotent.
Obtaining diagonal elements is a harder problem, but we will need to study
it at any rate; we will solve it in \S \ref{subs:richmat}. Once we have
enough diagonal elements, we will usually be able to obtain a large subset
$D$ of them on which every root is injective by means of a covering argument.

An exception occurs when we can obtain many commuting elements
inside the kernel of a root. We can still use Prop.\ \ref{prop:guggen};
the set $Y$ being used need not lie in a torus - it can lie in any
abelian subgroup. Some case work is needed, however. In \S \ref{subs:grotes}, we will
study the
matter in detail for the special case $G=\SL_3$.

\subsection{A sum-product theorem in $(\mathbb{Z}/p\mathbb{Z})^n$}\label{subs:sumpro}
We will prove a sum-product theorem of $(\mathbb{Z}/p\mathbb{Z})^n$. As a
matter of fact, we shall use only the case $n=2$; a theorem close
to the one we need was already proven for $n=2$ by Bourgain \cite{B2}.
However, \cite{B2} requires the assumption that $|A|>p^{\epsilon}$.
We cannot assume $|A|>p^{\epsilon}$ in our applications. We thus need
to prove a sum-product theorem ourselves without that restriction.
(Our statement will be less precise than that in \cite{B2} in another
respect.)

We may as well start by reproving the sum-product theorem for
$\mathbb{Z}/p\mathbb{Z}$ using Proposition \ref{prop:guggen}. In this we are 
simply following upon the steps of \cite{GK} or \cite{B}. The matter
will take only a few lines.

\begin{lem}\label{lem:sumprod}
 Let $p$ be a prime. Let $A\subset \mathbb{Z}/p\mathbb{Z}$. 
Assume $|A|<p^{1-\delta}$, $\delta>0$. Then
\begin{equation}\label{eq:akbar}
\text{either}\;\;\;\; |A\cdot A|\gg |A|^{1 + \epsilon} \;\;\;\; \text{or}
\;\;\;\; |A+A|\gg |A|^{1+\epsilon},\end{equation}
where $\epsilon>0$ and the implied constants depend only on $\delta$.
\end{lem}
\begin{proof}
Suppose (\ref{eq:akbar}) does not hold with implied constants equal to $1$. 
Then, by the Katz-Tao Lemma
(\cite{TV}, Lemma 2.53), there is a subset $A'\subset A$ with
$|A'|\geq \frac{1}{2} |A|^{1 - \epsilon} - 1$ and
\begin{equation}\label{eq:riv}
|A'\cdot A' - A' \cdot A'| \ll |A|^{1 + O(\epsilon)} ,\end{equation}
where the implied constant is absolute. We have to show that this is 
impossible.

Let $G = \mathbb{Z}/p\mathbb{Z}$, $\Upsilon = (\mathbb{Z}/p\mathbb{Z})^*$
(acting on $G$ by multiplication), $A_0 = \{1\}$, $Y_0 = \emptyset$, 
and set both $A$ and $Y$ in the statement of Prop.\ \ref{prop:guggen}
equal to $A'$. Since there are no zero divisors in $\mathbb{Z}/p\mathbb{Z}$,
condition (\ref{eq:conocon}) is satisfied. Thus, we may apply
Prop.\ \ref{prop:guggen}, and obtain that
\[|A'\cdot A' + A'\cdot A - A'\cdot A' + A'\cdot A' -A'\cdot A' + A'\cdot A'|
\geq \frac{1}{2} \min(|A|^2,p) \gg |A|^{1 + \delta}.\]
Hence, by the Pl\"unnecke-Ruzsa estimates (\cite{TV}, Cor.\ 6.29),
\[|A'\cdot A' - A'\cdot A'|\gg |A'|^{1 + \frac{\delta}{6}}.\]
For any $\epsilon<\delta/6$, this is in contradiction to (\ref{eq:riv})
provided that $|A|$ is larger than a constant depending only on $\delta$
and $\epsilon$. We may in fact assume that $|A|$ is larger than a constant,
as otherwise (\ref{eq:akbar}) is trivial. We have reached a contradiction.
\end{proof}

\begin{prop}\label{prop:nehr}
Let $p$ be a prime. Let $A\subset (\mathbb{Z}/p\mathbb{Z})^n$, $n\geq 1$.
Assume that either $|A|<p^{1-\delta}$, $\delta>0$, or
$p^{k+\delta}<|A|<p^{k+1 - \delta}$, $\delta>0$, $1\leq k < n$. Then
\[\text{either\;\;\;\;} |A\cdot A| \gg |A|^{1+\epsilon} 
\text{\;\;\;\; or \;\;\;\;} |A+A| \gg |A|^{1 +\epsilon},\]
where $\epsilon>0$ and the implied constant depend only on $n$ and $\delta$.
\end{prop}
Stronger statements are possible. Doing away with the conditions
$|A|<p^{1-\delta}$, $p^{k+\delta}<|A|<p^{k+1-\delta}$
would take some detailed case work and a
catalogue of counterexamples: consider $A = (\mathbb{Z}/p\mathbb{Z})
\times \{0\} \times \dotsb \times \{0\}$, for example.
 
\begin{proof}
We proceed by induction on $n$. For $n=1$, the statement is
true by Lemma \ref{lem:sumprod}. Let $\pi_j:(\mathbb{Z}/p\mathbb{Z})^n
\to \mathbb{Z}/p\mathbb{Z}$ be the projection map to the $j$th coordinate;
let $\pi_{\setminus j}:(\mathbb{Z}/p\mathbb{Z})^n\to 
(\mathbb{Z}/p\mathbb{Z})^{n-1}$ be the projection map to all coordinates
save the $j$th one.
We may assume that $A$ is a subset of $((\mathbb{Z}/p\mathbb{Z})^*)^n$:
if at least half of $A$ lies in $((\mathbb{Z}/p\mathbb{Z})^*)^n$,
then we may work with $A\cap ((\mathbb{Z}/p\mathbb{Z})^*)^n$ instead of $A$,
and if more than half of $A$ lies outside $((\mathbb{Z}/p\mathbb{Z})^*)^n$,
then $|A \cap \pi_j^{-1}(0)| > \frac{1}{2 n} |A|$ for some $j$, and
we may pass to $\pi_{\setminus j}(A\cap \pi_j^{-1}(0))$ and apply 
the inductive hypothesis.

Assume, then, that $n>1$ and $A \subset ((\mathbb{Z}/p\mathbb{Z})^*)^n$. 
Suppose first that $\min(|A|^{1/n},p^{\delta/n}) \leq |\pi_1(A)| \leq p^{1 - \delta/n}$.
Then, by Lemma \ref{lem:sumprod}, either
\begin{equation}\label{eq:kirsch}\begin{aligned}
|\pi_1(A\cdot A)| &= |\pi_1(A)\cdot \pi_1(A)| \gg |\pi_1(A)|^{1 + \epsilon}
\\ &\geq 
\min\left(p^{\frac{\delta \epsilon}{n}}, |A|^{\epsilon/n}\right) \cdot
|\pi_1(A)|
\geq |A|^{\delta e/n^2} \cdot |\pi_1(A)|
\end{aligned}\end{equation}
or
\begin{equation}\label{eq:discont}\begin{aligned}
|\pi_1(A + A)| &= |\pi_1(A) + \pi_1(A)| \gg |\pi_1(A)|^{1 + \epsilon}\\
&\geq 
\min\left(p^{\frac{\delta \epsilon}{n}}, |A|^{\epsilon/n}\right) \cdot
|\pi_1(A)|
\geq |A|^{\delta \epsilon/n^2} \cdot |\pi_1(A)|.
\end{aligned}\end{equation}
Let $x\in \mathbb{Z}/p\mathbb{Z}$ be such that the number of
elements of $S_x = A \cap (\pi_1^{-1}(\{x\})$ is maximal. 
Then $|S_x| \geq \frac{|A|}{|\pi_1(A)|}$. Let us examine the consequences
of (\ref{eq:kirsch}) and (\ref{eq:discont}).

If (\ref{eq:kirsch}) holds, then
\[\begin{aligned}
|A\cdot A\cdot A| &\geq |A\cdot A\cdot S_x| \geq
|\pi_1(A\cdot A)| |S_x| \geq |A|^{\delta \epsilon/n^2} 
|\pi_1(A)| |S_x|\\ &\geq |A|^{\delta \epsilon/n^2} 
|\pi_1(A)| \cdot \frac{|A|}{|\pi_1(A)|} =
|A|^{1 + \delta \epsilon/n^2} .
\end{aligned}\]
Since $(\mathbb{Z}/p\mathbb{Z})^*$ is abelian, we may use the
Pl\"unnecke's inequality (\cite{TV}, Cor.\ 6.28) to obtain
\begin{equation}\label{eq:dordor}
|A\cdot A| \gg |A|^{1 + \frac{\delta \epsilon}{3 n^2}} .\end{equation}

If (\ref{eq:discont}) holds, then one shows that
\[|A+A|\gg |A|^{1 + \frac{\delta \epsilon}{3 n^2}} \]
in exactly the same way that we showed (\ref{eq:dordor}).
Thus we are done with the case
$p^{\delta/n} \leq |\pi_1(A)| \leq p^{1 - \delta/n}$.

Suppose now that either $|\pi(A)| < \min(|A|^{1/n},p^{\delta/n})$
or $p^{1 - \delta/n} < |\pi(A)| \leq p$. Choose $x\in \mathbb{Z}/p\mathbb{Z}$
such that the number of elements of $S_x = A \cap (\pi_1^{-1}(\{x\}))$ is
maximal. If $|A|<p^{1 - \delta}$, then $|A|^\frac{n-1}{n} < |S_x| <
p^{1 - \delta}$; if
$p^{k+\delta} < |A| < p^{k+1 - \delta}$, $k\geq 1$, then either
$p^{k+\frac{n-1}{n} \delta} < |S_x| < p^{k+1 - \delta}$ or
$p^{k-1+\delta} < |S_x| < p^{k - \frac{n-1}{n} \delta}$. 
In all of these
cases, we may apply the inductive hypothesis (with $\frac{n-1}{n} \delta$
instead of $\delta$), and, moreover,
$|S_x|>|A|^{\frac{\delta}{n}}$. Hence either
\begin{equation}\label{eq:misch}
|S_x \cdot S_x| \gg |S_x|^{1 + \epsilon} \geq |A|^{\frac{\delta \epsilon}{n}}
\cdot |S_x|
\end{equation}
or
\begin{equation}\label{eq:egg}
|S_x +  S_x| \gg |S_x|^{1 + \epsilon} \geq |A|^{\frac{\delta \epsilon}{n}}
\cdot |S_x|.
\end{equation}
If (\ref{eq:misch}) holds, then
\[\begin{aligned}
|A\cdot A\cdot A| &\geq |A\cdot S_x \cdot S_x| \geq |\pi_1(A)| \cdot 
|S_x\cdot S_x|\\ &\gg |\pi_1(A)|\cdot |A|^{\frac{\delta \epsilon}{n}}
\cdot |S_x| \geq |A|^{1 + \frac{\delta \epsilon}{n}}\end{aligned}\]
and so, by Pl\"unnecke's inequality,
\[|A\cdot A| \gg |A|^{1 + \frac{\delta \epsilon}{3 n}} .
\]
If (\ref{eq:egg}) holds instead, we obtain in exactly the same way that
\[|A+A| \gg |A|^{1 + \frac{\delta \epsilon}{3 n}} .\]
\end{proof}

\begin{Rem} There is now an alternative route to the one taken in this subsection. Instead of proceeding as above,
one may derive Prop.\ \ref{prop:nehr} from \cite[Thm.\ 5.4]{T2}.
\end{Rem}

\subsection{Linear relations over rings}\label{subs:schw}
The consequences of the sum-product theorem we are about to derive are closely
related to incidence theorems. Such theorems have been linked to
sum-product phenomena ever since Elekes's brief and elegant
proof \cite{E} of the sum-product
theorem over $\mathbb{R}$ (originally due to Erd\"os and Szemer\'edi \cite{ES}) 
by means of 
an incidence theorem over $\mathbb{R}$ (first proven by Szemer\'edi and Trotter \cite{ST}).
Over finite fields, the topological arguments that can be used to prove
incidence theorems over $\mathbb{R}$ do not work: over
$\mathbb{Z}/p\mathbb{Z}$, a line does not divide the plane into two halves.
Thus, it seems necessary to prove incidence theorems using sum-product
results,
rather than the other way around. This is exactly what was done in
\cite[\S 6]{BKT}: Bourgain, Katz and Tao proved an incidence theorem
over $\mathbb{Z}/p\mathbb{Z}$ using their sum-product theorem.

We shall now prove -- over $(\mathbb{Z}/p \mathbb{Z})^n$, not over
$\mathbb{Z}/p\mathbb{Z}$ -- some results that are not quite the same
as incidence theorems, but are akin to them. The basic idea is the same:
we are to show that there cannot be too many linear relations among too
few objects.

We will need a very simple counting lemma.
\begin{lem}\label{lem:batho}
Let $A$, $B$ be finite sets. Let $S\subset A\times B$. For every $a\in A$,
let $B_a = \{b\in B: (a,b)\in S\}$. 

Then there is an $a_0\in A$ such that
\[\sum_a |B_{a_0}\cap B_a| \geq \frac{|S|^2}{|A| |B|} .\]
\end{lem}
\begin{proof}
For every $s\in S$, let $A_b = \{a\in A: (a,b)\in S\}$. Then
\[\begin{aligned}
\sum_{b\in B} |A_b|^2 &= \sum_{a_1\in A} \sum_{a_2\in A} 
|\{b\in B: a_1,a_2\in A_b\}|\\
&= \sum_{a_1\in A} \sum_{a_2\in A} |B_{a_1} \cap B_{a_2}|.\end{aligned}\]
Thus, if we let $a_0$ be such that $\sum_{a\in A} |B_{a_0} \cap B_a|$ is maximal,
\[\sum_{a\in A} |B_{a_0} \cap B_a| \geq \frac{1}{|A|} \sum_{b\in B} |A_b|^2 .\]
At the same time, by Cauchy's inequality,
\[\sum_{b\in B} |A_b|^2 \geq \frac{1}{|B|} \left(\sum_{b\in B} |A_b|\right)^2.\]
Finally,
\[\sum_{b\in B} |A_b| = |S|\]
and so we are done.
\end{proof}

The proposition we are about to prove can be summarised as follows. Let $X$
be a subset of a field (or a ring). Suppose that there are many linear
relations satisfied by many $(n+1)$-tuples of elements of $X$. Then
$X$ has a large subset that grows neither under addition nor under
multiplication. 
\begin{prop}\label{prop:pizar}
Let $R$ be a ring. Let $X\subset R$, $Y\subset (R^*)^n$, $n\geq 2$. Assume that
the projection $\pi_1:(R^*)^n\to R^*$ given by 
$(y_1,y_2,\dotsc,y_n)\mapsto y_1$ is injective on $Y$.
Assume as well that $|Y|> c |X|$, $0<c<1$.

For each $\vec{y}$, let $X_{\vec{y}}$ be a subset of $X^n$ with 
$|X_{\vec{y}}| > c |X|^n$. Suppose 
\begin{equation}\label{eq:golova}
\vec{y} \cdot X_{\vec{y}} = \{ y_1 x_1 + \dotsc + y_n x_n : \vec{x}\in 
X_{\vec{y}} \}
\end{equation}
is contained in $X$ for every $\vec{y}\in Y$.

Then there is a subset $X_0\subset X$ such that $|X_0|\gg c^C |X|$ and
\[|X_0 + X_0| \ll \frac{1}{c^C} |X_0|,\;\;\;\;\;\;\;\;
|X_0\cdot X_0| \ll \frac{1}{c^C} |X_0|,\]
where $C$ is a positive absolute constant and the implied constants are 
absolute.
\end{prop}
It would be desirable to replace both the assumption that $\pi_1|_Y$ is injective
and the assumption that $|Y|>c |X|$ by much weaker postulates. The statement as it stands
will do for ${\rm SL}_3$, but probably not for ${\rm SL}_n$, $n>3$. Weakening the assumptions would 
probably involve using the techniques in \cite{TV}, \S 2.7, instead of the Balog-Gowers-Szemer\'edi
theorem. (The following proof is, incidentally, the only place where the Balog-Gowers-Szemer\'edi
theorem is used in this paper.)
\begin{proof}
Let 
$X'$ be the set of all $x\in X$ such that \[|\{((x_2,x_3,\dotsc,x_n),\vec{y})\in X^{n-1}\times Y :
(x_1,x_2,\dotsc,x_n)\in X_{\vec{y}}\}| > \frac{1}{2} c |X|^{n-1} |Y|.
\]
We have
\begin{equation}\label{eq:fardant}
\sum_{\vec{y}\in Y} |\{(x_1,x_2,\dotsc,x_n)\in X_{\vec{y}}: x_1\in X'\}|
> \frac{1}{2} c |X|^n |Y|
\end{equation}
as $\sum_{\vec{y}\in Y} |\{(x_1,x_2,\dotsc,x_n)\in X_{\vec{y}}\}| >
\sum_{\vec{y}\in Y} c |X|^n = c |X|^n |Y|$ and the contribution of the terms
with $x_1\notin X'$ is clearly $\leq \frac{1}{2} c  |X|^n |Y|$. Immediately
from (\ref{eq:fardant}), $|X'| > \frac{1}{2} c |X|$.

Define
$X_{\vec{y}}' = \{(x_1,x_2,\dotsc,x_n)\in X_{\vec{y}}: x_1\in X'\}$.
Let $\vec{y}_0\in Y$ be such that $|X_{\vec{y}_0}'|$ is maximal. By
(\ref{eq:fardant}) and the pigeonhole principle,
$|X_{\vec{y}_0}'|> \frac{1}{2} c |X|^n$. Now, by (\ref{eq:golova}), 
\[|\{y_{0,1} x_1 + y_{0,2} x_2 + \dotsc + y_{0,n} x_n :
(x_1,x_2,\dotsc,x_n)\in X_{\vec{y}_0}\}'| \leq |X|.\]
We apply the Balog-Szemer\'edi-Gowers theorem (\ref{prop:bsg}) with
\[A_1 = y_{0,1} X_{\vec{y}_0}',\; 
A_2 = y_{0,2} X_{\vec{y}_0},\; \dotsc ,\;
A_n = y_{0,n} X_{\vec{y}_0}\] and
\[S = \{(y_{0,1} x_1, \dotsc, y_{0,n} x_n): 
 (x_1,x_2,\dotsc,x_n)\in X_{\vec{y_0}}'\},\]
and obtain that there is a subset $X'' \subset X'$ such
that
\begin{equation}\label{eq:orgone}|X''| \gg c |X'|\;\;\;\;\;\text{and}\;\;\;\;\;
|y_{0,1} X'' + y_{0,1} X''| \ll \frac{1}{c^{C_1}} |X''|,\end{equation}
where $C_1$ is a positive absolute constant and the implied constants
are also absolute. Obviously, $|X'' + X''| = |y_{0,1} X'' + y_{0,1} X''|$,
and so $|X'' + X''|\ll \frac{1}{c^{C_1}} |X''|$.

Apply Lemma \ref{lem:batho} with $A = X''$, $B = Y \times X^{n-1}$, and
\[S = \{(x_1,(\vec{y},(x_2,x_3,\dotsc,x_n)))\in A\times B:
(x_1,x_2,\dotsc,x_n)\in X_{\vec{y}}\} .\]
We obtain that there is a $x_0\in X''$ such that
\[\sum_{x\in X''} |B_{x_0}\cap B_x|\geq \frac{|S|^2}{|X''| \cdot |Y| |X|^{n-1}},\]
where $B_x = \{(y,(x_2,\dotsc,x_n))\in Y\times X^{n-1} : (x,x_2,\dotsc,x_n)\in X_{\vec{y}}\}$.
Now, since $X''\subset X'$, we obtain from the definition of $X'$ that
\[|S| > |X''|\cdot \frac{1}{2} c |X|^{n-1} |Y|.\]
Thus
\begin{equation}\label{eq:etoile}
\sum_{x\in X''} |B_{x_0}\cap B_x| > \frac{1}{4} c^2 |X''| |X|^{n-1} |Y|.\end{equation}

Define the map $f:Y\times X^{n-1}\to R$ by
\[f(\vec{y},(x_2,x_3,\dotsc,x_n)) = x_0 y_1 + x_2 y_2 + x_3 y_3 + \dotsb + x_n y_n .\]
Since (\ref{eq:golova}) is a subset of $X$, the set $f(B_{x_0})$ is a subset of $X$. Let
$r_0\in f(B_{x_0})$ be such that
\[\sum_{x\in X''} |\{b\in B_{x_0} \cap B_x : f(b) = r_0\}|\]
is maximal. By (\ref{eq:etoile}), $\sum_{x\in X''} |\{b\in B_{x_0} \cap B_x : f(b) = r_0\}|$
is then $\geq \frac{1}{4} c^2 |X''| |X|^{n-2} |Y|$. 
For $x\in X''$ and $\vec{y}\in Y$ given, there are at most $|X|^{n-2}$ elements of $B_x$
such that $f(b) = r_0$. (This is so because, if $x_2$ varies and $y,x_3,x_4,\dotsc,x_n$
are held fixed, then $f(b)$ varies with $x_2$.) Thus, there are at least $\frac{1}{4} c^2
|X''| |Y|$ pairs $(x_1,\vec{y})\in X''\times Y$ such that there is at least one tuple 
$(x_2,x_3,\dotsc,x_n)\in X^{n-1}$ for which
\[(\vec{y},(x_1,x_2,\dotsc,x_n))\in B_{x_0} \cap B_x\]
and
\[x_0\cdot y_1 + x_2 \cdot y_2 + x_3\cdot y_3 + \dotsb + x_n\cdot y_n = r.\]

Let $S'\subset X''\times Y$ be the set of all such pairs $(x_1,\vec{y})$. 
For any
$(x_1,\vec{y})\in S'$, there are $x_2,x_3,\dotsc,x_n\in X$ such that
\[\begin{aligned}
x_1\cdot y_1 + x_2\cdot y_2 + \dotsb + x_n y_n &= (x_1 - x_0)\cdot y_1 + x_0\cdot y_1 +
x_2\cdot y_2 + \dotsb + x_n \cdot y_n\\
&= (x_1 - x_0)\cdot y_1 + r
\end{aligned}\]
and
\[
x_1\cdot y_1 + x_2\cdot y_2 + \dotsb + x_n\cdot y_n \in X.\]
Thus
\[\{
(x_1 - x_0)\cdot y_1 : (x_1,\vec{y})\in S'\}\subset X - r.\]
Hence \[|\{x\cdot y : (x,y)\in S''\}|\leq |X|,\]
where $S'' = \{(x,y)\in (X'' - x_0)\times \pi_1(Y) : (x+x_0,\pi_1^{-1}(y))\in S'\}$.
(Recall that $\pi_1:(y_1,y_2,\dotsc,y_n)\mapsto y_1$ is injective on $Y$.)
Clearly $|S''| = |S'|$, and so $|S''|\geq \frac{1}{4} c^2 |X''| |Y|$.

We now apply the Balog-Szemer\'edi-Gowers theorem (Prop.\ \ref{prop:bsg})
again, this time with multiplication,
not addition, as the operation, and the following inputs:
$n=2$, $A_1 = X'' - x_0$, $A_2 = \pi(Y)$, $S = S''$.
We obtain that there is a subset $X'''\subset (X'' - x_0)$
with
\[|X'''|\gg c |X''| \;\;\;\;\;\;\;\;\;\text{and}\;\;\;\;\;\;\;\;\;
|X'''\cdot X'''| \ll \frac{1}{c^{C_2}} |X'''|.\]
At the same time, because of (\ref{eq:orgone}),
\[|X''' + X'''| = |(X''' - x_0) + (X''' - x_0)| \ll \frac{1}{c^{C_1}} |X''| \ll
\frac{1}{c^{C_1 + C_2}} |X'''| .\]
We let $X_0 = X'''$ and are done.
\end{proof}

We can finally state and prove what we worked for in this subsection.
\begin{cor}\label{cor:espada}
Let $R = (\mathbb{Z}/p\mathbb{Z})^m$, $m\geq 1$. Let $X\subset R$, $Y\subset
(R^*)^{n}$, $n \geq 2$. Assume that, for some
$j\in \{1,2,\dotsc,n\}$,
the projection $\pi_j:(R^*)^{n}\to R^*$ given by 
$(y_1,y_2,\dotsc,y_{n})\mapsto y_j$ is injective on $Y$.
Assume that either $|X|\leq p^{1-\delta}$, $\delta>0$, or
$p^{k+ \delta} \leq |X| \leq p^{k+1 - \delta}$ for some $k\geq 1$, $\delta>0$.

For each $\vec{y}$, let $X_{\vec{y}}$ be a subset of $X^{n}$ such that
\[
\vec{y} \cdot X_{\vec{y}} = \{y_1 x_1 + \dotsc + y_{n} x_{n} : 
\vec{x}\in X_{\vec{y}}\}
\]
is contained in $X$. Then either 
\begin{equation}\label{eq:airpo}
|Y|\ll |X|^{1 - \eta}\;\;\;\;\;\;\;\text{or}
\;\;\;\;\;\;\;\;
|X_{\vec{y}}| \ll |X|^{n-\eta}\;\; \text{for some $\vec{y}\in Y$},
\end{equation}
 where $\eta>0$ and the implied constants depend 
only on $\delta$ and $m$.
\end{cor}
We could explain this as follows, leaving a few conditions aside. 
Let $X$ be a subset of a ring $R$. Consider an $n$-dimensional box $X^n$.
Let there be many ($\geq |X|^{1-\eta}$, $\eta$ small) linear forms $f$ such
that for each form $f$ there are many ($\geq |X|^{n-\eta}$) elements of the
box on which the form $f$ takes values in $X$. Corollary \ref{cor:espada}
shows that the situation just described {\em cannot} happen.
\begin{proof}
Immediate from Prop.\ \ref{prop:pizar} and Prop.\ \ref{prop:nehr}.
(If $j\ne 1$, permute the first and $j$th coordinates of $(R^*)^{n}$ before
applying Prop.\ \ref{prop:pizar}.)
\end{proof}

\section{Escape, non-singularity and their conditions}\label{sec:orwise}

Much of our work will consist in showing that certain statements are
generically true in an effective sense -- that is to say, they are true when
their parameters lie outside a variety of positive codimension
{\em and bounded degree}.
 We will then obtain quantitative bounds from
these effective results by means of the technique of
{\em escape from subvarieties}.

The following will be a typical situation. Say we are able to show that
 a map $f:G\to V$ from an algebraic group $G/K$ 
to a variety $V/K$ is non-singular almost everywhere
in an effective sense, meaning
that there
is a variety $X_G\subset
G$ of positive codimension in $G$ and bounded degree such that,
for every point $y$ in the image of the
 restriction $g:=f|_{G\setminus X_G}$, the preimage $g^{-1}(y)$ of $y$
consists of a bounded number of points (i.e., it is the union of a bounded
number of irreducible zero-dimensional varieties). This is a useful
situation to be in, as then, for any finite
subset $E\subset G(K)\setminus X_G(K)$,
the image $f(E)$ satisfies $|f(E)|\gg |E|$; since we are investigating growth,
we are certainly interested in maps that do not make sets smaller. 

Suppose we are simply given a set $E\subset G(K)$. Then, under a very broad
set of circumstances, escape from subvarieties will give us that there
are $\gg |E|$ elements of $E_k$ lying in $G(K)\setminus X_G(K)$,
where $k$ is bounded by a constant. Call the set of such elements
$E'$. Then, by what we said before, $|f(E')| \gg |E'|$, and so
$|f(E_k)| \geq |f(E')| \gg |E'| \gg |E|$, which is a conclusion
we will often desire.

\subsection{Escape from subvarieties}
Eskin, Mozes and Oh \cite{EMO} have shown how to escape from varieties
by means of a group action. While their result was formulated over
$\mathbb{C}$, it carries over easily to other fields.
The following proposition is based closely on \cite[Prop.\ 3.2]{EMO}.
\begin{prop}\label{prop:carbo}
Let $G$ be a group. Consider a linear representation of $G$
on a vector space $\mathbb{A}^n(K)$ over a field $K$.
Let $V$ be an affine subvariety of $\mathbb{A}^n$.

Let $A$ be a subset of $G$; let $\mathscr{O}$ be an $\langle A\rangle$-orbit
in $\mathbb{A}^n(K)$ not contained in $V$. Then there are constants $\eta>0$ and $m$
depending only on $\vdeg(V)$
such that,
for every $x\in \mathscr{O}$,
there are at least $\max(1, \eta |A|)$ elements $g\in A_m$ such that
$g x\notin V$.
\end{prop}
This may be phrased as follows: one can escape from $V$ by the action
of the elements of $A$.
\begin{proof}
Let us begin by showing that there are elements $g_1,\dotsc,g_l \in A_{r}$
such that, for every $x\in \mathscr{O}$, at least one of the
$g_i \cdot x$'s is not in $V$. (Here $l$ and $r$ are bounded
in terms of $\vdeg V$ alone.) We will proceed by descent (that is,
induction) on $\vdeg V$, paying special attention to the number
$s_V$ of irreducible components of $V$ of maximal dimension $\dim V$.
(Notice that $s_V$ is bounded in terms of $\vdeg V$: in fact,
$s_V \leq (\vdeg V)_{\dim V}$.)

We shall always pass from $V$ to a variety $V'$ with either (a)
$\dim V' < \dim V$ or (b) $\dim V' = \dim V$ and $s_{V'} < s_V$. Moreover,
$\vdeg V'$ will be bounded in terms of $\vdeg V$ alone. We will iterate
 until we arrive at a variety $V'$ of dimension $0$ with $s_{V'} = 0$, i.e.,
an empty variety.
It is clear that this process terminates in a number of steps bounded
in terms of $\vdeg V$ alone.

Let $V_+$ be the union of all irreducible components of $V$ of
maximal dimension (i.e., dimension $\dim V$). If $V_+$ and
$\mathscr{O}$ are disjoint, we set
$V' = V\setminus V_+$ and are done. Suppose otherwise.
Since $\mathscr{O}$ is not contained in $V_+$, we can find
$x_0\in V_+\cap \mathscr{O}$, $g\in A \cup A^{-1}$ such that
$g x_0 \notin V_+$, i.e., $x_0 \notin g^{-1} V_+$.
Hence the set of components of maximal dimension $\dim V$ in $V$ is not the
same
as the set of components of maximal dimension $\dim g^{-1} V = \dim V$
in $g^{-1} V$.  It follows that $V' = g^{-1} V \cap V$
does not contain $V_+$, and thus has fewer components of dimension
$\dim V$ than $V$ has.

We have thus passed from $V$ to $V'$, where either (a) $\dim V' <
\dim V$ or (b) $\dim V' = \dim V$ and $s_V' < s_V$. Bezout's theorem
assures us that $\vdeg V'$ is bounded in terms of $\vdeg V$ alone.
By the inductive hypothesis, we
already know that there are
$g_1',\dotsc,g_{l'}' \in A_{r'}$
such that, for every $x\in \mathscr{O}$, at least one of the
$g_i' \cdot x$'s is not in $V'$. (Here $l'$ and $r'$ are bounded
in terms of $\vdeg V'$ alone.)
Since at least one of the $g_i' \cdot x$'s
is not in $V' = g^{-1} V\cap V$, either one of the $g_i'\cdot x$'s is not
in $V$ or one of the $g_i'\cdot x$'s is not in $g^{-1} V$, i.e.,
one of the $g g_i' \cdot x$'s is not in $V$. Set
\[\begin{aligned}
g_1 &= g_1',\; g_2 = g_2',\; \dotsc,\; g_{l'} = g_{l'}'\\
g_{l'+1} &= g g_1',\; g_{l'+2} = g g_2',\; \dotsc,\; g_{2 l'} = g
g_{l'}',\;\;\;\;\; l = 2 l' .\end{aligned}\]
(As can be seen, $g_i \in A_r$, where $r = r'+1$.)
We conclude that, for every $x\in \mathscr{O}$, at least one of
the $g_i \cdot x$'s is not in $V$.

The rest is easy: for each $x\in \mathscr{O}$ and each $g\in A$,
at least one of the elements $g_i g \cdot x$, $1\leq i\leq l$
($g_i\in A_r$)
will not be in $V$.
Each possible $g_i g$ can occur for at most $l$ different elements $g\in A$;
thus, there are at least $\min(1,|A|/l)$ elements $h = g_i g$
of $A_{r+1}$ such that $h x \notin V$.
\end{proof}

Many statements can be proven by the same kind of induction that one
uses to prove escape.

%
%

\begin{prop}\label{prop:kartar}
Let $K$ be a field. Let $G/K$ be an algebraic subgroup of
$\GL_n/K$. Let $S$
be a subgroup of $G(\overline{K})$ contained in a subvariety $V$ of $G$
of positive codimension.

Then $S$ is contained in an algebraic subgroup $H$ of $G$ of positive
codimension and degree bounded in terms of $\vdeg(V)$ alone.
\end{prop}
\begin{proof}
We shall show that $S$ is contained in the stabiliser of a
subvariety of $G$, and that this stabiliser satisfies the conditions
required of $H$ in the statement. We will proceed by induction on
$\vdeg V$, focusing on $\dim(V)$ and $s_V$ (defined as in the proof of
\ref{prop:carbo}), which it encodes.
 We shall always pass from $V$ to a variety $V'$ with either
(a) $\dim(V')<\dim(V)$ or
(b) $\dim(V') = \dim(V)$ and $s_{V'} < s_V$.
Moreover, $\vdeg V'$ will be bounded in terms of $\vdeg V$ alone.
We will iterate until we either find an algebraic group containing $S$
or arrive at a variety $V'$ with dimension $0$ and $s_{V'} = 0$
(i.e., the empty variety).

Let $V_+$ be the union of irreducible components of $V$ of dimension $\dim(V)$.
If $V_+ \ne V_+^{-1}$, we set $V' = V\cap V^{-1}$; we shall have
either (a) $\dim(V')<\dim(V)$ or
(b) $\dim(V') = \dim(V)$ and $s_{V'} < s_V$, and, by
Bezout's theorem, the degree of $V'$ is bounded in terms of the degree of
$V$. Since $S$ is a group, $S = S^{-1} \subset V^{-1}(\overline{K})$, and so
$S\subset (V \cap V^{-1})(\overline{K}) = V'(\overline{K})$. We then use the
inductive hypothesis and are done. We may thus assume from here on that we are
in the other case, viz., $V_+ = V_+^{-1}$.

Suppose first that there is a pair $(g,x)\in (S,V_+(\overline{K}))$ such
that $g\cdot x$ lies outside $V_+(\overline{K})$.
 Then $V' = g V\cap V$
has either (a) $\dim(V')<\dim(V)$ or (b) $\dim(V') =
\dim(V)$ and $s_{V'}<s_V$, and,
by Bezout's theorem, the degree of $V'$ is bounded in terms of the degree of
$V$. Since $S$ is a group, $S = g S \subset g V(\overline{K})$, and so
$S\subset (V\cap g V)(\overline{K})$. We then use the inductive hypothesis
and are done. We may thus assume that there is no pair
$(g,x)\in (S,V_+(\overline{K}))$ such
that $g\cdot x$ lies outside $V_+(\overline{K})$.

Suppose now that there is a pair $y,z \in V_+(\overline{K})$ such that
$y\cdot z^{-1} \notin V_+(\overline{K})$. Then $V' =
V z^{-1} \cap V$ has either (a)
$\dim(V')<\dim(V)$ or (b) $\dim(V') = \dim(V)$ and $s_{V'}<s_V$, etc. At the same time,
by our previous assumption, there is no $g\in S$ such that $g z$ lies
outside $V_+(\overline{K})$; hence $S\subset V z^{-1}$. Since
$S\subset V$, we conclude that $S\subset V z^{-1}\cap V = V'$.
We use the inductive hypothesis and are done.

We are left with the case where $V_+ = V_+^{-1}$ and there is no
pair $y,z\in V_+(\overline{K})$ such that $y\cdot z^{-1} \notin
V_+(\overline{K})$. Then $V_+$ is an algebraic group. We are assuming
that there is no pair $(g,x)\in (S,V_+(\overline{K}))$ such
that $g\cdot x$ lies outside $V_+(\overline{K})$; since $V_+(\overline{K})$
is a group, it contains the identity, and thus we have that there
is no $g\in S$ such that $g\cdot e = g$ lies outside $V_+(\overline{K})$,
i.e., we have $S\subset V_+(\overline{K})$. We set $H = V_+$ and are done.
\end{proof}
\begin{Rem}
In the above, we have implicitly used the fact that multiplication in
a linear algebraic group does not change the degree of the varieties therein:
$\vdeg(g V) = \vdeg(V)$ (and, in particular, $\deg(g V) = \deg(V)$ for
pure-dimensional varieties $V$). This is the only sense in which we
have used ``linearity'' (i.e., the assumption in Prop.\ \ref{prop:carbo}
that we are working with a linear representation, and the condition
in Prop.\ \ref{prop:kartar} that $G$ be a subgroup of $\GL_n$).
\end{Rem}

\subsection{Non-singularity and almost-injectivity}

If a map $f$ is injective, then, for every finite subset $E$ of the
 domain, $|f(E)| = |E|$. If $f$ is such that the preimage $f^{-1}(\{x\})$
of every point $x$ consists of at most $k$ points, then
$|f(E)| \geq \frac{1}{k} |E|$. This simple fact lies at the root of
several of our arguments. 

\begin{Rem}
Injectivity already played a role in
section \S \ref{sec:grosp}. The idea both there and in the applications
we shall later give
to the results about to be given here is the following: if $f$ is a map from
a product $A\times B$ to a set $C$, and $f$ is ``almost injective''
in the sense just described, then, for any $E_1\subset A$, $E_2\subset B$,
the image $f(A,B)$ has $\geq \frac{1}{k} |A| |B|$ elements. In other words,
we have obtained a rather strong kind of growth, provided that $f$ can
be defined by means of ``allowable'' operations, e.g., group operations
involving only already accessible quantities.
\end{Rem}

First, let us see how non-singularity gives us ``almost injectivity''.
(A regular map $f:X\mapsto Y$
is said to be {\em non-singular} at a point $x = x_0$ if
its derivative $D f|_{x=x_0}$ at $x=x_0$ is a non-singular linear map
from $(T X)_{x=x_0}$ to $(T Y)_{y = f(x_0)}$.)

\begin{lem}\label{lem:ofor}
Let $X\subset \mathbb{A}^{m_1}$ and $Y\subset \mathbb{A}^{m_2}$ 
be affine varieties defined over a field $K$.
Let $f:X\to Y$ be a regular map.
Let $V$ be a subvariety of $X$ such that
the derivative $D f|_{x=x_0}$ of $f$ at $x=x_0$
is a nonsingular linear map for all $x_0$ on $X$ outside $V$.

Let $S\subset X(\overline{K})\setminus V(\overline{K})$. Then
\[|f(S)|\gg_{\vdeg(X),\deg_{\pol}(f)} |S|.\]
\end{lem}
\begin{proof}
It will be enough to show that the intersection of $X(\overline{K})
\setminus V(\overline{K})$ with the preimage $Z = f^{-1}(y_0)$ of
any point $y_0$ on $Y$ consists of a number of irreducible zero-dimensional
varieties (that is, points) bounded
in terms of $\deg(V)$ and the degree of the polynomials defining $f$.
Now $Z$ is the intersection $X \cap \bigcap_j X_j$, where $X_j$,
$1\leq j\leq n$, is the variety in $\mathbb{A}^m$ defined by
by $(f(x))_j = (y_0)_j$, where we denote by $y_j$ the
$j$th coordinate of an element $y$ of $\mathbb{A}^n$. 
Thus, by Bezout's theorem (Lem.\ \ref{lem:bezout}), the degree
$\vdeg(Z)$ of $Z$
is $\ll_{\vdeg(X), \vdeg(X_1),\dotsc, \vdeg(X_n)} 1$. 
The degree $\vdeg(X_j)$ of the hypersurface
$X_j$ is bounded in terms of the degree of the polynomial $(f(x))_j$,
and so \[\vdeg(Z) \ll_{\vdeg(X), \deg((f(x))_1),\dotsc, \deg((f(x))_n)} 1.\]  
Thus, it remains only to show that any point $x_0$
on $Z$ not lying on $V$ lies on a component of $Z$ of dimension $0$.

Suppose it were not so. Then there would be a direction $\vec{v}\ne 0$ such
that
\[D f|_{x=x_0}(\vec{v}) = 0;\]
any direction $\vec{v}\ne 0$ on the tangent space to $Z$ at $x=x_0$ would do.
Then $D f_{x = x_0}$ would have to be singular. However, this would mean
that $x_0$ would have to lie on $V$. Contradiction.
\end{proof}

We can avoid a subvariety in an algebraic group by escape from subvarieties.
\begin{lem}\label{lem:lemfac}
Let $G\subset \GL_n$ be an algebraic group defined over a field $K$.
Let $V$ be a subvariety of $G$ such that $V(K)$ is a proper subset of
$G(K)$. Let $E\subset G(K)$ be a set
of generators of $G(K)$. 

Then 
\[|E_k \cap (G(K)\setminus V(K))| \gg_{\vdeg(V)} |E|,\]
where $k \ll_{\vdeg(V)} 1$.
\end{lem}
\begin{proof}
By escape from subvarieties (Prop.\ \ref{prop:carbo}) with $A = E$,
$V$ as given, $x=1$, and $G$ and $\mathscr{O}$ both equal to $G(K)$.
(We are implicitly using the fact that $G$ is contained in an affine
space, viz., $\mathbb{A}^{n^2}$.)
\end{proof}

\begin{cor}\label{cor:gotrol}
Let $G\subset \GL_n$ be an algebraic group and $Y\subset \mathbb{A}^{m}$ 
an affine variety, both defined over a field $K$.
Let $f:G\to Y$ be a regular map.
Let $V$ be a subvariety of $G$ such that
$V(K)$ is a proper subset of $G(K)$. Assume that 
the derivative $D f|_x$ of $f$ at $x$
is a nonsingular linear map for all $x$ on $G$ outside $V$.

Let $E\subset G(K)$ be a set
of generators of $G(K)$. Then
\[|f(E_k \cap (G(K)\setminus V(K)))|\gg_{\vdeg(G), \vdeg(V), \deg_{\pol}(f)} |E|,\]
where $k\ll_{\vdeg(V)} 1$.
\end{cor}
\begin{proof}
Immediate from Lemma \ref{lem:lemfac} and Lemma \ref{lem:ofor} --
the latter with $m_1 = n^2$, $m_2=m$ and $S = E_k \cap (G(K)\setminus V(K))$.
\end{proof}

Lemma \ref{lem:lemfac} has as one of its assumptions that $V(K)$ be a proper
subset of $G(K)$. In practice, we will often want to assume instead that
$V$ is a proper subvariety of $G$. Let us see how to obtain the former
assumption
using the latter one.

In the statement below, {\em perfect} and {\em reductive} are standard
technical terms (from abstract
 algebra and the theory of algebraic groups, respectively).
The group $\SL_n$ (defined over any field $K$) is
reductive, and a product of reductive groups is reductive as well.
This is all we will need to know when applying Lem.\ \ref{lem:utilo}
in the present paper.

\begin{lem}\label{lem:utilo}
Let $G\subset \GL_n$ be an irreducible algebraic group defined over a field
$K$. Assume either that $K$ is perfect or that $G$ is reductive.
 Let $V/\overline{K}$ be a proper subvariety of $G$. Then
\[V(K) \subsetneq G(K)\]
provided that $|K|$ is larger than a constant depending only on $n$,
$\vdeg(V)$ and $\vdeg(G)$.
\end{lem}
The assumption that $K$ is perfect or $G$ is reductive will be used
only in the case of $K$ infinite. When $K$ is finite, we will use a counting
argument that does not require the assumption. (The assumption would be
fulfilled in any case, as every finite field is perfect.) When $K$ is
infinite, we do not need to assume that $|K|$ is larger than a constant
depending only on $n$, $\vdeg(V)$ or $\vdeg(G)$. (Of course, when
$K$ is infinite, such an
assumption
is satisfied immediately anyhow, since $|K|=\infty$.) 
\begin{proof}
{\em Case 1: $K$ finite.}
Since $G$ is irreducible and $V\subset G$ is a proper subvariety
of $G$, the maximal dimension $m$ of the components of $V$ is $\leq
\dim(G)-1$. Hence, by (\ref{eq:otoronco}) and the fact that $V\subset 
\GL_n\subset
\mathbb{A}_n^2$,
\[|V(K)|\ll_{\vdeg(V),n} |K|^{\dim(G)-1}.\]

At the same time, by the Lang-Weil theorem \cite[Thm.\ 1]{LW},
the projective closure $\overline{G}$ of $G$ satisfies
\[|\overline{G}(K)| - |K|^{\dim(G)} = O_{\vdeg(G), n}\left(|K|^{\dim(G) -
    \frac{1}{2}}\right).\]

Since $G$ is irreducible, so is $\overline{G}$, and hence the intersection
of $\overline{G}$ with the hyperplane at infinity (i.e., the part of
projective space $\mathbb{P}^{n^2}$ that is not in affine space
$\mathbb{A}^{n^2}$) has dimension $<\dim(G)$. We can use either
the Lang-Weil theorem or an estimate such as (\ref{eq:otoronco}) again, and
obtain
\[|(\overline{G}\setminus G)(K)| \ll_{\vdeg(G),n} |K|^{\dim(G)-1}.\]

Hence
\[|G(K)\setminus V(K)| = |G(K)| - |V(K)| \gg_{\vdeg(G),n} |K|^{\dim(G)} -
O_{\vdeg(V),n}(|K|^{\dim(G)-1/2}),\]
which is positive for $|K|$ greater than a constant depending only on 
$\vdeg(G)$, $\vdeg(V)$ and $n$.

{\em Case 2: $K$ infinite.} By \cite[Cor.\ V.18.3]{Bor}, $G(K)$ is
Zariski-dense in $G$, that is to say, it is not contained in any 
proper subvariety of $G$. In particular,
$G(K)$ is not contained in $V$.
\end{proof}

It may have seemed odd at first sight that Lem.\ \ref{lem:ofor} required
a map to be non-singular outside a variety. In fact, this is a natural
condition; for example, a map between two spaces of the same dimension
is non-invertible precisely when the determinant $\delta$ of its derivative
does not vanish, and we can certainly see that $\delta=0$ defines
a variety.

The following lemma is in the spirit of what was just said. The lemma could
be stated in much more general terms; the fact that $G$ will be an algebraic
group is helpful but not essential.
\begin{lem}\label{lem:remor}
Let $G\subset \GL_n$ be an algebraic group
defined over a field $K$. Let $X/K$ and $Y/K$ be affine varieties such that
$\dim(G) = \dim(Y)$. Let $f:X\times G\to Y$ be a regular map. Let
$f_x:G\to Y$ be defined by $f_x(g) = f(x,g)$.

Then there is a subvariety $Z_{X\times G} \subset X\times  G$ such that,
for all $(x,g_0)\in (X\times G)(\overline{K})$, the derivative
\[(D f_x)|_{g=g_0}:(T G)|_{g=g_0} \to (T Y)|_{f(x,g_0)}\]
is non-singular if and only if $(x,g_0)$ does not lie on $Z_{X\times G}$.
Moreover, 
\begin{equation}\label{eq:coroco}
\vdeg(Z_{X\times G}) \ll_{\vdeg(X\times G),\deg_{\pol}(f),n} 
1.\end{equation}
\end{lem}
\begin{proof}
For $g_0\in G(\overline{K})$, consider the map
\begin{equation}\label{eq:feas}
g \mapsto f_{x}(g_0 g).\end{equation}
Its derivative at $g= I$ is nonsingular precisely when the derivative
of $f_{x}$ at $g = g_0$ is nonsingular.
Now, the derivative of (\ref{eq:feas}) at $g=I$ is nonsingular precisely when
a $\dim(G)$-by-$\dim(G)$ determinant $D$ is non-zero. The entries of 
the determinant $D$ are
polynomials on the entries of $g$ and $x$; hence,
$D = 0$ defines a variety $Z_{X\times G}$. The degree of $D$
(as a polynomial) is bounded in terms of $n$ and $\deg_{\pol}(f)$;
thus, $D=0$ defines a variety of degree $\ll_{n,\deg_{\pol}(f)}$,
and so (\ref{eq:coroco}) follows by Bezout's theorem.
\end{proof}

\subsection{Sticking subgroups in generic directions}

Let $H_1, H_2,\dotsc H_k$ be algebraic subgroups of an algebraic group $G/K$.
Say the tangent spaces $\mathfrak{h}_j\subset \mathfrak{g}$ to 
$H_j\subset G$
at the origin are such that the dimension of their sum equals the sum of
their dimension. Then we might possibly like to conclude that, for any finite
sets $E_j\subset H_j(K)$, 
\begin{equation}\label{eq:rotodo}
|E_1 \cdot E_2 \dotsb E_k| \gg |E_1|\cdot |E_2| \dotsb |E_k|.\end{equation}
 
Unfortunately, matters are not so simple. By escape and a few simple
arguments, we would indeed be able to obtain such a conclusion, provided
that we assumed that $E_j$ generates $H_j(K)$. We will not, however,
be able to assume as much in the applications that will come up later:
we will be provided with a generating set $A$ of $G(K)$, but not 
with generating sets of $H_j(K)$. The solution is to multiply {\em conjugates}
of the subgroups $H_j(K)$, rather than the subgroups themselves. Because
$A$ generates $G(K)$, we will be able -- by escape -- to take conjugates
of $H_j(K)$ by generic elements of $G(K)$. 
As we shall see, this is good enough to obtain conclusions much like
(\ref{eq:rotodo}) -- except for the fact that they will involve conjugates of $E_j$ by elements of $A_k$, 
rather than the sets $E_j$
themselves.

We recall that every algebraic group $G/K$ 
acts on its Lie algebra (i.e., its tangent space $\mathfrak{g}$ at the origin)
 by conjugation;
the {\em adjoint map} $\Ad_g:\mathfrak{g}\to \mathfrak{g}$
is the action of an element $g\in G(\overline{K})$.
Recall as well the definition of {\em linear independence} of subspaces
given in \S \ref{subs:indep}.
\begin{lem}\label{lem:vili}
Let $G$ be an algebraic group defined over a field $K$. 
Let $V_1, V_2,\dotsc ,V_k$ be linear subspaces of $\mathfrak{g}(\overline{K})$,
where $\mathfrak{g}$ is the tangent space to $G$ at the origin.
Suppose that there are $g_1,g_2,\dotsc,g_k\in G(\overline{K})$ such that
the linear spaces
\begin{equation}\label{eq:toroky}
\Ad_{g_1}(V_1), \Ad_{g_2}(V_2), \dotsc, \Ad_{g_k}(V_k)\end{equation}
are linearly independent.

Then there is a proper subvariety $X\subset G^k$ 
such that, for all $g=(g_1,g_2,\dotsc,g_k)\in G^k(\overline{K})\setminus
X(\overline{K})$, the spaces
(\ref{eq:toroky}) are linearly independent. Moreover,
$\vdeg(X)\ll_{\dim(G)} 1$.
\end{lem}
\begin{proof}
Let $v_{r,1},v_{r,2},\dotsc ,v_{r,l_r}$ be a basis for $V_r$,
$1\leq r\leq k$. For $g = (g_1,g_2,\dotsc,g_k) \in G^k(\overline{K})$,
let $w_1(g) = \Ad_{g_1}(v_{1,1})$,
$w_2(g) = \Ad_{g_1}(v_{1,2})$,\dots ,
$w_{l_1}(g) = \Ad_{g_1}(v_{1,{l_1}})$, $w_{l_1+1}(g) = \Ad_{g_2}(v_{2,1})$,
\dots, $w_{m}(g) = \Ad_{g_k}(v_{k,l_k})$, where $m = \sum_{1\leq r\leq k} 
l_r$. We are told that the spaces $\Ad_{g_1'}(V_1), \Ad_{g_2'}(V_2),\dotsc, \Ad_{g_k'}(V_k)$
are linearly independent for {\em some} $g_1',g_2',\dotsc,g_k'\in G(\overline{K})$; this is the same as saying that
the vectors $w_1(g'), w_2(g'),\dotsc, w_m(g')$ are linearly
independent for {\em some} $g'\in G^k(\overline{K})$.

Let $n = \dim(G)$. Let $v_{m+1},v_{m+2},\dotsc,v_n$ be $n-m$ vectors in
$\mathfrak{g}(\overline{K})$ such that
\[w_1(g'), w_2(g'),\dotsc, w_m(g'),v_{m+1}, v_{m+2}\dotsc, v_n\] are
linearly independent. Then the determinant $\delta(g)$ of the $n$-by-$n$ matrix
having
\[w_1(g), w_2(g),\dotsc, w_m(g),v_{m+1}, v_{m+2}\dotsc, v_n\]
as its rows is non-zero for $g=g'$. Thus, the subvariety $X$ of $G^k$ defined
by $\delta(g)= 0$ is a proper subvariety of $G^k$. For all $g\in G^k(
\overline{K})$ not on $X$, the determinant $\delta(g)$ is non-zero,
and thus 
$w_1(g), w_2(g),\dotsc, w_m(g),v_{m+1}, v_{m+2}\dotsc, v_n$ are
linearly independent; in particular,
$w_1(g), w_2(g),\dotsc, w_m(g)$ are linearly independent. This is
the same as saying that the linear spaces
(\ref{eq:toroky}) are linearly independent for all $g\in G^k(\overline{K})$
not on $X$.
\end{proof}

\begin{prop}\label{prop:crece}
Let $G$, $H$ and $F$ be algebraic groups defined over a field $K$. 
Let $\phi:G\times H\to F$,
$\psi:G\times H\to G$ be regular maps satisfying
\begin{equation}\label{eq:astora} \phi(g,h_1 h_2\!) = \phi(g,h_1\!) \cdot \phi(\psi(g,h_1),h_2\!)\end{equation}
 for all $g\in G$, $h_1,h_2\in H$, and
\begin{equation}\label{eq:ostaro} \psi(\psi(g,h),h^{-1}) = g\end{equation}
 for all $g\in G$, $h\in H$.

Define $\phi_g:H\to F$ by $\phi_g(h) = \phi(g,h)$.
For all $g_0\in G(\overline{K})$, $h_0\in H(\overline{K})$, 
write $(D \phi_{g_0})|_{h=h_0}$ 
for the linear map from
$TH|_{h=h_0}$ to $TF|_{f = \phi(g_0,h_0)}$ given by
\begin{equation}\label{eq:malfro}(D \phi_{g_0})|_{h=h_0}
:= \left(\frac{\partial}{\partial h} \phi_{g_0}(h) |_{h=h_0}\right)(v).
\end{equation}
Assume that $(D \phi_{g_0})|_{h=e}$ 
is non-singular for all $g_0\in G(\overline{K})$ outside a proper subvariety $X_G$ of $G$.
Then 
\begin{enumerate}
\item\label{it:yuto}
$(D \phi_{g_0})|_{h=h_0}$ 
 is non-singular exactly when 
$(g_0,h_0\!)\in (G\times H)(\overline{K})$ lies outside a proper subvariety
$Y_{G\times H}$ of $G\times H$, 
\item\label{it:ahan} 
$\deg(Y_{G\times H}) \ll_{\deg_{\pol}(\phi),\; \deg_{\pol}(\psi),\; \dim(H)} 1$,
\item\label{it:phoro} the fibre $(Y_{G\times H})_{g=g_0}$ is a proper subvariety of $H$ for all $g_0\in G(\overline{K})$ not
on $X_G$, and
\item\label{it:socor}
 the fibre $(Y_{G\times H})_{h=h_0}$ is a proper subvariety of $G$ for all $h_0\in H(\overline{K})$.
\end{enumerate}
\end{prop}
We will need to use conditions (\ref{eq:astora}) and (\ref{eq:ostaro}) in order to prove
conclusion (\ref{it:socor}), and only for that purpose. The said conditions tell us that every point on $H$ is in some sense
like every other point. If we did not have them, (\ref{it:ahan}) and
(\ref{it:phoro}) would still hold.

Before we prove Prop.\ \ref{prop:crece}, let us see why we should care:
for a map $\phi$ that we are rather interested in, there is a $\psi$
such that (\ref{eq:astora}) and (\ref{eq:ostaro}) hold.

\begin{lem}\label{lem:agreste}
Let $G$ be an algebraic group defined over a field $K$; let $H_0/K$, $H_1/K$,\dots ,
$H_{\ell}/K$ be subgroups thereof.
Let $\phi:G'\times H' \to G$, $\psi:G'\times H'\to G'$ 
(where $G' = G^{\ell+1}$ and $H' = H_0 \times H_1\times \dotsb \times
H_{\ell}$,
$\ell$ arbitrary) 
be given by
\[\phi((g_0,g_1,\dotsc,g_{\ell}),(h_0,h_1,\dotsc,h_{\ell})) = g_0 h_0 g_0^{-1} \cdot g_1 h_1 g_1^{-1} \dotsb
 g_{\ell} h_{\ell} g_{\ell}^{-1} .
\]
and
\[\psi((g_0,g_1,\dotsc,g_{\ell}),(h_0,h_1,\dotsc,h_{\ell})) =
(g_0',g_1',\dotsc,g_{\ell}'),\]
where $g_{\ell}' = g_{\ell}$ and $g_j' = g_{j+1}' h_{j+1}^{-1} g_{j+1}^{-1} g_j$ for $0\leq j\leq \ell-1$.

Then $\psi$ and $\phi$ satisfy (\ref{eq:astora}) and (\ref{eq:ostaro}).
\end{lem}
\begin{proof}
Equation (\ref{eq:ostaro}) follows easily from the definition of $g_j'$. 

By the definition of $g_j'$,
\begin{equation}\label{eq:hattusas}\begin{aligned}
g_j' &= g_{\ell} h_{\ell}^{-1} g_{\ell}^{-1} \cdot
g_{\ell-1} h_{\ell-1}^{-1} g_{\ell-1}^{-1} \dotsb
g_{j+1} h_{j+1}^{-1} g_{j+1}^{-1} \cdot g_j\\
&= \phi((g_{j+1},\dotsc,g_{\ell}),(h_j,h_{j+1},\dotsc,h_{\ell}))^{-1} \cdot
g_j
\end{aligned}\end{equation}
for all $0\leq j\leq \ell$. Hence
\begin{equation}\label{eq:troya}
\phi((g_j,g_{j+1},\dotsc,g_{\ell}),(h_j,h_{j+1},\dotsc,h_{\ell})) g_j' h_j'
= g_j h_j g_j^{-1} \cdot g_j \cdot h_j' = g_j h_j h_j'\end{equation}
for $h_j'$ arbitrary. Applying (\ref{eq:troya}) and then (\ref{eq:hattusas}),
we conclude that
\[\phi((g_j,g_{j+1},\dotsc,g_{\ell}),(h_j,h_{j+1},\dotsc,h_{\ell})) g_j' h_j' g_j'^{-1}\]
equals
\[g_j h_j h_j' g_j^{-1} \cdot
\phi((g_{j+1},g_{j+2},\dotsc,g_{\ell}),(h_{j+1},h_{j+2},\dotsc,h_{\ell})).\]
 Using this last equality in turn for $j=0,1,\dotsc,\ell$, we obtain
that
\[\phi((g_0,g_{1},\dotsc,g_{\ell}),(h_0,h_{1},\dotsc,h_{\ell})) g_0' h_0' g_0'^{-1} g_1' h_1' g_1'^{-1}
\dotsb g_{\ell}' h_{\ell}' g_{\ell}'^{-1}\]
equals
\[g_0 h_0 h_0' g_0^{-1} \cdot g_1 h_1 h_1' g_1^{-1} \dotsb g_{\ell} h_{\ell} h_{\ell}' g_{\ell}^{-1};\]
this is the same as saying that (\ref{eq:astora}) holds.
\end{proof}

\begin{proof}[Proof of Proposition \ref{prop:crece}]
Define
\begin{equation}\label{eq:rojot}\begin{aligned}
\rho_{g_0,h_0}(v) &:= \left(\frac{\partial}{\partial h} ((\phi(g_0,h_0))^{-1} \phi(g_0,h_0 h))\right) |_{h=e}(v)\\
&= \left(\frac{\partial}{\partial h} \phi(\psi(g_0,h_0),h)\right) |_{h=e}(v).\end{aligned}\end{equation}
It is clear that $(D \phi_{g_0})|_{h=h_0}$ 
is non-singular if and only if $\rho_{g_0,h_0}$ is non-singular. Now
$\rho_{g_0,h_0}$ is a linear map from the vector space $V = (TH)|_{h=e}$ to the vector space
$W = (T F)|_{f = e}$; both $V$ and $W$ are independent of $g_0$ and $h_0$. Hence $\rho_{g_0,h_0}$ is non-singular
exactly when a $\dim(V)$-by-$\dim(V)$ determinant $\delta$ equals $0$. The entries of $\delta$ are polynomials on the coordinates of
$g_0$ and $h_0$. Let $Y_{G\times H}$ be the subvariety of $G\times H$
defined by $\delta=0$. Then conclusion (\ref{it:yuto}) holds by
definition.
It is clear
that $\deg(Y_{G\times H})\ll_{\deg(\delta)} 1$; since $\deg(\delta)
\ll_{\deg_{\pol}(\phi),\deg_{\pol}(\psi),\dim(H)} 1$, it follows that
$\deg(Y_{G\times H}) \ll_{\deg_{\pol}(\phi),\deg_{\pol}(\psi),\dim(H)}
1$. Thus
conclusion (\ref{it:ahan}) holds.

By the assumptions of the proposition,
$(D \phi_{g_0})|_{h=e}$ 
is non-singular for all $g_0\in G(\overline{K})$ 
outside $X_G$. This is the same as saying that
 $(g_0,e)$ lies outside  $Y_{G\times H}$, and so
$(Y_{G\times H})_{g=g_0}$ is a subvariety of $H$ not containing $e$;
in particular, $(Y_{G\times H})_{g=g_0}$ is a proper subvariety of $H$, i.e.,
conclusion (\ref{it:phoro}) holds.

Now, by (\ref{eq:rojot}), $\rho_{g_0,h_0}$ is non-singular exactly when $(D \phi_{\psi(g_0,h_0)})|_{h=e}$ 
 is non-singular, i.e., exactly when $\psi(g_0,h_0)$ lies outside $X_G$. By
(\ref{eq:ostaro}), $r_{h_0}:g \mapsto \psi(g,h_0)$ is a regular map with a regular map as its inverse; hence,
$r_{h_0}^{-1}(X_G)$ is a proper subvariety of $G$. By what we just said, $(Y_{G\times H})_{h= h_0} =
r_{h_0}^{-1}(X_G)$, and so we have obtained conclusion (\ref{it:socor}).
\end{proof}

We shall now see how to escape from a variety $Y_{G\times H}$ such as the
one given by Prop.\ \ref{prop:crece}, even if we are not given
a set of generators of $H(K)$.

\begin{lem}\label{lem:mka}
Let $G\subset GL_n$ be an irreducible algebraic group defined over a 
 field $K$. Assume either that $K$ is perfect or that $G$ is
reductive.
Let $H/\overline{K}$
be an algebraic subgroup of $G$. Let
$Y_{G\times H}$ be a proper subvariety 
of $G\times H$ such that
 the fibre $(Y_{G\times H})_{h=h_0}$ is a proper subvariety of $G$ for all $h_0\in H(\overline{K})$.

 Let $A \subset G(K)$ be a set of generators
of $G(K)$, and let $E$ be a subset of $H(K)$.
Then there is a $g_E\in A_k$, $k\ll_{\vdeg(Y_{G\times H})} 1$,  such that at least
 \[\gg_{\vdeg(Y_{G\times H})} |E|\]
elements of $E$ lie outside $(Y_{G\times H})|_{g=g_E}$,
provided that
$|K|$ is larger than a constant depending only on $n$ and
$\vdeg(Y_{G\times H})$.
\end{lem}
\begin{proof}
Let $h_0\in H(\overline{K})$ be arbitrary. By one of the assumptions,
 the fibre $(Y_{G\times H})_{h=h_0}$ is a proper subvariety of $G$.
Hence, by escape in $G$ (Lemma 
\ref{lem:lemfac}, together with Lemma \ref{lem:utilo}),
there is a $g_1 \in A_{k_1}$, $k_1 \ll_{\vdeg(Y_{G\times H})} 1$,
such that $g_1$ does not lie on $(Y_{G\times H})_{h=h_0}$. 
Let $Y_1\subset H$ 
be the union of
all connected components of $(Y_{G\times H})_{g=g_1}$ that contain
elements of $E$; since (by the definition of $g_1$ and $X_G$) the
fibre $(Y_{G\times H})_{g=g_1}$ is a proper subvariety of $H$, clearly
$Y_1$ is a proper subvariety of $H$ as well. Moreover,
$\vdeg(Y_1)\ll_{\vdeg(Y_{G\times H})} 1$.

If $|E\cap Y_1(K)| < \frac{1}{2} |E|$, we have 
$|E\setminus (E\cap Y_1(K))| \geq \frac{1}{2} |E|$ and we are done. Assume
otherwise, and let $E_1 = E\cap Y_1(K)$. Choose a point $h_1\in E_1$ lying
on a component of $Y_1$ of maximal dimension. By escape in $G$ 
(Lemmas \ref{lem:lemfac} and \ref{lem:utilo}), there is a $g_2\in A_{k_2}$,
$k_2 \ll_{\vdeg(Y_{G\times H})} 1$, such that $g_2$ does not lie in 
$(Y_{G\times H})_{h=h_1}$.
Let $Y_2$ be the union of all connected components of 
$Y_1 \cap (Y_{G\times H})_{g=g_2}$ containing elements of $E$. Since $Y_2$ does not contain $h_1$, it does not contain all components of $Y_1$ of maximal
dimension. Hence either (a) $\dim(Y_2) < \dim(Y_1)$ or (b) 
$s_{Y_2} < s_{Y_1}$, where, for a variety $V$, we write
 $s_V$ for the number of components of maximal dimension. Moreover, by Bezout's
 theorem (Lem.\ \ref{lem:bezout}), $\vdeg(Y_2) \ll_{\vdeg(Y_1),\vdeg(Y_{G\times
     H})} 1$.
 
 Starting with $j=2$, we recur, doing what we just did: 
 if $|E_{j-1} \cap Y_j(K)| < \frac{1}{2} |E_{j-1}|$,
 we have $|E_{j-1} \setminus (E_{j-1} \cap Y_j(K))|\geq 
 \frac{1}{2} |E_{j-1}|
 \geq \frac{1}{2^j} |E|$, and we stop; otherwise, we let 
 $E_j = E_{j-1} \cap Y_j(K)$, we choose a point $h_j\in E_j$ lying
on a component of $Y_j$ of maximal dimension,
  we find a $g_{j+1}\in A_{k_{j+1}}$, 
$k_{j+1} \ll_{\vdeg(Y_{G\times H})} 1$, such that $g_{j+1}$ does not lie on
$(Y_{G\times H})_{h=h_j}$,
 we let $Y_{j+1}$ be the union of connected components of 
$Y_j \cap (Y_{G\times H})_{g = g_{j+1}}$ containing elements of $E$, etc.
Thanks to Bezout's theorem, we reach $Y_j = \emptyset$ (and thus we stop) after a number of steps $\ll_{\vdeg(Y_{G\times H})} 1$. Hence
$|E_{j-1} \setminus (E_{j-1} \cap Y_j(K))|
 \geq \frac{1}{2^j} |E|$ (where $j$ is the index $j$ we are at when we stop) implies
$|E_{j-1} \setminus (E_{j-1} \cap Y_j(K))| \gg_{\vdeg(Y_{G\times H})} 1$.
\end{proof}

It is time to put together what we have proven in this subsection.
\begin{prop}\label{prop:galoshes}
Let $G\subset \GL_n$ be an irreducible
algebraic group defined over a field $K$; let $H_0/\overline{K}$, 
$H_1/\overline{K}$,\dots ,
$H_{\ell}/\overline{K}$ be algebraic subgroups thereof.
Assume either that $K$ is perfect or that $G$ is reductive.
Write $\mathfrak{g}$ for the Lie algebra of $G$, and 
$\mathfrak{h_j}$ for the Lie algebra of $H_j$, $0\leq j\leq \ell$.
Assume there are $g_0,g_1,\dotsc,g_{\ell}\in G(\overline{K})$ such that
\begin{equation}\label{eq:hotho}\Ad_{g_0}(\mathfrak{h}_0), \Ad_{g_1}(\mathfrak{h}_1),\dotsc ,
\Ad_{g_{\ell}}(\mathfrak{h}_{\ell})  
\end{equation}
are linearly independent.

 Let $A \subset G(K)$ be a set of generators
of $G(K)$.
Then there are $g_0, g_1,\dotsc , g_{\ell}\in A_k$, $k\ll_n 1$, such that
\[|g_0 E_0 g_0^{-1} \cdot g_1 E_1 g_1^{-1} \cdot \dotsb \cdot
 g_{\ell} E_{\ell} g_{\ell}^{-1}|  \gg_{n,\deg(H_0),\deg(H_1),\dotsc,
\deg(H_\ell)}  |E_0| |E_1| \dotsb |E_{\ell}|\]
for any non-empty subsets $E_j\subset H_j(K)$, $0\leq j\leq \ell$.
\end{prop}
In the present paper, we will always use Prop.\ \ref{prop:galoshes} with
$G = \SL_n$, which is semisimple and hence reductive.
\begin{proof}
We will apply Prop.\ \ref{prop:crece} with $\overline{K}$ instead of $K$,
$G' = G^{\ell+1}$ instead of $G$, 
$H = H_0 \times H_1 \times \dotsb \times H_{\ell}$, $F = G$ (our $G$, that is, not
$G'$) and $\phi$ and $\psi$ as in Lem.\ \ref{lem:agreste}. The derivative
(\ref{eq:malfro}) is non-singular for $h_0=e$ whenever the linear spaces 
(\ref{eq:hotho}) are linearly independent; by Lemma \ref{lem:vili} and
the assumption on 
(\ref{eq:hotho}) (namely, that the spaces are independent for some
$(g_0,g_1,\dotsc,g_\ell)\in G'(
\overline{K})$), the spaces are independent for all 
$(g_0,g_1,\dotsc,g_\ell)\in G'(
\overline{K})$ outside a proper subvariety $X_{G'}$ of $G'$
(with $\vdeg(X_{G'})\ll_n 1$).
The conditions of Prop.\ \ref{prop:crece} are thus fulfilled, and we obtain a
variety $Y_{G'\times H}$ (with $\vdeg(Y_{G'\times H})\ll_n 1$)
as in its statement.

We now apply Lemma
\ref{lem:mka} to $G'$, $H$ and $Y_{G'\times H}$.
(We may use Lemma \ref{lem:mka} because we may assume that
$|K|$ is larger than a constant depending only on $n$, as otherwise the
statement we seek to prove is trivially true. Notice also that,
if $G$ is reductive, then $G' = G^{\ell+1}$ is reductive.) We then apply Lemma
\ref{lem:ofor} with $X = H$
and are done.
\end{proof}

\subsection{Examining subspaces at the origin}

Recall the definition of {\em linear independence} of subspaces
given in \S \ref{subs:indep}.
\begin{prop}\label{prop:elysium}
Let $G\subset \GL_n$ be an algebraic group defined over a field $K$; write
$\mathfrak{g}$ for its Lie algebra. Let $\mathfrak{h}$ be a subspace
of $\mathfrak{g}$. Suppose that there are elements
$\vec{g}_1, \vec{g}_2,\dotsc, \vec{g}_{\ell}$ of $\mathfrak{g}(\overline{K})$ such that 
the spaces
\begin{equation}\label{eq:redpla}
\mathfrak{h}, \lbrack \vec{g}_1, \mathfrak{h}\rbrack, \lbrack \vec{g}_2, \mathfrak{h}
\rbrack, \dotsc, \lbrack \vec{g}_{\ell}, \mathfrak{h}\rbrack\end{equation}
are linearly independent and of dimension $\dim(\mathfrak{h})$. Suppose that the characteristic
$\charac(K)$ of $K$ is either $0$ or greater than $k$, where
$k = \dim(\mathfrak{h})$.

Then there is a proper subvariety $X$ of $G^{\ell}$ such that,
for all $(g_1,g_2,\dotsc,g_{\ell})\in G^{\ell}(\overline{K})$ not on $X$,
the spaces
\[\mathfrak{h}, \Ad_{g_1}(\mathfrak{h}), \Ad_{g_2}(\mathfrak{h}),\dotsc,
\Ad_{g_{\ell}}(\mathfrak{h})\]
are linearly independent and of dimension $\dim(\mathfrak{h})$. Moreover, $\vdeg(X) \ll_n 1$.
\end{prop}
\begin{proof}
Let $e_1, e_2, \dotsc, e_k$ be a basis of $\mathfrak{h}$.
Write \[\theta = e_1\wedge e_2\wedge \dotsb \wedge e_k\;\;\;\; \text{and}
\;\;\;\;
 \theta_j = \lbrack \vec{g}_j, \lbrack \vec{g}_j, \lbrack \dotsb \lbrack
 \vec{g}_j, \theta\rbrack \dotsb \rbrack \rbrack \rbrack\;\;\;
\text{($k$ times)}\;\;\;\;\; \text{for $1\leq j\leq \ell$}.\] 
Here recall that, since the brack $\lbrack \cdot , \cdot\rbrack$ is
essentially a derivative (namely, the derivative of $\Ad_g$), it
interacts with $\wedge$ as in the product rule: $\lbrack \vec{g}, 
v \wedge w\rbrack = \lbrack \vec{g},v \rbrack \wedge w +
v \wedge \lbrack \vec{g}, w\rbrack$.

Let us examine
the wedge product
\[\Theta = 
\theta \wedge \theta_1 \wedge \theta_2 \wedge \dotsb \wedge \theta_{\ell}.\]
Any term of $\theta_j$ containing a term of the form
$\dotsb \wedge e_r \wedge \dotsb$ will be lost, as its wedge product with
$\theta$ will be $0$. The only terms of $\theta_j$ remaining are
$k!$ identical terms of the form
\[\omega_j = \lbrack \vec{g}_j, e_1\rbrack \wedge \lbrack \vec{g}_j,e_2\rbrack \wedge
\dotsb \wedge \lbrack \vec{g}_j, e_k\rbrack.\]
We thus have
\[\Theta = (k!)^{\ell}\cdot (\theta \wedge \omega_1 \wedge \omega_2 \wedge
\dotsb \wedge \omega_{\ell}).\]
By the condition stating that the spaces (\ref{eq:redpla}) are linearly
independent, we have $\Theta \ne 0$ (provided that, as we are assuming,
$\charac(K)$ does not divide $k!$).

Now, $\theta_j$ is a derivative, viz., the $k$th order derivative at the
origin
of
\[\Ad_{g_{j,1} g_{j,2} \dotsb g_{j,k}}(\theta) = 
\Ad_{g_{j,1}}(\Ad_{g_{j,2}}(\dotsb (\Ad_{g_{j,k}}(\theta)) \dotsb))\]
taken with respect to the variables $g_{j,1}$, $g_{j,2}$,\dots,
$g_{j,k}$ one time each, always in the same direction $\vec{g}_j$.
Hence $\Theta$ is itself a $(k \cdot \ell)$th order derivative (at
the origin) of
\begin{equation}\label{eq:cedric}\theta \wedge \bigwedge_{1\leq j\leq \ell} 
\Ad_{g_{j,1} g_{j,2} \dotsb_{j,k}}(\theta).\end{equation}
Since $\Theta$ is non-zero, it follows that (\ref{eq:cedric})
is not identically zero as the $g_{j,i}$ vary within $G(\overline{K})$. Setting $g_j = g_{j,1} g_{j,2} \dotsb g_{j,k}$,
we see that there are $g_1,g_2,\dotsc,g_{\ell}\in G(\overline{K})$
such that
\begin{equation}\label{eq:chort}
\theta \wedge \Ad_{g_1}(\theta) \wedge \Ad_{g_2}(\theta) 
\wedge \dotsc \wedge \Ad_{g_l}(\theta) = 0\end{equation}
does not hold. Define the variety $X$ by the equation (\ref{eq:chort}).
\end{proof}

\subsection{Subgroups of unipotent subgroups and tori}

We will later want to know what kinds of subgroups a torus can have. 
The following lemma will be enough.

We recall that a {\em torus} is an algebraic group isomorphic to
$(\GL_1)^m$ for some $m\geq 1$. (This should be clear by now, though we are used
to speaking of maximal or non-maximal tori {\em of} a group $G$, i.e., subgroups
of $G$ that happen to be tori.) Just as we often see  algebraic groups $G$
as (algebraic) subgroups of $\GL_n \subset \mathbb{A}^{n^2}$, it makes sense to consider 
tori $T$ (isomorphic to $(\GL_1)^m$) given as algebraic subgroups of 
$(\GL_1)^n\subset \mathbb{A}^n$.

A {\em character} $\alpha:T\to \GL_1$
of a torus $T\subset (\GL_1)^n$ is a map of the form
$(x_1,x_2,\dotsc,x_n)\to x_1^{a_1} x_2^{a_2} \dotsc x_n^{a_n}$ for some
$a_1,a_2,\dotsc,a_n\in \mathbb{Z}$ (called the {\em exponents} $a_j$ of $\alpha$).

\begin{lem}\label{lem:alin}
Let $K$ be a field.
Let $T/\overline{K} \subset (\GL_1)^n$ be a torus. 
Let $V/\overline{K}$ be a proper algebraic subgroup of $T$.

Then $V$ is contained in the kernel of a
non-trivial
character $\alpha:T\to \mathbb{A}^1$ whose exponents are bounded in terms of
$n$ and $\vdeg(V)$ alone.
\end{lem}
\begin{proof}
Let $H$ be the identity component of $V$; by the definition of
the degree of a variety (\S\ref{subs:convu}), the degree $\deg(H)$ 
of the irreducible variety $H$ is bounded in terms of $\vdeg(V)$ alone.
Now \cite[Prop.\ 3.3.9(c)]{BoG} (applied with $X=H$)
states that $H$ must be of the form 
$\phi_A(\widetilde{X(H)} \times (\GL_1)^{r})$, where $r = \dim(H)$,
$\widetilde{X(H)}$ is a closed subvariety of $(\GL_1)^{n-r}$ and
$\phi_A:(\GL_1)^n\to (\GL_1)^n$ is an (invertible) {\em monoidal 
transformation} (\cite[Def.\ 3.2.4]{BoG}) given by a matrix
$A\in \SL_n(\mathbb{Z})$. Since $\dim(H) = \dim(\GL_1^{r})$
and $\phi_A$ is invertible,
$\widetilde{X(H)}$ must be $0$-dimensional; since $H$ is connected
and $\phi_A$ is invertible, $\widetilde{X(H)}$ must be consist of a single
point; since $H$ is a group and $\phi_A$ is an isomorphism of algebraic
groups, that single point must be the identity. In other words,
$H = \phi_A(\{e\} \times (\GL_1)^{r})$, where $e$ is the identity in
$(\GL_1)^{n-r}$. 

Thus $H$ is in the kernel of the character
$g\to ((\phi_A)^{-1}(g))_j = (\phi_{A^{-1}}(g))_j$ for every $1\leq j\leq n-r$.
If $T$ were in the kernel of every such character, its dimension would
be $r$, i.e., the same as the dimension of $H$; since $T$ is irreducible
and $H$ is a proper subgroup of $T$, this cannot be the case. Let,
then, $\alpha_0:T\to \mathbb{A}^1$ be the restriction to $T$ of the
character $g\to (\phi_{A^{-1}}(g))_j$ for some $j$ for which such a
restriction is not trivial. The exponents of $g$ are entries of $A^{-1}$;
by \cite[Remark 3.3.10]{BoG}, the entries of $A^{-1}$ are $\ll_{n,\delta(H)}$,
where $\delta(H)$ is the ``essential degree'' of $H$
(as defined in \cite[\S 3.3.1]{BoG}). Now, by \cite[Prop.\ 3.3.2]{BoG},
$\delta(H) \leq \deg(H)$. Hence the exponents of $g$ are
$\ll_{n,\deg(H)} 1$.

The number $m$ of connected components of $V$ is bounded by $\vdeg(V)$.
Now $V$ must consist of $m$ cosets of the form
$x H$, where $x\in G(\overline{K})$ is such that $x^m$ lies on $H$.
We define $\alpha:T\to \mathbb{A}^1$ to be the character such that
$\alpha(g) = \alpha_0(g^m)$, and are done.
\end{proof}

Let $U$ be a unipotent subgroup of $\SL_3$. We need to classify the
subgroups of $U(\mathbb{Z}/p\mathbb{Z})$. This turns out to be an easy task.
\begin{lem}\label{lem:betson}
Let $K = \mathbb{Z}/p\mathbb{Z}$. Let $G = \SL_3$, and let $B/K$ be
a Borel subgroup thereof; let $U/K$ be the subgroup of unipotent matrices
of $B$.

Then every subgroup $H$ of $U(K)$ is conjugate in $B(\overline{K})$
to one of the following subgroups:
\begin{equation}\label{eq:nostor}H = \{I\},\end{equation}
\begin{equation}\label{eq:dotor}
H = U(K),\end{equation}
\begin{equation}\label{eq:mata1}
H = \left\{\left(\begin{matrix}1 & x & y\\0 &1 &0\\0 & 0 & 1\end{matrix}
\right) : x,y\in \mathbb{Z}/p\mathbb{Z}\right\},\end{equation}
\begin{equation}\label{eq:mata2}
H = \left\{\left(\begin{matrix}1 & 0 & y\\0 &1 &z\\0 & 0 & 1\end{matrix}
\right) : y,z\in \mathbb{Z}/p\mathbb{Z}\right\},\end{equation}
\begin{equation}\label{eq:matorner}
H = \left\{\left(\begin{matrix}1 & 0 & y\\0 &1 &0\\0 & 0 & 1\end{matrix}
\right) : y\in \mathbb{Z}/p\mathbb{Z}\right\},\end{equation}
\begin{equation}\label{eq:matb1}
H = \left\{\left(\begin{matrix}1 & x & 0\\0 &1 &0\\0 & 0 & 1\end{matrix}
\right) : x\in \mathbb{Z}/p\mathbb{Z}\right\},\end{equation}
\begin{equation}\label{eq:matb2}
H = \left\{\left(\begin{matrix}1 & 0 & 0\\0 &1 &z\\0 & 0 & 1\end{matrix}
\right) : z\in \mathbb{Z}/p\mathbb{Z}\right\},\end{equation}
\begin{equation}\label{eq:doloro}
H = \left\{\left(\begin{matrix}1 & x & y\\
0 &1 & x\\0 & 0 & 1\end{matrix}
\right) : x,y\in \mathbb{Z}/p\mathbb{Z}\right\},\end{equation}
\begin{equation}\label{eq:dogar}
H = \left\{\left(\begin{matrix}1 & x & \frac{x^2}{2}\\
0 &1 & x\\0 & 0 & 1\end{matrix}
\right) : x\in \mathbb{Z}/p\mathbb{Z}\right\}\;\;\;\;
\text{(if $p>2$)}.\end{equation}
\end{lem}
\begin{proof}
Let $N$ be the normal subgroup of $U$ consisting of the matrices of the
form \[\left(\begin{matrix}1 &0 &y\\0 &1&0\\ 0 & 0 & 1\end{matrix}\right).\]
We may identify $U(K)/N(K)$ with $\mathbb{Z}/p\mathbb{Z} \times
\mathbb{Z}/p\mathbb{Z}$ by the bijection
\[
\left(\begin{matrix}1 &x &y\\0 &1&z\\ 0 & 0 & 1\end{matrix}\right) N(K)
\mapsto  (x,z).\]
Consider $H' = H/(H\cap N(K))$, which can be seen as a subgroup of $U(K)/N(K)
\simeq \mathbb{Z}/p\mathbb{Z} \times \mathbb{Z}/p\mathbb{Z}$ by the
inclusion $H\subset U(K)$.
If $H' = \{(0,0)\}$, then either $H = \{I\}$ or $H$ is as in
(\ref{eq:matorner}).

Suppose $H' = \mathbb{Z}/p\mathbb{Z}\times
\mathbb{Z}/p\mathbb{Z}$. We may then choose two matrices
\[g= \left(\begin{matrix}1 & x& y\\0 &1 &z\\0 & 0 & 1\end{matrix}\right)\in H,
\;\;\;\;
g' = \left(\begin{matrix}1 & x'& y'\\0 &1 &z'\\0 & 0 & 1\end{matrix}\right)
\in H\]
with $x z' \ne x' z$ (as we can specify $x$, $x'$, $z$, $z'$ arbitrarily).
The two matrices $g$, $g'$ do not commute. Hence $g g' g^{-1} g'^{-1}\ne I$.
Because $U(K)/N(K)$ is abelian,
$g g' g^{-1} g'^{-1}$ must lie in $N(K)$; since $N(K)\simeq
\mathbb{Z}/p\mathbb{Z}$ and $g g' g^{-1} g'^{-1}\ne I$, we see
that $g g' g^{-1} g'^{-1}$ must generate $N(K)$. Hence
$N(K) \subset H$, and so, since
$H' = H/(H\cap N(K))$ is all of $\mathbb{Z}/p\mathbb{Z}\times
\mathbb{Z}/p\mathbb{Z}$, we conclude that $H$ is all of $U(K)$.

Suppose $H' = \mathbb{Z}/p\mathbb{Z}\times \{0\}$ or
$H' = \{0\} \times \mathbb{Z}/p\mathbb{Z}$. Then it is easy to show
that we are either in cases (\ref{eq:mata1}) or (\ref{eq:matb1})
(if $H' = \mathbb{Z}/p\mathbb{Z}\times \{0\}$) or
cases (\ref{eq:mata2}) or (\ref{eq:matb2}) 
(if $H' = \{0\} \times \mathbb{Z}/p\mathbb{Z}$).
(We initially obtain
\[\left(\begin{matrix}1 &x &rx\\0 & 1 & 0\\0 & 0 & 1\end{matrix}\right)\]
instead of (\ref{eq:matb1}), but this is conjugate to (\ref{eq:matb1})
in $B(K)$ by an element of $U(K)$. The same happens for (\ref{eq:matb2}).)

Suppose, finally, that $H'$ is of the form
$\{(x,rx) : x\in \mathbb{Z}/p\mathbb{Z}\}$ for some
$r\in \mathbb{Z}/p\mathbb{Z}$, $r\ne 0$.
If $H$ contains a non-trivial element of $N$,
we obtain (\ref{eq:doloro}) after conjugation by an element of
$B(\overline{K})$.
Suppose $H$ contains no non-trivial element of
$N(K)$. Then, for every $x\in \mathbb{Z}/p\mathbb{Z}$, there
is exactly one element $y = y(x)$ of $\mathbb{Z}/p\mathbb{Z}$ such that
\[\left(\begin{matrix}1 & x & y(x)\\ 0 & 1 & r x\\ 0 & 0 & 1\end{matrix}
\right)\]
is in $H$. Thus, for every $m\in \mathbb{Z}$,
\[\left(\begin{matrix}1 & 1 & y(1)\\ 0 & 1 & r\\0 &0 & 1\end{matrix}
\right)^m =
\left(\begin{matrix}1 & m & m\cdot y(1) + (1 + 2 + \dotsb + (m-1)) \cdot r\\ 0 & 1 & r m
\\0 &0 & 1\end{matrix}\right)\]
must be equal (mod $p$) to
\[\left(\begin{matrix}1 & m & y(m)\\ 0 & 1 & r m\\ 0 & 0 & 1\end{matrix}
\right) \in H.
\]
If $p=2$, we set $m=2$ and obtain a contradiction to our assumption that
$H$ contains no non-trivial element of $N(K)$. Assume, then, that
$p>2$. Then we obtain $y(x) = x y(1) + \frac{x (x-1)}{2} r =
 x\cdot (y(1) - r/2) + \frac{x^2}{2} r$.
Then
\[\left(\begin{matrix}1 & m & y(m)\\ 0 & 1 & r m\\ 0 & 0 & 1\end{matrix}
\right) =
\left(\begin{matrix}\rho & \rho^{-2} c & 0\\
0 & \rho & 0\\ 0 & 0 &\rho^{-2}\end{matrix}\right) \cdot
\left(\begin{matrix} 1 & x & \frac{x^2}{2}\\0 & 1 & x\\ 0 & 0 & 1
\end{matrix}\right)\cdot
\left(\begin{matrix}\rho & \rho^{-2} c & 0\\
0 & \rho & 0\\ 0 & 0 &\rho^{-2}\end{matrix}\right)^{-1}
\] where $c = y(1) - r/2$ and $\rho\in \overline{K}$ is any cube root of $r$.
This means that $H$ is a conjugate of (\ref{eq:dogar}) by an element
of $B(\overline{K})$, and so we are done.
\end{proof}

\section{Tori and conjugacy classes}\label{sec:torcon}
Let $A\subset \SL_n(K)$, $K$ any field. We mean to show that, if
$A$ grows slowly under multiplication, then (a) many elements of $A$ lie
on a torus, and (b) there are not many more conjugacy classes
intersecting $A$ than there
are elements on the torus. Somewhat counter-intuitively, we shall begin by
giving an {\em upper} bound on the number of elements of $A$ that can
lie on a torus.

The methods in this section seem to be robust as far as the group type
and the ground field are concerned.
We shall work -- by and large -- on $\SL_n(K)$, $n$ arbitrary,
rather than only on $\SL_3(\mathbb{Z}/p\mathbb{Z})$. A few lemmas will be
proven for all classical Chevalley groups.

\subsection{The intersection with a maximal torus: an upper bound}\label{sec:grog}
Let $A$ be a set of generators of $G = \SL_n(K)$. We shall show that,
given any torus $T$,
the intersection of $A$ with $T$ is not too large.

By a {\em classical Lie algebra} over a field $K$ we mean $\mathfrak{sl}_n$,
$\mathfrak{so}_n$ or $\mathfrak{sp}_{2 n}$ ($n\geq 1$). By a {\em classical
Chevalley group} over $K$ we mean $\SL_n$, $\SO_n$ or $\Sp_{2 n}$. We shall
see $\SL_n$, $\SO_n$ and $\Sp_{2 n}$ as subvarieties of the affine space of
matrices $M_n$. If $\mathfrak{g}$ is a Lie algebra defined over a field $K$,
we denote by $\mathfrak{g}^*$ the $K$-linear space of $K$-linear functions
on $\mathfrak{g}(K)$.

\begin{lem}\label{lem:trabe}
Let $\mathfrak{g}/K$ be a classical Lie algebra over a field $K$
with $\charac(K)>2$. Let $\mathfrak{t}$ be a Cartan subalgebra of
$\mathfrak{g}$.
Let $\Phi$ be its set of
roots and let $V= \mathfrak{t}^*$.
Then there
is a partition $\Phi = \Phi_1 \cup \Phi_2 \dotsb \cup \Phi_\ell$ 
such that each $\Phi_j$, $1\leq j\leq \ell$, is a basis of $V$.
\end{lem}
\begin{proof}
Let us consider each of the classical root systems individually.
We shall see them as abstract root systems, i.e., as subsets of $V$,
which can be seen simply as a linear space over $K$ with no further structure.

We look first at $A_n$. Then $V$ can be identified with the subspace of
$K^{n+1}$ for which the coordinates sum to $0$, and the
set of roots $\Phi$
with the set of vectors in $V$ having one coordinate equal to $1$, one
coordinate equal to $-1$, and all other coordinates equal to $0$.
Define $\Phi_j$ ($1\leq j\leq n+1$) to be the set of roots
$v_j - v_i$, $i\ne j$. Then every $\Phi_j$ is a basis of $V$.

Now look at $B_n$. Then $V$ can be identified with $K^n$, and $\Phi$
with the set of vectors having at most two coordinates in $\{-1,1\}$,
and all other coordinates equal to $0$. We let $\Phi_j$
($1\leq j\leq n$) be the set of roots $v_j - v_i$, $i\ne j$, together
with the root $v_j$; let $\Phi_{n+j}$ ($1\leq j\leq n$)
be the set of roots $v_j + v_i$, $i>j$, together with
$- v_j - v_i$, $i<j$, and $- v_j$.

Let us now consider $C_n$. Then $V$ can be identified with $K^n$, and
$\Phi$ with the set of vectors having two coordinates in $\{-1,1\}$
and all other coordinates equal to $0$, together with the vectors having
one coordinate in $\{-2,2\}$ and all other coordinates equal to $0$.
The choice of $\Phi_j$ is almost as for $B_n$:
we let $\Phi_j$
($1\leq j\leq n$) be the set of roots $v_j - v_i$, $i\ne j$, together
with the root $2 v_j$; let $\Phi_{n+j}$ ($1\leq j\leq n$)
be the set of roots $v_j + v_i$, $i>j$, together with
$- v_j - v_i$, $i<j$, and $- 2 v_j$.

Finally, we consider $D_n$. Then $V = K^n$, and $\Phi$ can be identified with
 the set of vectors
having two coordinates in $\{-1,1\}$ and all other coordinates equal to $0$.
Then let
$\Phi_j$ ($1\leq j\leq n-1$) be the set of roots $v_j - v_i$, $i\neq j$,
together with the root $v_j + v_n$; let $\Phi_{j + n - 1}$ ($1\leq j\leq n-1$)
be the set of roots $v_j + v_i$, $i>j$, together with the roots
$-(v_j + v_i)$, $i<j$, and the root $-(v_j + v_n)$.
\end{proof}
\begin{lem}\label{lem:ground}
Let $\mathfrak{g}/K$ be a classical Lie algebra over a field $K$
with $\charac(K)\ne 2$. Let $\mathfrak{t}$ be a Cartan subalgebra of
$\mathfrak{g}$. Let $\ell =\frac{\dim(G)}{\dim(T)} - 1$.
Then
there are elements $\vec{g}_1,\vec{g}_2,\dotsc,\vec{g}_\ell \in \mathfrak{g}$
such that the spaces
\begin{equation}\label{eq:purce}\mathfrak{t}, \lbrack \vec{g}_1, \mathfrak{t}\rbrack,
\dotsb  , \lbrack \vec{g}_\ell, \mathfrak{t}\rbrack\end{equation}
are linearly independent and of dimension $\dim(\mathfrak{t})$. 
\end{lem}
\begin{proof}
Let $\Phi = \Phi_1 \cup \Phi_2 \cup \dotsb \cup \Phi_\ell$ be a partition as
in Lemma \ref{lem:trabe}. For $1\leq k\leq \ell$, 
choose one non-zero element $v_{k,j}$ 
in the root space corresponding to each
element $\alpha_{k,j}$
of $\Phi_k$; denote the set of such elements for given $k$ by
$\{v_{k,j}\}_{1\leq j\leq m}$, where 
$m = |\Phi_k| = \dim(\mathfrak{t}^*)$.
Let $\vec{g}_k = \sum_{1\leq j\leq m} v_{k,j}$.
Then, for every $t\in \mathfrak{t}(K)$, we have 
$\lbrack t, \vec{g}_k\rbrack = \sum_j \alpha_{k,j}(t) \cdot v_{k,j}$.

Now let $e_1, e_2,\dotsc, e_m$ be a basis for $\mathfrak{t}(K)$.
(Since $\dim(\mathfrak{t}(K)) =
\dim(\mathfrak{t}^*)$, a basis of $\mathfrak{t}(K)$ has $m$ elements.)
Then the linear map 
$f_k:v\to \lbrack \vec{g}_k, v\rbrack$ from $\mathfrak{t}(K)$
to the span $V_k$ 
of the root spaces $\{\alpha_{k,j}\}_{1\leq j\leq m}$ is given by
 a square $m$-by-$m$ matrix with entries $\{\alpha_{k,j}(e_i)\}_{1\leq i,j
\leq m}$. Since (by Lem.\ \ref{lem:trabe}) 
the roots in $\Phi_k$ form a basis of $\mathfrak{g}^*$,
they are linearly independent, and so the matrix is non-singular.
Thus, the image of $f_k$ is all of $V_k$. 
In other words, $\lbrack \mathfrak{t},\vec{g}_j\rbrack$ equals
the span of the root spaces $\{\alpha_{k,j}\}_{1\leq j\leq m}$. 

By
\cite[\S 26.2, Cor.\ B]{Hum}, the Lie algebra $\mathfrak{g}$ is
the direct sum of $\mathfrak{t}$ and the root spaces.
Since $\lbrack \mathfrak{t},\vec{g}_j\rbrack
= - \lbrack \vec{g}_j,\mathfrak{t}\rbrack = \lbrack
\vec{g}_j,\mathfrak{t}\rbrack$,
we are done.
\end{proof}

If $G$ is a classical Chevalley
group, then both $G$ and all of its maximal tori are irreducible varieties over
any field $K$ (see, e.g., \cite{Bor}, \S 1.2, \S 8.5(2) and \S 8.7).

\begin{prop}\label{prop:grati}
Let $G\subset \GL_n$ be a classical Chevalley group defined over a field
$K$ with $\charac(K)\ne 2$.
Let $T$ be a maximal torus
of $G$ defined over $\overline{K}$.

 Let $A \subset G(K)$ be a set of generators
of $G(K)$, and let $E$ be a subset of $T(K)$.
Let $\ell =\frac{\dim(G)}{\dim(T)} - 1$.
Then there are $g_0, g_1,\dotsc , g_{\ell}\in A_k$, $k\ll_n 1$, such that
\[|g_0 E g_0^{-1} \cdot g_1 E g_1^{-1} \cdot \dotsb \cdot
 g_{\ell} E g_{\ell}^{-1}|  \gg_n |E|^{\ell+1}.\]
\end{prop}
\begin{proof}
By Proposition \ref{prop:galoshes} with $H_j =T$ and $E_j = E$
for $0\leq j\leq \ell$. (The condition on (\ref{eq:hotho}) is fulfilled by
Lemma \ref{lem:ground} and Proposition \ref{prop:elysium}.)
\end{proof}

\begin{cor}\label{cor:destin}
Let $G\subset \GL_n$ be a classical Chevalley group. Let $K$ be a field
with $\charac(K)\ne 2$.
Let $T$ be a maximal torus
of $G$ defined over $\overline{K}$.
 Let $A \subset G(K)$ be a set of generators
of $G(K)$.

Then
\[|A \cap T(K)| \ll_n |A_{k}|^{\frac{\dim(T)}{\dim(G)}},\]
where $k\ll_n 1$.
\end{cor}
\begin{proof}
Immediate by Prop.\ \ref{prop:grati} (with $E = A\cap T(K)$).
\end{proof}

\subsection{A lower bound on the number of conjugacy classes}
Let $A$ be a set of generators of $\SL_n(K)$. We shall show that there
are many conjugacy classes represented by elements of $A$ -- or, at any
rate, by elements of $A_k$.

Given a matrix $g$ in $\SL_n(K)$, we define
$\kappa(g)\in \mathbb{A}^{n-1}(K)$ to be the tuple
\[(a_{n-1},a_{n-2},\dotsc,a_1)\]
of coefficients of
\[\lambda^n + a_{n-1} \lambda^{n-1} + a_{n-2} \lambda^{n-2} + \dotsc +
a_1 \lambda + (-1)^{n} = \det(\lambda I - g) \in K\lbrack \lambda\rbrack\]
(the characteristic polynomial of $g$).

As is well-known, $\kappa(g) = \kappa(h g h^{-1})$ for any $h$, i.e.,
$\kappa(g)$ is invariant under conjugation. If $g$ is a regular semisimple
element of $\SL_n$ -- that is, if its eigenvalues are all distinct --
then $\kappa(g)$ actually determines the conjugacy class
$\Cl_G(g)$ of $g$.

\begin{lem}\label{lem:gole}
 Let $G = \SL_n$. Let $K$ be a field. For $h_0,h_1,\dotsc,h_n$,
define $f_{h_0,h_1,\dotsc,h_n}$ to be the map
\begin{equation}\label{eq:belthy}
f_{h_0,h_1,\dotsc,h_n}:g\mapsto (\kappa(h_0 g), \kappa(h_1 g), \dotsc,
\kappa(h_n g))\end{equation}
from $G$ to $\mathbb{A}^{(n-1) \cdot (n+1)} =
\mathbb{A}^{n^2 - 1}$.

Let $T/\overline{K}$ be a maximal torus of $G$.
Then
 there are $h_0\in G(\overline{K})$, $h_1\in T(\overline{K})$
 and $g_0\in G(\overline{K})$
such that the derivative of
$f_{h_0,h_1,h_1^2,\dotsc,h_1^n}$ at $g=g_0$ is a nonsingular linear map.
\end{lem}
\begin{proof}
We may write the elements of $G(\overline{K})$ so that the elements of
$T(\overline{K})\subset G(\overline{K})$ become diagonal matrices.
Let
\[g_0 = \left(\begin{matrix}
0 & 1 & 0 & \dotsb &0 & 0\\
0 & 0 & 1 & \dotsb &0 & 0\\
\vdots & \vdots & \vdots & \vdots & \vdots & \vdots\\
0 & 0 & 0 & \dotsb &1 & 0\\
0 & 0 & 0 & \dotsb &0 & 1\\
(-1)^{n-1} & 0 & 0 & \dotsb & 0 & 0\end{matrix}\right).\]
Let $\vec{r} = (r_1,r_2,\dotsc,r_n)$ be a vector in $\overline{K}^n$
with $r_1 \cdot r_2 \dotsb r_n = 1$. Define
\begin{equation}\label{eq:ororh}h_1 = \left(\begin{matrix}
r_{1} & 0 & \dotsc & 0\\
0 & r_{2} & \dotsc & 0\\
\vdots & \vdots & \vdots & \vdots\\
0 & 0 & \dotsc & r_{n}\end{matrix}\right)\end{equation}
for $1\leq i\leq n$.

Let us look, then, at the derivative at $g=I$ of $g\mapsto
\kappa(h_1^i g_0 g)$ for $1\leq i\leq n$. The derivative 
at $g=I$ of the map taking $g$ to the coefficient of $\lambda^{n-1}$ in
$\det(\lambda I - h_1^i g_0 g)$ (i.e., to $(-1)$ times the trace of $h_1^i
g_0 g$)
is equal to the map taking each matrix $\gamma$ in the tangent space
$\mathfrak{g}$ to $G$ at the origin to 
\[(-1)\cdot (r_{1}^i \gamma_{2,1} + r_{2}^i \gamma_{3,2} + \dotsc + r_{j}^i \gamma_{j+1,j} +
\dotsc + (-1)^{n-1} r_{n}^i \gamma_{1,n} ),
\]
where we write $\gamma_{i,j}$ for the entries of the matrix $\gamma$.

The derivative at $g=I$
of the map taking $g$ to the coefficient of $\lambda^{n-2}$
in $\det(\lambda I - h_1^i g_0 g)$ is the map taking each $\gamma$ in
$\mathfrak{g}$ to 
\[\begin{aligned}
r_{1}^i r_{2}^i \gamma_{3,1} &+ r_{2}^i r_{3}^i \gamma_{4,2} + \dotsc +
r_{j}^i r_{j+1}^i \gamma_{j+2,j} + \dotsc + 
r_{n-3}^i r_{n-2}^i \gamma_{n-1,n-3} + 
r_{n-2}^i r_{n-1}^i
\gamma_{n,n-2} \\
&+ r_{n-1}^i\cdot (-1)^{n-1} r_n^i \gamma_{1,n-1} +
(-1)^{n-1} r_{n}^i r_1^i \gamma_{2,n}.
\end{aligned}\]
In general, for $1\leq k\leq n-1$, the derivative at $g=I$ of the map taking
$g$ to the coefficient of $\lambda^{n-k}$ in $\det(\lambda I - h_1^i g_0 g)$ is the map taking $\gamma$ to
\begin{equation}\label{eq:prese}(-1)^k
\sum_{j=1}^{n-k} 
(r_{j}^i \cdot r_{\underline{j+1}}^i \dotsb r_{\underline{j+k-1}}^i) \cdot
\gamma_{\underline{j+k}, j} +
(-1)^{k+n-1} \sum_{j=n-k+1}^{n} 
(r_{j}^i \cdot r_{\underline{j+1}}^i \dotsb r_{\underline{j+k-1}}^i) \cdot
\gamma_{\underline{j+k}, j},
\end{equation}
where by $\underline{a}$ we mean the only element of $\{1,2,\dotsc,n\}$
congruent to $a$ modulo $n$.

We see that the entries of $\gamma$ present in (\ref{eq:prese})
are disjoint for distinct $1\leq k\leq n-1$ (and disjoint from $\{\gamma_{1,1},
\gamma_{2,2},\dotsc,\gamma_{n,n}\}$, which would appear for $k=0$).
Now, for $k$ fixed, (\ref{eq:prese}) gives us a linear form on $n$ variables
$\gamma_{\underline{j+k},j}$ for each $1\leq i\leq n$. Let us
check that, for every $1\leq k\leq n-1$,
 these linear forms are linearly independent,
provided that $\vec{r}$ was chosen correctly.

This is the same as checking that the $n-1$ determinants
\begin{equation}\label{eq:odorem}\left|(r_{j}^i\cdot r_{\underline{j+1}}^i
\dotsb r_{\underline{j+k-1}}^i)\right|_{1\leq i,j\leq n}\end{equation}
for $1\leq k\leq n-1$ are non-zero for some choice of
$r_1,r_2,\dotsc, r_n$ with $r_1\cdot r_2 \dotsb r_n = 1$.
(What we really want to check is that the determinant (\ref{eq:odorem})
is non-zero after all signs in 
some columns are flipped; since those flips do not affect the
absolute value of the determinant, it is just as good to check that
the determinant (\ref{eq:odorem}) itself is non-zero.)
These are Vandermonde determinants, and thus are equal
to \[(-1)^{\lfloor n/2\rfloor} \cdot
\prod_{j_1<j_2} (r_{j_2}\cdot r_{\underline{j_2+1}}
\dotsb r_{\underline{j_2+k-1}} -
r_{j_1}\cdot r_{\underline{j_1+1}}
\dotsb r_{\underline{j_1+k-1}}) .
\]
For any given $k, j_1, j_2$ with $j_1\ne j_2$, there are certainly
$r_1, r_2,\dotsc,r_n\in \overline{K}$ with $r_1 r_2 \dotsb r_n = 1$
such that
$r_{j_1}\cdot r_{\underline{j_1+1}}
\dotsb r_{\underline{j_1+k-1}} \ne
r_{j_2}\cdot r_{\underline{j_2+1}}
\dotsb r_{\underline{j_2+k-1}}$. Thus,
$r_{j_1}\cdot r_{\underline{j_1+1}}
\dotsb r_{\underline{j_1+k-1}} =
r_{j_2}\cdot r_{\underline{j_2+1}}
\dotsb r_{\underline{j_2+k-1}}$ defines a subvariety $W_{k,j_1,j_2}$ of
positive codimension in the (irreducible) variety $V\subset \mathbb{A}^n$ of
all tuples $(r_1,r_2,\dotsb,r_n) \ne 1$ with $r_1 r_2 \dotsb r_n = 1$.
Therefore, $W = \cup_{1\leq k,j_1,j_2\leq n,\; j_1 \ne j_2} W_{k,j_1,j_2}$
is a finite union of subvarieties of $V$ of positive codimension.
Take $\vec{r}$ to be any point of $V(\overline{K})$ outside
$W(\overline{K})$.

It remains to choose
$h_0$
so that the derivative of
\[g\mapsto \kappa(h_0 g)\]
at $g = I$ is a linear map of full rank on the diagonal entries
$\gamma_{1,1},\gamma_{2,2}\dotsc,\gamma_{n-1,n-1}$ of $\mathfrak{g}$.
Let
\begin{equation}\label{eq:babyl}h_0 = \left(\begin{matrix}
s_{1} & 0 & \dotsc & 0\\
0 & s_{2} & \dotsc & 0\\
\vdots & \vdots & \vdots & \vdots\\
0 & 0 & \dotsc & s_{n}\end{matrix}\right),\end{equation}
where $s_1, s_2,\dotsc,s_n \in \overline{K}$ fulfil
$s_1 s_2 \dotsb s_n = 1$.
Then the derivative at $g=I$ of the map taking $g$ to the coefficient
of $\lambda^{n-1}$ in $\det(\lambda I - h_0 g)$ (i.e., to $(-1)$ times
the trace of $h_0 g$) equals the map taking $\gamma$ to
\[(-1) \cdot (
s_1 \gamma_{1,1} + s_2 \gamma_{2,2} + \dotsb + s_n \gamma_{n,n}).\]
In general, the derivative of the map taking $g$ to the coefficient
of $\lambda^{n-k}$ ($1\leq k\leq n-1$) in $\det(\lambda I - h_0 g)$
equals the map taking $\gamma$ to 
\[(-1)^k\cdot (c_{k,1} \gamma_{1,1} + c_{k,2} \gamma_{2,2} + \dotsb 
+ c_{k,n} \gamma_{n,n}),\]
where $c_{k,i}$ is the sum of all monomials
$s_{j_1} s_{j_2} \dotsc s_{j_k}$, $1\leq j_1<j_2<\dotsb<j_k\leq n$,
such that one of the indices $j_l$ equals $i$. (For example,
$c_{2,1} = s_1 \cdot (s_2 + s_3 + \dotsb + s_n)$.) 
Thus, our task is to find
for which $s_1, s_2,\dotsc, s_n$ the determinant
\[|c_{i,j} - c_{i,n}|_{1\leq i,j\leq n-1}\]
is non-zero.
Clearly, this will happen precisely when
\[|c_{i-1,j}|_{1\leq i,j\leq n}\ne 0,\]
where we adopt the (sensible) convention that $c_{0,j} = 1$ for all $j$.

A brief computation gives us that
\[|c_{i-1,j}|_{1\leq i,j\leq n} =
(-1)^{\lfloor n/2\rfloor} \cdot |s_j^{i-1}|_{1\leq i,j\leq n}.\]
This is a Vandermonde determinant; it equals $\prod_{j_1<j_2}
(s_{j_2} - s_{j_1})$. The equation $s_{j_2} = s_{j_1}$ defines a subvariety
of positive codimension in the variety $V\subset \mathbb{A}^n$ of all
$s_1,s_2,\dotsc,s_n$ with $s_1 s_2 \dotsb s_n = 1$. Thus, we may choose
$s_1,s_2,\dotsc,s_n$ such that $s_1 s_2 \dotsb s_n = 1$ and
$\prod_{j_1 < j_2} (s_{j_2} - s_{j_1}) \ne 0$.
\end{proof}

The proposition below can be applied with $W$ empty. 
We will later need to invoke
it with $W$ equal to the variety of elements of $G$ that
are not regular semisimple.
\begin{prop}\label{prop:orodor}
Let $G = \SL_n$. Let $K$ be a field. Let $X = G^{n+1}$,
$Y = \mathbb{A}^{(n-1) \cdot (n+1)} =
\mathbb{A}^{n^2 - 1}$. Let $f:X\times G\to Y$ be the map given by
\[
f((h_0,h_1,\dotsc,h_n),g) = (\kappa(h_0 g), \kappa(h_1 g), \dotsc,
\kappa(h_n g)).\]
 Let $W$ be a proper subvariety of $G$ (which may be
empty).
Let $A\subset G(K)$ be a set of generators of $G(K)$.

Then
there are elements $h_0,h_1,\dotsc,h_n \in A_k$, $k\ll_{n,\vdeg(W)} 1$, such that
\[|f((h_0,h_1,\dotsc,h_n),A_{k}\setminus (A_{k} \cap W(K)))| \gg_{n,\vdeg(W)} |A|.\]
\end{prop}
\begin{proof}
Let $Z_{X\times G}$ be as in Lem.\ \ref{lem:remor}.
By Lemma \ref{lem:gole}, at least one point of $(X\times G)(\overline{K})$
lies outside $Z_{X\times G}$; thus $Z_{X\times G}$ is a proper subvariety
of $X\times G$. By the argument in \S \ref{subs:fibcou}, 
the points $x_0$ on $X$ such that $(Z_{X\times G})_{x=x_0}$
is all of $G$ lie on a proper subvariety $Z_X$ of $X$ of degree
$\ll_{\vdeg(Z_{X\times G})} 1$ (and so, by (\ref{eq:coroco}), 
$\vdeg(Z_X)\ll_n 1$). By Lem.\ \ref{lem:utilo}, 
there are points of $X(G)$ outside
$Z_X$, provided that we assume
 that $|K|$ is greater than a constant depending only on $n$.
(If $|K|\ll_n 1$, what we seek to prove is trivially true.) We
can then use escape from groups (Lem.\ \ref{lem:lemfac}) to the group
$X = G^{n-1}$ and the set of generators $E=A\times A\times \dotsb \times A$
of $X$, and obtain that there is a tuple
$\vec{h} = (h_0,h_1,\dotsc,h_n) \in E_k$, $k\ll_{n} 1$, such that
$\vec{h}$ lies on $X\setminus Z_X$.

Define $V = (Z_{X\times G})_{x = \vec{h}} \cup W$. Again by 
Lem.\ \ref{lem:utilo}, 
there are points of $G(K)$ outside $V$ (assuming again, as we may, that
$|K|$ is greater than a constant depending only on $n$). 
We can then use a general result on the consequences of being non-singular
almost everywhere, 
namely, Cor.\ \ref{cor:gotrol} (with $E=A$) and obtain that
\[|f_{\vec{h}}(A_{k'}\cap (G(K)\setminus V(K)))|\gg_{\vdeg(G), \vdeg(V), \deg_{\pol}(f_{\vec{h}})
} |A|,\]
and so (since $\vdeg(G)\ll_n 1$, $\deg_{\pol}(f_{\vec{h}})\ll_n 1$,
$\vdeg(V) \ll_{n,\vdeg(W)} 1$ and $W\subset V$)
\[|f_{\vec{h}}(A_{k'}\cap (G(K)\setminus W(K)))|\gg_{n, \vdeg(W)} |A|,\]
where $k'\ll_{n,\vdeg(W)} 1$.
\end{proof}

Recall that $\Cl_G(A)$
denotes the set of all conjugacy classes in $G$ that
contain at least one element of $A$.
\begin{cor}\label{cor:london}
Let $G = \SL_n$. Let $K$ be a field.
Let $A\subset G(K)$ be a set of generators of $G(K)$.

Let $W$ be a (possibly empty) proper subvariety of $G$. Then 
\[|\Cl_G(A_k\setminus (A_k \cap W(K)))| \gg_{n,\vdeg(W)} |A|^{\frac{1}{n+1}},
\]
where $k\ll_{n,\vdeg(W)} 1$. In particular,
\[|\Cl_G(A_k)| \gg_n |A|^{\frac{1}{n+1}},\]
where $k\ll_n 1$.
\end{cor}
Here $\frac{1}{n+1} = \frac{n-1}{n^2 -1}$
is the exponent one would expect for $\SL_n$: the variety $G = \SL_n$ is
 of dimension $n^2-1$, and the characteristic polynomial of a matrix
has $n-1$ coefficients other than the leading and the constant terms, which
are identically $1$.
\begin{proof}
From Prop.\ \ref{prop:orodor},
we have that there are at least $\gg_{n,\vdeg(W)} |A|$ distinct $(n+1)$-tuples
$(\kappa(g_0), \kappa(g_1), \dotsc, \kappa(g_n))$, where
$g_j = h_j g$ is an element of $A_{2 k}$, $k\ll_{n, \vdeg(W)} 1$.
Clearly, this implies that there are at least $\gg_{n,\vdeg(W)} |A|^{1/(n+1)}$
distinct elements $\kappa(g)$, $g\in A_{2 k}$.
Two matrices $g$, $g'$ in distinct conjugacy classes in $G$
cannot have the same characteristic polynomial.
The statement now follows immediately.
\end{proof}

\subsection{From conjugacy classes to a maximal torus}\label{subs:richmat}
The following lemma uses nothing, and yet the rest of the section spins
around it.
\begin{prop}\label{prop:ostrogoth}
Let $G$ be a group. Let $A, A' \subset G$. Then there is a $g\in A'$
such that
\[|C_G(g) \cap A^{-1} A| \geq \frac{|A|}{|A A' A^{-1}|} \cdot
|\Cl_G(A')|.\]
\end{prop}
\begin{proof}
Write $c_g$ for the number of elements of $A^{-1} A$ commuting with a given
$g\in G$. For every $g$,
\[|\{h g h^{-1}: h\in A\}| \geq \frac{|A|}{c_g} .\]
(Otherwise there would be a $h_0\in A$ such that $h_0 g h_0^{-1} =
h g h^{-1}$ for more than $c_g$ elements $h$ of $A$ -- and, since
$h_0 g h_0^{-1} = h g h^{-1}$ implies that $h^{-1} h_0 \in A^{-1} A$ commutes
with $g$, we would have a contradiction.)
At the same time, for $g_1$, $g_2$ in different conjugacy classes,
\[\{h g_1 h^{-1} : h \in A\} \text{\;\;\;\;\;\; and \;\;\;\;\;\;}
\{h g_2 h^{-1} : h \in A\}\]
are disjoint.

Hence
\[|\{h g h^{-1} : h\in A, g\in A'\}| \geq \sum_g \frac{|A|}{c_g},
\]
where the sum is over representatives $g\in A'$ of conjugacy classes
intersecting $A'$. Therefore, there is a $g\in A'$ such that
\[\frac{|A|}{c_g} \leq \frac{1}{|\Cl_G(A')|} \cdot
 |\{h g h^{-1} : h\in A, g\in A'\}|,
\]
and so
\[c_g \geq \frac{|A|}{ |\{h g h^{-1} : h\in A, g\in A'\}|} \cdot
|\Cl_G(A')| \geq
\frac{|A|}{|A A' A^{-1}|} \cdot |\Cl_G(A')|.
\]
\end{proof}
\begin{Rem}
It should be clear from the proof that $c_g = |C_G(g) \cap A^{-1} A|$
is large not just for one $g\in A'$, but for many $g\in A'$. We shall
not need this fact.
\end{Rem}

We say that an element of an algebraic group is {\em regular semisimple}
if the connected component of its centraliser that contains the identity
is a maximal 
torus. In $\SL_n$, a regular semisimple element is simply an element
with distinct eigenvalues; its centraliser is always connected, and thus
equals a maximal torus.

The following is a special case of a much more
general statement.
\begin{lem}\label{lem:showe}
Let $G = \SL_n$. Let $K$ be a field.
Then there is a subvariety $W/K$ of $G$ of positive codimension
and degree $\vdeg(W) \ll_n 1$ such that every element $g\in G(K)$ not on
$W$ is regular semisimple.
\end{lem}
In fact, the variety $W$ will be defined over $\mathbb{Z}$, independently of
$K$; we just
need to check that it has positive codimension over $K$, i.e., that
there are points in $G(\overline{K})\setminus W(\overline{K})$.
\begin{proof}
An element of $\SL_n$ is regular semisimple if (and only if) its eigenvalues
are distinct. Let $W$ be the variety of all $g\in G$ whose characteristic
polynomials have multiple roots, i.e., define $W$ by
$\disc(\det(\lambda I - g)) = 0$. As we can easily find points in
$G(\overline{K})\setminus W(\overline{K})$ (say, diagonal elements
with distinct entries), we are done.
\end{proof}
Since $W$ is a subvariety of $\SL_n$ of positive codimension, we may
escape from it.
\begin{cor}[to Prop.\ \ref{prop:ostrogoth}]\label{cor:dophus}
Let $G = \SL_n$. Let $K$ be a field. Let $A\subset G(K)$ be a set of
generators of $G(K)$.

Then there is a maximal torus $T/\overline{K}$ of $G$ such that
\begin{equation}\label{eq:vawi}
|A_k \cap T(K)| \gg_n \frac{|A|}{|A_{k+2}|} \cdot |A|^{\frac{1}{n+1}} ,
\end{equation}
where $k\ll_n 1$.
\end{cor}
If $|A\cdot A \cdot A| \ll |A|^{1+\epsilon}$, then (by the
tripling lemma, viz., Lemma \ref{lem:furcht}) the inequality (\ref{eq:vawi}) reads:
$|A_k \cap T(K)| \gg |A|^{\frac{1}{n+1} - O_k(\epsilon)}$.
\begin{proof}
Let $W$ be as in Lemma \ref{lem:showe}.
By Corollary
\ref{cor:london},
\[
\Cl_{G(K)}(A') \gg |A|^{\frac{1}{n+1}} ,
\]
where $A' = A_k \setminus (A_k \cap W(K))$.
At the same time, by Proposition \ref{prop:ostrogoth},
\[
|C_{G(K)}(g) \cap A^{-1} A| \geq \frac{|A|}{|A A' A^{-1}|}
\cdot |\Cl_{G(K)}(A')|
\]
for some $g\in A'$.

By the definition of $W$,
all elements of $A'$ are regular semisimple; in other words,
the centraliser $C_{G(K)}(g)$ lies on a maximal torus. Hence
\[|T(K) \cap A^{-1} A| \gg_n \frac{|A|}{|A A' A^{-1}|} \cdot |A|^{\frac{1}{n+1}}
\geq \frac{|A|}{|A_{k+2}|} \cdot |A|^{\frac{1}{n+1}},\]
where $T/\overline{K}$ is any maximal torus containing $C_{G(K)}(g)$.
\end{proof}

\subsection{An upper bound on the number of conjugacy classes.}

Consider a set $A\subset \SL_n(K)$ such that $|A\cdot A \cdot A|\ll |A|^{1 + \epsilon}$.
Using Prop.\ \ref{prop:ostrogoth} and the fact that there are not
too few conjugacy classes, we have just shown that there
is a torus $T$ such that there are not too
few elements on $T$. Using, again, Prop.\ \ref{prop:ostrogoth} and the
fact that there are not too many elements on $T$, we shall now show
that there are not too many conjugacy classes.

\begin{cor}[to Cor.\ \ref{cor:destin} and Prop.\ \ref{prop:ostrogoth}]\label{cor:basu}
Let $G = \SL_n$. Let $K$ be a field. Let $A\subset G(K)$ be any set of
generators of $G(K)$. Assume that $|A|$ is greater than a constant
depending on $n$.

Then
\[|\Cl_G(A \cap \Sigma(K))| \ll_n \frac{|A A A^{-1}|}{|A|} |A_k|^{\frac{1}{n+1}},\]
where $k\ll_n 1$ and $\Sigma$ is the Zariski-open set of regular semisimple elements of
$G$.
\end{cor}
\begin{proof}
Let $A' = A \cap \Sigma(K)$.
By Prop.\ \ref{prop:ostrogoth}, there is a $g\in A'$ such that
\[|C_G(g) \cap A^{-1} A| \geq \frac{|A|}{|A A' A^{-1}|} \cdot
|\Cl_G(A')|.\]
Since $g$ is regular semisimple, its centraliser $T=C_G(g)$ is 
a maximal torus. By Cor.\ \ref{cor:destin} (applied to $A^{-1} A$ rather
than $A$),
\[|T(K) \cap A^{-1} A| \ll_n |A_s|^{\frac{\dim(T)}{\dim(G)}} =
|A_s|^{\frac{1}{n+1}},\]
where $s\ll_n 1$. Thus
\[|\Cl_G(A')| \ll_n \frac{|A A' A^{-1}|}{|A|} |A_s|^{\frac{1}{n+1}} .\]
\end{proof}

In brief: we already knew that, for any set of generators $A$,
\[\begin{aligned}
&|A\cap T(K)| \ll |A_k|^{\frac{1}{n+1}}\;\;\;\;\;\text{and}\\
|A|^{\frac{1}{n+1}} \;\ll\; &|\Cl(A_k)|.
\end{aligned}\]
We now know that, if $A$ does not grow (i.e., $|A\cdot A\cdot A|\ll
|A|^{1 + \epsilon}$) then the inequalities can be reversed:
\begin{equation}\label{eq:colote}\begin{aligned}
|A_k|^{\frac{1}{n+1} - O(\epsilon)} \;\ll\; &|A_k\cap T(K)|\;\;\;\;\;\;\;
\text{and}\\
&|\Cl(A')| \ll |A_k|^{\frac{1}{n+1} + O(\epsilon)},
\end{aligned}\end{equation}
where $A'$ is the set of regular semisimple elements of $A$.

Our plan in \S \ref{sec:armon} will be to derive a contradiction
from this tight
situation. We shall eventually construct what may be seen as a
counterexample to an incidence theorem: the elements of the torus shall give
us the lines (i.e., the linear relations), and the conjugacy classes shall give us the points. There will
be too many lines with many points on each and too few points in total.

Before we finish this section, we must do some auxiliary work on
intersections with non-maximal
tori. 
\subsection{Intersections with non-maximal tori}\label{subs:indness}

We already know that we can find a torus $T$ such that
$|A\cap T(K)|$ is large; we will now show that, for any $T$ and
for any subtorus $T'\subset T$ given as the kernel of a character of $T$,
the intersection $|A\cap T'(K)|$ is small.

(A {\em character} of a maximal torus is a homomorphism from $T(\overline{K})$ to $\overline{K}^*$
given as an algebraic map defined over $\overline{K}$. If $G=\SL_n$ and
$T$ is a maximal torus of $G$ given as the group of diagonal matrices, then
the characters are the maps of the form $t\mapsto \prod_{1\leq j\leq n}
t_{jj}^{m_j}$, where $m_j\in \mathbb{Z}$ and $\sum_j m_j = 0$. There is
an analogous notion of {\em character} for Lie algebras; the characters of
a Cartan subalgebra of $\sL_n$ (seen as the algebra of diagonal matrices)
are maps of the form $t\mapsto \sum_j m_j t_{jj}$, where $m_j\in \mathbb{Z}$ and
$\sum_j m_j = 0$.)

\begin{lem}\label{lem:balader}
Let $\mathfrak{g} = \sL_n$ be defined over a field $K$. Assume $\charac(K)\nmid n$.
Let $\mathfrak{t}$ be a Cartan subalgebra
of $\mathfrak{g}$.
Let $\mathfrak{t}'$ be the kernel of a non-trivial character
$\alpha:\mathfrak{t}\to \mathbb{A}^1$.

Then there are elements $\vec{g}_0,\vec{g}_1,\dotsc,\vec{g}_n\in 
\mathfrak{g}(K)$ such that the
spaces
\begin{equation}\label{eq:finzo}
\mathfrak{t}', \lbrack \vec{g}_0,\mathfrak{t}'\rbrack, \lbrack
\vec{g}_1,\mathfrak{t}'\rbrack, \lbrack \vec{g}_2,\mathfrak{t}'\rbrack, \dotsc,
\lbrack \vec{g}_n,\mathfrak{t}'\rbrack
\end{equation}
are linearly independent and of dimension $\dim(\mathfrak{t}')$.
\end{lem}
Compare this to Lemma \ref{lem:ground}, where a result that looks much the
same holds for $\mathfrak{t}$, $\lbrack \vec{g}_1, \mathfrak{t}\rbrack$,
\dots, $\lbrack \vec{g}_n, \mathfrak{t}\rbrack$, i.e., for one space
fewer than in (\ref{eq:finzo}). This discrepancy is what we shall
use soon (Cor.\ \ref{cor:cadil}) 
in order to show that the elements in the intersection of
a non-growing set $A$ and a maximal torus $T$ cannot be concentrated on
a proper subtorus $T'$.
\begin{proof}
Write the elements of $\mathfrak{g}$ as matrices so that
$\mathfrak{t}$ becomes the algebra of diagonal matrices with trace $0$.
Let $e_{i,j}$ be the matrix having a $1$ at the $(i,j)$th entry and $0$s
at all other entries. We define
\[\vec{g}_j = \mathop{\sum_i}_{i\ne j} e_{i,j}\]
for $1\leq j\leq n$, where the sum goes through all $i$ from $1$ to $n$ other
than $j$. For every $t\in \mathfrak{t}$, the matrix
$\lbrack \vec{g}_j,t\rbrack$ has $0$s at the $(j,j)$th entry and throughout all
columns save for the $j$th column. In fact, for any $g\in \mathfrak{g}(K)$
and any $t\in \mathfrak{t}(K)$, the matrix $\lbrack g,\mathfrak{t}\rbrack$
has $0$s throughout the diagonal.

Thus, it remains only to find a $\vec{g}_0\in \mathfrak{g}(K)$ not in $\mathfrak{t}$ such that
the linear space $V_0 = \lbrack \vec{g}_0,\mathfrak{t}'\rbrack$ and
the linear space \[V = \lbrack \vec{g}_1,\mathfrak{t}'\rbrack +
\lbrack \vec{g}_2,\mathfrak{t}'\rbrack + \dotsc + \lbrack \vec{g}_n,
\mathfrak{t}'\rbrack\] intersect only at the origin. (We can already
see that each space $\lbrack \vec{g}_i,\mathfrak{t}' \rbrack$, $1\leq i\leq n$,
intersects the sum of all the others only at the origin, and that
the space $\mathfrak{t}$ intersects the sum $V$ of all of them
only at the origin. Since $\vec{g}_0$ will not be in $\mathfrak{t}$, 
$V_0 = \lbrack \vec{g}_0,\mathfrak{t}'\rbrack$ will have the same dimension as $\mathfrak{t}$.)

For $1\leq i_0\leq n$,
let $s(i_0)$ be the matrix in $\mathfrak{t}$ having $(s(i_0))_{i i} = -1$ for
all $i\ne i_0$ and $(s(i_0))_{i_0 i_0} = n-1$. If
$\charac(K)\nmid n$, such matrices span $\mathfrak{t}(K)$
as a linear space; since $\mathfrak{t}'\ne \mathfrak{t}$, there must
be at least one such matrix $s(i_0)$ not in $\mathfrak{t}'(K)$. 
Fix that $i_0$ from now on.
Because $s(i_0)\notin \mathfrak{t}'(K)$, there
is no non-zero matrix $t'$ in $\mathfrak{t}'$ with all of its diagonal entries
other than $t'_{i_0,i_0}$ equal to each other. 
 Again by $\charac(K)\nmid n$,
there is also no non-zero matrix $t'$ in $\mathfrak{t}'$ with all of
its diagonal entries equal to each other.

Now define
\begin{equation}\label{eq:goyot}
\vec{g}_0 = \mathop{\sum_j}_{j\ne i_0} e_{i_0,j} .\end{equation}
Suppose there is a $t\in \mathfrak{t}$ such that $\lbrack \vec{g}_0,t
\rbrack \in V$. Then, for every $j$ between $1$ and $n$, the $j$th column
of the matrix $\lbrack \vec{g}_0,t\rbrack$ equals the $j$th column of
the matrix $\lbrack \vec{g}_j,t_j'\rbrack$ for some $t_j' \in \mathfrak{t}$.

Let $1\leq j\leq n$. If $j=i_0$, then, as can be computed easily from
 the definition (\ref{eq:goyot}) of $\vec{g}_0$,
the $j$th column of $\vec{g}_0$ has all of its entries equal to $0$.
Suppose $j\ne i_0$. Then the $j$th column of
$\lbrack \vec{g}_0,t\rbrack$ has all of its entries equal to $0$ save
for the $(i_0)$th entry, which is equal to $t_{j,j} - t_{i_0,i_0}$.
The $j$th column of $\lbrack \vec{g}_j,t_j'\rbrack$ is
\[\left(\begin{matrix} t_{j j}' - t_{1 1}'\\t_{j j}' - t_{2 2}'\\ \vdots\\
t_{j j}' - t_{n n}'\end{matrix}\right).\] If these two $j$th columns
were equal, then all diagonal entries of $t'$ (save possibly for
$t'_{j_0,j_0}$) are equal to each other. As we have seen, this implies
that $t'= 0$. Hence the $j$th column of 
$\lbrack \vec{g}_0,t\rbrack$ has all of its entries equal to $0$.

We let $j$ vary between $1$ and $n$ and obtain that, in every column
of $\lbrack \vec{g}_0,t\rbrack$, all of the entries are equal to $0$;
in other words, $\lbrack \vec{g}_0,t\rbrack = 0$. We conclude that the intersection
of $\lbrack \vec{g}_0,\mathfrak{t}\rbrack$ and $V$ is $\{0\}$.
\end{proof}

\begin{prop}\label{prop:denfert}
Let $G = \SL_n$. Let $K$ be a field. Assume that $K$ is either finite with
$\charac(K)\nmid n$ or infinite with $\charac(K)=0$. 
Let $T$ be
a maximal torus of $G$ defined over $\overline{K}$.
Let $\alpha:T\to \mathbb{A}^1$ be a character of $T$, and let
 $T'$ be the kernel of $\alpha$.

For every $(g_0,g_1,\dotsc,g_n)
\in (G(K))^{n+1}$, let $f_{g_0,g_1,\dotsc,g_n}:(T')^{n+2} \to G$ be the map
defined by \begin{equation}\label{eq:roto}
f_{g_0,g_1,\dotsc,g_n}(t,t_0, t_1,\dotsc,t_n) =
t \cdot g_0 t_0 g_0^{-1} \cdot g_1 t_1 g_1^{-1} \dotsb g_n t_n g_n^{-1} .
\end{equation}

Let $A\subset G(K)$ be a set of generators of $G(K)$, and let $E$ be
a non-empty subset of $T'(K)$.
Then there are $g_0,g_1,\dotsc,g_n\in A_k$, $k\ll_n 1$, such that
\[
|f_{g_0,g_1,\dotsc,g_n}(E,E,\dotsc,E)| \gg_{n,\deg(T')} |E|^{n+2}.
\]
\end{prop}
\begin{proof}
We may assume that the derivative of $\alpha$ does not vanish at the origin:
if it does, then the characteristic $\charac(K)$ is equal to 
$p$ for some prime $p$, and $\alpha = \beta^p$ for some character
$\beta:T\to \mathbb{A}^1$; since the Frobenius map $x\to x^p$ 
is an automorphism for $K$ finite, it follows that $\ker(\alpha) =
\ker(\beta)$, and so we can use $\beta$ instead of $\alpha$. (Repeat
if needed.)

The Lie algebra $\mathfrak{t}'$ of $T'$ lies in the kernel of the derivative
$\alpha_0$ of $\alpha$ at the origin, which is a character of the Lie algebra
$\mathfrak{t}$ of $T$; as we have just said, $\alpha_0$ is not identically
zero. We may thus apply Lemma \ref{lem:balader}; it asserts that
the assumptions of Prop.\ \ref{prop:elysium}
are fulfilled (with $\ell = n+1$ and $\mathfrak{h} = \mathfrak{t}'$). 
The conclusions of Prop.\ \ref{prop:elysium} provide the
linear-independence assumption of Prop.\ \ref{prop:galoshes}
(for $H_0= H_1 = \dotsc = H_{\ell} = T'$); we apply
Prop.\ \ref{prop:galoshes}, and are done. 
\end{proof}
\begin{cor}\label{cor:cadil}
Let $G = \SL_n$. Let $K$ be a field. Assume that $K$ is either finite with
$\charac(K)\nmid n$ or infinite with $\charac(K)=0$. 
Let $T$ be
a maximal torus of $G$ defined over $\overline{K}$.
Let $\alpha:T\to \mathbb{A}^1$ be a character of $T$, and let
 $T'$ be the kernel of $\alpha$.

Let $A\subset G(K)$ be a set of generators of $G(K)$. Then
\begin{equation}\label{eq:fabius}
|A \cap T'(K)| \ll_{n,\deg(T')} |A_{k}|^{\frac{1}{n+2}},
\end{equation}
where $k\ll_n 1$.
\end{cor}
\begin{proof}
Immediate from Prop.\ \ref{prop:denfert}
 and the definition of $f_{g_0,g_1,\dotsc,g_n}$.
\end{proof}

Since we already know from Cor.\ \ref{cor:dophus} that
$|A_r \cap T(K)| \gg |A_{r k}|^{\frac{1}{n+1} - O(\epsilon)}$
for some $r\ll_n 1$
(assuming that $|A\cdot A\cdot A| \ll |A|^{1 + \epsilon}$),
the inequality (\ref{eq:fabius}) (applied to $A_r$ instead of $A$)
implies that only a very small
fraction of the elements of $A_r \cap T(K)$ lie in the kernel $T'$
of a given character $\alpha$.

\subsection{Special tuples of coefficients of characteristic polynomials}

In \S \ref{sec:armon}, we will need to work with tuples of the form
\[(\kappa(h_0 g),\kappa(t g), \kappa(t^2 g),\dotsc,
\kappa(t^n g)).\]
We need to show that there are many such tuples. 
Corollary \ref{cor:cadil} will make a crucial appearance towards the end.

\begin{prop}\label{prop:warsaw}
Let $G = \SL_n$. Let $K$ be a finite field.
Let $W/\overline{K}$ be a proper subvariety of $G$.
Let $T/\overline{K}$ be a maximal torus of $G$.
 Let $A\subset G(K)$
be a set of generators of $G(K)$. Let $E\subset T(K)$. Then, provided
that $|K|$ is larger than a constant depending only on $n$, either
\begin{enumerate}
\item\label{it:goro} there is an element $h_0\in A_k$, $k\ll_n 1$,
 and a subset
$E'\subset E_k$ with $|E'|\gg_n |E|$ such that, for each $t\in E'$,
there are $\gg_{n,\vdeg(W)} |A|$ distinct tuples
\[(\kappa(h_0 g), \kappa(t g), \kappa(t^2 g),\dotsc, \kappa(t^n g))
\in \mathbb{A}^{n^2-1}(K)\]
with $g\in A_{k'}$, $k'\ll_{n,\vdeg(W)} 1$ satisfying $h_0 g\notin W(K), 
g, t g, t^2 g, \dotsc, t^n g \notin W(K)$, or
\item\label{it:gara} $E$ is contained in the kernel of a non-trivial character
$\alpha:T\to \mathbb{A}^1$ whose exponents are  bounded in terms of $n$ alone.
\end{enumerate}
\end{prop}
\begin{proof}
Let $X = G\times T$ and $Y = (\mathbb{A}^{n-1})^{n+1} = 
\mathbb{A}^{n^2-1}$. Let $f:X\times G\to Y$
be given by
\[f((h,t),g) = (\kappa(h g),\kappa(t g), \kappa(t^2 g),\dotsc,
\kappa(t^n g)).\]
Let $Z_{X\times G}$ be as in Lemma \ref{lem:remor}; by (\ref{eq:coroco}),
$\vdeg(Z_{X\times G})\ll_n 1$. Thanks to Lem.\ \ref{lem:gole}, we know
$Z_{X\times G}$ is a proper subvariety of $X\times G$. 

Let $Z_{G\times T\times G}$ be $Z_{X\times G}$ under the identification
$G\times T\times G = X\times G$; write the elements of $Z_{G\times T\times G}$
in the form $(h,t,g)$. By the argument in \S \ref{subs:fibcou},
there is a proper subvariety $Z_G\subset G$ (with $\vdeg(Z_G)\ll_{\vdeg(Z_{G\times T\times G})} 1$, and so $\vdeg(Z_G)\ll_n 1$) such that, 
for all $h_0\in G(\overline{K})$ not on $Z_G$, 
the fibre $(Z_{G\times T\times G})_{h=h_0}$ is a proper subvariety of
 $T\times G$.

By escape from groups (Lem.\ \ref{lem:lemfac} and Lem.\ \ref{lem:utilo};
it is here that that $|K|\gg_n 1$ is used),
there is an $h_0\in A_k$, $k\ll_n 1$, such that $h_0$ lies outside $Z_G$;
thus, by the definition of $Z_G$, the fibre
$V_{T\times G}:=(Z_{G\times T\times G})_{h=h_0}$ is a proper subvariety of
 $T\times G$. Again by \S \ref{subs:fibcou}, there is a proper
subvariety $V_T$ with $\vdeg(V_T)\ll_{\vdeg(V_{T\times G})} 1$ (and so
$\vdeg(V_T)\ll_n 1$)
 such that, for all $t_0\in T(\overline{K})$ not on $V_T$,
the fibre $(V_{T\times G})_{t = t_0}$ is a proper subvariety of $G$.

Suppose first that $\langle E\rangle\not\subset V_T(K)$. We may then
use escape from subvarieties (Prop.\ \ref{prop:carbo} with
$A = E$, $V = V_T(K)$ and $G = \mathscr{O} = \langle E\rangle$) to
obtain a subset $E'\subset E_k$ ($k\ll_n 1$) with
$|E'|\gg_n |E|$ and $E'\subset T(K)\setminus V_T(K)$. Now consider any
 $t_0\in E'$. 
The fibre $(V_{T\times G})_{t = t_0}$ is a proper subvariety of $G$,
and, since $W$ is a proper subvariety of $G$, we conclude that
\[V' = (V_{T\times G})_{t = t_0} \cup h_0^{-1} W \cup W \cup t_0^{-1} W
\cup \dotsc \cup t_0^{-n} W.\]
is a proper subvariety of $G$ as well (with $\vdeg(V')\ll_{n,\vdeg(W)} 1$).
We now recall the definition of $Z_{X\times G}$ (a variety outside
which the map $f$ is non-singular) and
 use the result on non-singularity
(Corollary \ref{cor:gotrol} applied to the function 
$f_{h_0,t_0}:G\to Y$ given by
$f_{h_0,t_0}(g)=f((h_0,t_0),g)$; here Lem.\ \ref{lem:utilo} supplies the
condition $V(K)\subsetneq G(K)$, which is a requirement for the
application of Cor.\ \ref{cor:gotrol}) to obtain that
\[|f_{h_0,t_0}(A_{k'} \cap (G(K)\setminus V'(K)))| \gg_{n,\vdeg(W)} |A|\]
with $k'\ll_{n,\vdeg(W)} 1$. This gives us conclusion (\ref{it:goro}).

Suppose now that $\langle E\rangle\subset V_T(K)$. Then, by Prop.\
\ref{prop:kartar}, $\langle E\rangle$ is contained in an algebraic
subgroup $H$ of $T$ of positive codimension and 
degree $\vdeg(H)\ll_{\vdeg(V_T(K))} 1$ (and so $\vdeg(H)\ll_n 1$).
By Lemma \ref{lem:alin}, we obtain that $H$ is contained in the kernel
of a non-trivial character $\alpha:T\to \mathbb{A}^1$ whose exponents are
$\ll_n 1$.
\end{proof}

As before, we write $\Sigma$ for the (algebraic) set of regular semisimple
elements of $G$; in the case of $G = \SL_n$, this is simply the
(algebraic) set consisting of every $g$ whose eigenvalues are all distinct.
The sets of points $\Sigma(K)$ and $\Sigma(\overline{K})$ are what
one would expect, viz., the sets consisting of the
 elements of $G(K)$ and $G(\overline{K})$
having distinct eigenvalues. For $G = \SL_n$, the complement of $\Sigma$ is a variety $W$ with
$\vdeg(W) \ll_n 1$.

\begin{cor}\label{cor:nipon}
Let $G = \SL_n$. Let $K$ be a field. 
Assume that $K$ is either finite with
$\charac(K)\nmid n$ or infinite with $\charac(K)=0$. 
Let $T/\overline{K}$ be a maximal torus
of $G$.
Let $A\subset G(K)$
be a set of generators of $G(K)$. Suppose that $|A|$ is
 greater than a constant depending only on $n$.

Then there is an $\epsilon_0$ depending only on $n$ such that, if
$|A\cdot A\cdot A|\leq |A|^{1 + \epsilon}$ for some positive
$\epsilon< \epsilon_0$, then
there is an element $h_0\in A_k$, $k\ll_n 1$, and a subset
$E'\subset A_k \cap T(K)$ with $|E'|\gg_n |A_k\cap T(K)|$ such that, for each $t\in E'$,
there are $\gg_n |A|$ distinct tuples
\[(\kappa(h_0 g), \kappa(t g), \kappa(t^2 g),\dotsc, \kappa(t^n g))
\in \mathbb{A}^{n^2-1}(K)\]
with $g\in A_k$ satisfying $h_0 g\in \Sigma(K)$ and
$t^{\ell} g \in \Sigma(K)$ for $\ell = 0,1,2,\dotsc,n$, where
$k$ and the implied constants depend only on $n$.
\end{cor}
\begin{proof}
By Corollary \ref{cor:dophus}
and the tripling lemma (Lem.\ \ref{lem:furcht}),
\begin{equation}\label{eq:dodorn}
|A_k \cap T(K)| \gg_n |A|^{\frac{1}{n+1} - O_n(\epsilon)},
\end{equation}
where $k \ll_n 1$. Let $E = A_k
\cap T(K)$; let $W$ be the complement of $\Sigma$. Apply Prop.\ \ref{prop:warsaw}.
If case (\ref{it:goro}) of Prop.\ \ref{prop:warsaw} applies, we are done.

 It remains only to rule out
case (\ref{it:gara}) of Prop.\ \ref{prop:warsaw}.
Suppose $E$ is contained in the kernel $T'$ of a non-trivial character
$\alpha:T\to \mathbb{A}^1$ whose exponents are 
$\ll_n 1$. Then $\deg(T')\ll_n 1$.
We now apply Cor.\ \ref{cor:cadil}
(to $A_k$ rather than $A$), and obtain that
that \[|E| \ll_n |A_{k k'}|^{\frac{1}{n+2}} \ll_n |A|^{\frac{1}{n+2} +
  O_n(\epsilon)}\]
for some $k'\ll_n 1$,
in contradiction to (\ref{eq:dodorn}). (Recall that $E = A_k \cap T(K)$.)
\end{proof} 

\section{Growth of small and large sets in $\SL_2$ and $\SL_3$}\label{sec:armon}
For the sake of clarity and completeness, we shall do things twice: once
for $\SL_2$ and once for $\SL_3$. Of course, in the case of $\SL_2$,
we could refer to \cite{He} instead; since, however, the method in this paper
is somewhat different -- especially in this part of the argument -- 
we would like to work things 
out for both $\SL_2$ and $\SL_3$.

The key observation in the proofs below is the following. Consider $n+1$
diagonal matrices $t_0,t_1,\dotsc,t_n\in A_k$. The maps
$g\mapsto \tr(t_0 g)$, $g\mapsto \tr(t_1 g)$, \dots, 
$g\mapsto \tr(t_n g)$ from $\SL_n(K)$ to $K$ can be seen as linear forms
-- that is, homogeneous linear polynomials --
on the $n$ variables $g_{1,1}, \dotsc, g_{n,n}$ (the diagonal entries of $g$).

Any $n+1$ linear forms on $n$ variables must be linearly dependent. Hence
there are coefficients $c_0,c_1,\dotsc,c_n\in K^n$ depending on
$t_1,t_2,\dotsc,t_n$ (but not on $g$) such that
\begin{equation}\label{eq:gorto}
c_0 \tr(t_0 g) + c_1 \tr(t_1 g) + \dotsb + c_n \tr(t_n g) = 0
\end{equation}
for all $g$. Thus we have a linear relation holding for many tuples
(namely, the tuples $(\tr(t_0 g), \tr(t_1 g), \dotsc, \tr(t_n g))$ for any $g\in A$)
all of whose entries $\tr(t_j g)$ lie in a small set (viz., $\tr(A_{k+1})$). 

As $t_0,t_1,\dotsc,t_n$ vary, the coefficients $c_0,c_1,\dotsc,c_n$ will
vary as well. We will obtain too many linear relations (of the
form (\ref{eq:gorto})), and thus a contradiction to Corollary \ref{cor:espada}.

\subsection{Small sets in $\SL_2$}\label{sec:otrogo}
The treatment of $\SL_2$ in \cite{He} was based on the identity
\begin{equation}\label{eq:estre}
(x + x^{-1}) (y + y^{-1}) = (x y + (x y)^{-1}) + (x y^{-1} + (x y^{-1})^{-1}),
\end{equation}
which is a special case of the identity
\begin{equation}\label{eq:llanto}
\tr(g) \tr(h) = \tr(g h) + \tr(g h^{-1})
\end{equation}
valid in $\SL_2$ (but not in $\SL_n$, $n>2$). We shall now do without
(\ref{eq:estre}) and (\ref{eq:llanto}).

\begin{prop}\label{prop:baggage}
Let $G = \SL_2(\mathbb{Z}/p\mathbb{Z})$, $p$ a prime. Let $A\subset G$ be
a set of generators of $G$. Assume $|A| < p^{3 - \delta}$, $\delta>0$. Then
\begin{equation}\label{eq:consuf}
|A\cdot A \cdot A|\gg_{\delta} |A|^{1 + \epsilon},\end{equation}
where $\epsilon>0$ depends only on $\delta$.
\end{prop}
This is part (a) of the Key Proposition in \cite{He}.
\begin{proof}
Suppose $|A\cdot A\cdot A|\leq |A|^{1 + \epsilon}$. Then, by the tripling
lemma
(Lem.\ \ref{lem:furcht}),
$|A_{\ell}| \leq |A|^{1 + O_{\ell}(\epsilon)}$ for every $\ell$. Starting
from here, we shall arrive at a contradiction for $\epsilon$ small.

By Corollary \ref{cor:dophus}, there is a maximal torus $T/\overline{K}$
of $G$ such that
\begin{equation}\label{eq:hele}|A_k \cap T(K)| \gg \frac{|A|}{|A_{k+2}|} |A|^{1/3} \geq
|A|^{\frac{1}{3} - O(\epsilon)},
\end{equation} where $k$ and the implied constants are absolute. 
We may write the elements
of $T(K)$ as diagonal matrices, after conjugation by an appropriate element of
$\SL_2(\overline{K})$. We can thus see that any $3$ elements $t_0,t_1,t_2\in
A_k\cap T(K)$ are linearly dependent. (Linear dependences are invariant 
under conjugation.)

In particular, for $t_0 = I$, $t_1 = t$, $t_2 = t^2$ ($t\in T(K)$ given), we have
\begin{equation}\label{eq:colocok}
c_0 t_0 + c_1 t_1 + c_2 t_2 = 0\end{equation}
for $c_0 = 1$, $c_1 = -(r+r^{-1})$ and $c_2 =1$, where $r$ and $r^{-1}$
are the eigenvalues of $t$. The map \[\phi:t\mapsto (c_0,c_1,c_2)\]
from $T$ to $\mathbb{A}^3$ is almost injective:  the preimage of any
point $(c_0,c_1,c_2)\in K^3$ consists of at most two elements of
$T(\overline{K})$. (The only thing that is particularly good about the
choice $t_0 = I$, $t_1 = t$, $t_2 = t^2$ is that this almost-injectivity
is easy to prove for this choice, as we have just seen.)

It follows immediately from (\ref{eq:colocok}) that, for any $g\in G$,
\begin{equation}\label{eq:melan}
c_0 \tr(t_0 g) + c_1 \tr(t_1 g) + c_2 \tr(t_2 g) = 0.\end{equation}
If $g\in A$, then $t_0 g, t_1 g, t_2 g \in A_{k+2}$. 
(To see this, note that, if a basis is chosen for which $t$ is diagonal and
$\tr(t_i g)$ is then written out in full, the only entries of $g$
appearing in $\tr(t_i g)$ are the diagonal entries
$g_{i i}$; moreover, the coefficient of $g_{i i}$ in
(\ref{eq:melan}) is $c_0 (t_0)_{i i} + c_1 (t_1)_{i i} + c_2 (t_2)_{i i}$,
which is $0$ by (\ref{eq:colocok}).)

It is worthwhile to examine $A_{k+2}$ in some more detail. By Corollary
\ref{cor:basu},
\[|\tr(A')|\ll |A_{(k+2) k'}|^{\frac{1}{3} + O(\epsilon)} \ll |A|^{\frac{1}{3} + O(\epsilon)},\]
where $A'$ is the set of regular semisimple elements of $A_{k+2}$,
and $k'$ and the implied constant are absolute.
(We may apply Cor.\ \ref{cor:basu} because we may assume that
$|A|$ is larger than an absolute
 constant: if $|A|$ is smaller than an absolute constant,
the statement we seek to prove is trivial.)

In $\SL_2$, a non-semisimple element has trace $2$; thus, we may write simply
\begin{equation}\label{eq:lafa}
|\tr(A_{k+2})|\ll |A|^{\frac{1}{3} + O(\epsilon)} + 1 \ll |A|^{\frac{1}{3}
+ O(\epsilon)},\end{equation}
where the implied constants are absolute. (We are assuming, as we may, 
that $|A|$ is larger than an absolute constant, and that $\epsilon>0$
is smaller than an absolute constant.)

By escape from subvarieties (as in Lem.\ \ref{lem:lemfac}), there is an element
$z\in A_{k''}$ ($k''$ absolute) of the form
\[z = \left(\begin{matrix} a & b\\c &d\end{matrix}\right)\]
with $a$, $b$, $c$, $d$ non-zero. (We are still writing elements of
$\SL_2(K)$ as matrices in such a way that $T(\overline{K})$ is diagonal.)
Then, for any diagonal $t\ne \pm I$, the map
\[g\mapsto (\tr(g),\tr(t g),\tr(z g)) = (g_{1 1} + g_{2 2}, 
t_{1 1} g_{1 1} + t_{2 2} g_{2 2}, a g_{1 1} + d g_{2 2} + b g_{2 1} +
c g_{1 2})\]
is almost injective on $\SL_2$: since we know that $g_{1 1} g_{2 2}
- g_{1 2} g_{2 1} = 1$, the preimage of any point
$(\tr(g),\tr(t g), \tr(z g))$ consists of at most two elements.
Now $\tr(z g)\in \tr(A_{k + k''})$, and, as in (\ref{eq:lafa}),
\[|\tr(A_{k + k''})| \ll |A|^{\frac{1}{3} + O(\epsilon)},\]
where the implied constants (here and everywhere from now on) are absolute.
Hence the image of $A$ under the map
\[g\mapsto (\tr(g), \tr(t g))\]
has at least \begin{equation}\label{eq:tifr}\gg \frac{|A|}{|A|^{\frac{1}{3} + O(\epsilon)}} = 
|A|^{\frac{2}{3} - O(\epsilon)}\end{equation} elements.

We are now in the situation covered by Cor.\ \ref{cor:espada}: we
have many tuples ($\gg |A|^{\frac{2}{3} - O(\epsilon)}$) with entries 
(namely, $\tr(g)$, $\tr(t g)$ and $\tr(z g)$) in a
small set ($|\tr(A_{k+1})|\ll |A|^{\frac{1}{3} + O(\epsilon)}$) and these tuples
satisfy many linear relations (one for each element of $T(K)\cap A_k$).
More formally: let $R = \mathbb{Z}/p\mathbb{Z}$, 
$X = \tr(A_{k+2})$, 
\[Y = \left\{(r+r^{-1},-1)\in ((\mathbb{Z}/p\mathbb{Z})^*)^2 : 
r\ne \pm i,\;\; \left(\begin{matrix} r & 0\\0 & r^{-1}\end{matrix}\right)
\in A_k \cap T(\overline{K}) \right\}.\]
For each $\vec{y}=(r+r^{-1},-1)\in Y$, let $t_{\vec{y}}$ be an element
of $A_k \cap T(\overline{K})$ having $r$, $r^{-1}$ as its eigenvalues.
(There can be at most two such elements for given $\vec{y}$.)
We define
\[X_{\vec{y}} = (\tr(t_{\vec{y}} x), \tr(x)).\]
Then, by (\ref{eq:melan}), we have 
\[y_0 \tr(t_{\vec{y}} x) + y_1 \tr(x) = \tr(t_{\vec{y}}^2 x),\]
and thus
\[\vec{y} \cdot X_{\vec{y}} \subset X.\]
At the same time, 
\[|Y| \geq \frac{1}{2} (A_k \cap T(\overline{K})) - 1
\gg |A|^{\frac{1}{3} - O(\epsilon)} \gg |X|^{1 - O(\epsilon)}\]
by (\ref{eq:hele}), (\ref{eq:lafa}) and $X = \tr(A_{k+2})$, and
\[|X_{\vec{y}}| \gg |A|^{\frac{2}{3} - O(\epsilon)} \gg
|X|^{2 - O(\epsilon)}.\]
by (\ref{eq:tifr}). (All the constants are absolute.) We apply
Cor.\ \ref{cor:espada} and reach a contradiction, provided
that $\epsilon>0$ is smaller than a positive constant depending only on $\eta$ 
and that $|A|$ is larger than a constant depending only on $\eta$,
$\epsilon$ and $\delta$. (The condition on $|A|$ is needed so that
the condition $|X|<p^{1-\delta'}$, $\delta'>0$, of Cor.\ \ref{cor:espada}
is fulfilled; we fulfil it by means of (\ref{eq:lafa}) and the assumption
$|A|<p^{3-\delta}$.)

We set $\epsilon>0$ to be smaller than the positive constant just mentioned.
As is stated in Cor.\ \ref{cor:espada}, $\eta$ depends only on $\delta$.
Hence, for the contradiction to happen,
 it is enough to assume that $|A|$ is larger than a constant
depending only on $\delta$.
We can certainly assume this, as otherwise the statement
(\ref{eq:consuf}) is trivially true. We have thus indeed reached
a contradiction, and we are done.
\end{proof}
\subsection{Small and fairly large sets in $\SL_3$}
The main idea is essentially the same as that in \S \ref{sec:otrogo}.
Consider $4$ diagonal matrices $t_0,t_1,t_2,t_3\in \SL_3(K)$. The maps from
$\SL_n$ to $K\times K$ given by
\begin{equation}\label{eq:kerala}
g\mapsto \left(\begin{matrix}\tr(t_0 g)\\ \tr((t_0
    g)^{-1})\end{matrix}\right),\;
g\mapsto \left(\begin{matrix}\tr(t_1 g)\\ \tr((t_1
    g)^{-1})\end{matrix}\right),\;
g\mapsto \left(\begin{matrix}\tr(t_2 g)\\ \tr((t_2
    g)^{-1})\end{matrix}\right),\;
g\mapsto \left(\begin{matrix}\tr(t_3 g)\\ \tr((t_3
    g)^{-1})\end{matrix}\right)
\end{equation}
can be seen as linear forms (linear over $K\times K$, that is) on $3$
variables. The $3$ variables in question are
\[\left(\begin{matrix}g_{11}\\ (g^{-1})_{11}\end{matrix}\right),\;\;\;
\left(\begin{matrix}g_{22}\\ (g^{-1})_{22}\end{matrix}\right),\;\;\;
\left(\begin{matrix}g_{33}\\ (g^{-1})_{33}\end{matrix}\right),\]
which are elements of $K\times K$.

(We are interested in tuples of the form \[\left(\begin{matrix} \tr(h)\\
\tr(h^{-1})\end{matrix}\right),\;\;\;\;\;\; h\in \SL_3(K),\]
 because the tuple $\kappa(h)=(a_2,a_1)$ of coefficients of the characteristic
polynomial $t^3 + a_2 t^2 + a_1 t - 1$ 
of an element $h$ of $\SL_3(K)$ is $\kappa(h) = (-\tr(h),\tr(h^{-1}))$.)

Since the maps (\ref{eq:kerala}) are linear forms on $3$ variables,
they must be linearly dependent; that is, for each choice
$t_0,t_1,t_2,t_3\in A_k \cap T(\overline{K})$, there are\footnote{
Here and henceforth we write elements of $K\times K$ in the form
$\left(\begin{matrix} a \\b\end{matrix}\right)$. The multiplication rule
is \[\left(\begin{matrix} a \\b\end{matrix}\right) \cdot
\left(\begin{matrix} c \\d\end{matrix}\right) =
\left(\begin{matrix} a \cdot c \\ b \cdot d\end{matrix}\right).\]}
$c_0,c_1,c_2,c_3\in K\times K$ such that
\begin{equation}\label{eq:triv}
c_0 \left(\begin{matrix}\tr(t_0 g)\\ \tr((t_0
    g)^{-1})\end{matrix}\right) + 
c_1 \left(\begin{matrix}\tr(t_1 g)\\ \tr((t_1
    g)^{-1})\end{matrix}\right) +
c_2 \left(\begin{matrix}\tr(t_2 g)\\ \tr((t_2
    g)^{-1})\end{matrix}\right) +
c_3 \left(\begin{matrix}\tr(t_3 g)\\ \tr((t_3
    g)^{-1})\end{matrix}\right) = 0\end{equation}
for all $g\in \SL_3(K)$. 

Varying $t_0, t_1, t_2, t_3$ within $A_k \cap T(\overline{K})$, we will obtain
many linear relations of the form (\ref{eq:triv}), and, as in
\S \ref{sec:otrogo}, we will obtain a contradiction to Cor.\ \ref{cor:espada}
thereby.

\begin{lem}\label{lem:worot}
Let $G = \SL_3$. Let $K$ be a field. Let $T/\overline{K}$ be a maximal torus
of $G$. Let $\Sigma$ be the Zariski-open set of regular semisimple matrices
in $G$.

Then there is a map \[c:(T \cap \Sigma) \to \mathbb{A}^1/\overline{K}\] 
such that,
for any $t\in (T\cap \Sigma)(\overline{K})$,
\begin{equation}\label{eq:ontor}
I - c(t) \cdot t + c(t^{-1}) t^2 - t^3 = 0 .\end{equation}
Moreover, the preimage $\phi^{-1}(\{x\})$ of any $x\in \mathbb{A}^2$
under the map $\phi:(T\cap \Sigma) \to \mathbb{A}^2$ given by
$\phi(t) = (c(t),c(t^{-1}))$ has at most $6$ elements.
\end{lem}
We recall that a matrix in $\SL_n$ is {\em regular semisimple} if and only if
all of its eigenvalues are distinct.
\begin{proof}
Write the elements of $G$ so that the elements of $T$ become diagonal
matrices. Let $t = \left(\begin{matrix} \lambda_1 & 0 & 0\\
0 &\lambda_2 &0\\0 &0 &\lambda_3\end{matrix}\right)$. We define
\begin{equation}\label{eq:odorn}
\left(\begin{matrix} c_0(t)\\ c_1(t)\\ c_2(t)\end{matrix}\right) = 
\left(\begin{matrix} 1 & \lambda_1 & \lambda_1^2\\
1 & \lambda_2 & \lambda_2^2\\ 1 & \lambda_3 &
\lambda_3^2\end{matrix}\right)^{-1} \cdot
\left(\begin{matrix} \lambda_1^3\\ \lambda_2^3\\
    \lambda_3^3\end{matrix}\right) .
\end{equation}
Then $c_0(t) I + c_1(t) t + c_2(t) t^2 = t^3$. Starting from (\ref{eq:odorn}),
a quick computation
(using Cramer's rule, say)
gives us $c_0(t) = 1$ and $c_2(t) = - c_1(t^{-1})$.
Let $c(t) = -c_1(t)$. Then (\ref{eq:ontor}) holds.

Now, for any four distinct elements $\lambda_1, \lambda_2, \lambda_3, 
\lambda_4\in \overline{K}^*$, the determinant of the matrix
\begin{equation}\label{eq:gagar}\left(\begin{matrix}
1 & \lambda_1 & \lambda_1^2 & \lambda_1^3\\
1 & \lambda_2 & \lambda_2^2 & \lambda_2^3\\
1 & \lambda_3 & \lambda_3^2 & \lambda_3^3\\
1 & \lambda_4 & \lambda_4^2 & \lambda_4^3
\end{matrix}\right)\end{equation}
is a Vandermonde determinant, and hence (since $\lambda_1,\dotsc,\lambda_4$
are
distinct) non-zero. However, if the same relation (\ref{eq:ontor})
were satisfied by two matrices the union of whose sets of eigenvalues
has at least four distinct elements $\lambda_1,\dotsc,\lambda_4$,
then $1$, $\lambda_j$, $\lambda_j^2$ and $\lambda_j^3$ would satisfy
the same linear relation (\ref{eq:ontor}) for $j = 1,2,3,4$. In other words,
the columns of the matrix (\ref{eq:gagar}) would be linearly dependent.
We have reached a contradiction. Hence $(c(t),c(t^{-1})) = (c(t'),c(t'^{-1}))$
can hold for $t, t'\in (T\cap \Sigma)(\overline{K})$ only if the set of
eigenvalues of $t$ equals the set of eigenvalues of $t'$. For $t$ given,
this can happen for only $3! = 6$ possible values of $t'$.
\end{proof}

\begin{prop}\label{prop:ogoth}
Let $G = \SL_3(\mathbb{Z}/p\mathbb{Z})$, $p$ a prime. Let $A\subset G$ be
a set of generators of $G$. Assume 
either $|A| \leq p^{4 - \delta}$, $\delta>0$, or $p^{4 + \delta} \leq |A| \leq
p^{8 - \delta}$, $\delta>0$. Then
\begin{equation}\label{eq:consodio}
|A\cdot A \cdot A|\gg_{\delta} |A|^{1 + \epsilon},\end{equation}
where $\epsilon>0$ depends only on $\delta$.
\end{prop}
\begin{proof}
Let $K = \mathbb{Z}/p\mathbb{Z}$. We can assume $\charac(K)=p>3$, as otherwise
the result to be proven is trivial.

Suppose $|A \cdot A\cdot A|\leq |A|^{1 + \epsilon}$. Then $|A_l| \leq
|A|^{1 + O_l(\epsilon)}$ for every positive $l$. We shall proceed from here and
arrive at a contradiction for $\epsilon$ sufficiently small.

By Corollary \ref{cor:dophus}, there is a maximal torus $T/\overline{K}$
of $G$ such that
\begin{equation}\label{eq:boon}
|A_k \cap T(K)| \gg |A|^{\frac{1}{4} - O(\epsilon)},\end{equation}
where $k$ and the implied constants depend only on $n=3$, and are
hence absolute. (Because $n=3$ is fixed, all constants that
would usually depend on $n$ will be absolute.) 

Let $c:T\cap \Sigma \to \mathbb{A}^1$ be as in Lemma \ref{lem:worot}.
Then (\ref{eq:ontor}) implies that
\begin{equation}\label{eq:coat}
\left(\begin{matrix} c(t)\\ c(t^{-1})\end{matrix}\right) \cdot
\left(\begin{matrix} \tr(t g)\\ \tr((t g)^{-1})\end{matrix}\right) -
\left(\begin{matrix} c(t^{-1})\\ c(t)\end{matrix}\right) \cdot
\left(\begin{matrix} \tr(t^2 g)\\ \tr((t^2 g)^{-1})\end{matrix}\right) + 
\left(\begin{matrix} \tr(t^3 g)\\ \tr((t^3 g)^{-1})\end{matrix}\right) = 
\left(\begin{matrix} \tr(g)\\ \tr(g^{-1})\end{matrix}\right).\end{equation}

It is time to prepare ourselves to use Corollary \ref{cor:espada}.
 We first apply Cor.\ \ref{cor:nipon} to obtain a large subset $E'
\subset A_k\cap T(K)$
(meaning a set $E'\subset A_k\cap T(K)$ with $|E'| \gg |A_k \cap T(K)| \gg |A|^{\frac{1}{4} -
  O(\epsilon)}$,
where the constants are absolute)
satisfying the conclusion of Cor.\ \ref{cor:nipon}.
Let 
$R = (\mathbb{Z}/p\mathbb{Z})^2$, $X = \kappa(A_{k'} \cap \Sigma(K))$
(where we set $k'$ equal to the value of $k$ in Cor.\ \ref{cor:nipon}
plus thrice the value of $k$ in (\ref{eq:boon})),
\[Y = \left\{\left(\left(\begin{matrix}c(t)\\ c(t^{-1})\end{matrix}\right),
\left(\begin{matrix} - c(t^{-1})\\ - c(t)\end{matrix}\right),
\left(\begin{matrix} 1\\ 1\end{matrix}\right)\right) : t\in E'\right\};\]
let $X_{\vec{y}}$ be the
set of all tuples
\[(\kappa(t g), \kappa(t^2 g),\kappa(t^3 g))\]
with $g\in A_k$ satisfying $h_0 g\in \Sigma(K)$ (for some
fixed $h_0\in A_k$ given by Cor.\ \ref{cor:nipon}) and
$t^{\ell} g \in \Sigma(K)$ for $\ell = 0,1,2,3$. (The conclusion
of Cor.\ \ref{cor:nipon} was precisely that there are many such tuples.)

Having defined the sets to be used in our application of
Cor.\ \ref{cor:espada}, we must now verify 
the assumptions of Cor.\ \ref{cor:espada}. (We already started to do
so while defining the sets.) 
The projection $\pi_1:(R^*)^3\to R^*$ onto the first coordinate is
clearly injective on $Y$: if we know
$\left(\begin{matrix}c(t)\\ c(t^{-1})\end{matrix}\right)$,
we know
$\left(\begin{matrix} - c(t^{-1})\\ - c(t)\end{matrix}\right)$. By Corollaries
\ref{cor:london} and \ref{cor:basu},
\begin{equation}\label{eq:coraz}
|A|^{\frac{1}{4}} \ll |X| \ll |A|^{\frac{1}{4} + O(\epsilon)},\end{equation}
where the implied constants are absolute. (In applying Cor.\ \ref{cor:basu},
we are assuming, as we may, that $|A|$ is larger than an absolute constant;
otherwise the statement we seek to prove is trivial.)

We are assuming either $|A|\leq p^{4-\delta}$ or $p^{4 + \delta} \leq |A|\leq
p^{8 - \delta}$. Hence, for $\epsilon$ small enough in terms of $\delta$
and $p$ large enough in terms of $\delta$, (\ref{eq:coraz}) implies
that either
\[|X|\leq p^{1 - \delta/2} \;\;\;\;\;\text{or}\;\;\;\;\;
p^{1 + \frac{\delta}{2}} \leq |X| \leq p^{2 - \frac{\delta}{2}} .
\] Finally, for every $\vec{y} \in Y$ and every $\vec{x} \in X_{\vec{y}}$, 
(\ref{eq:coat}) gives us that
\[\vec{y}\cdot \vec{x} = \left(\begin{matrix} \tr(g)\\ \tr(g^{-1})\end{matrix}
\right) \in \kappa(A_{k'}).\]
Because of the way we defined $X_{\vec{y}}$, the tuple $\vec{y}\cdot
\vec{x}$ lies in $\kappa(\Sigma(K))$ as well.

Now we apply Corollary \ref{cor:espada}. It remains only to check that
neither assertion in the conclusion (\ref{eq:airpo}) holds. We will then have
obtained a contradiction. By Lemma \ref{lem:worot}, $|Y|\geq \frac{1}{6} |E'|$;
by Corollary \ref{cor:dophus}, \[
|E'| \gg |A|^{\frac{1}{4} - O(\epsilon)} \gg |X|^{1 - O(\epsilon)},\]
where the implied constants are absolute. We conclude
that the first assertion in (\ref{eq:airpo}) fails to hold for $\epsilon$
sufficiently small in terms of $\eta$.

Now, by Cor.\ \ref{cor:nipon}, assuming that $\epsilon$ is less than
an absolute constant $\epsilon_0$, we have that, for every $t\in E'$,
there are $\gg |A|$ distinct tuples 
\[(\kappa(h_0 g), \kappa(t g), \kappa(t^2 g), \kappa(t^3 g))\]
with $g\in A_k$ satisfying $h_0 g\in \Sigma(K)$ and
$t^{\ell} g \in \Sigma(K)$ for $\ell = 0,1,2,3$. Now, by Cor.\ \ref{cor:basu},
the number of possible values taken by the first variable
$\kappa(h_0 g)$ is at most $\ll |A|^{\frac{1}{n+1} + O(\epsilon)} = 
|A|^{\frac{1}{4} + O(\epsilon)}$, where the implied constants are absolute. 
Thus, the number of elements of $X_{\vec{y}}$
-- that is, the number of distinct tuples
$(\kappa(t g), \kappa(t^2 g), \kappa(t^3 g))$ --  is at least 
\[\gg |A|^{\frac{3}{4} - O(\epsilon)} \gg |X|^{3 - O(\epsilon)},\]
where the implied constants are absolute.
Hence the second assertion in (\ref{eq:airpo}) fails to hold for
$\epsilon$ sufficiently small in terms of $\eta$ (and of $n$, which is 
a constant) and $|A|$ larger than a constant depending only on $\eta$.

By Cor.\ \ref{cor:espada}, $\eta$ depends only on $n$ and $\delta$, and thus
only on $\delta$. We have it in the statement that we may assume that
$\epsilon$ is smaller than a constant depending on $\delta$. We may also
assume that $|A|$ is larger than a constant depending on $\delta$,
as the implied constant in (\ref{eq:consodio}) may be taken to depend
on $\delta$. Hence we are done.
\end{proof}
\section{Subgroups and solvable groups}\label{sec:harto}

We must examine how the existence of growth in subgroups of a group affects
growth in the group itself. In particular, we want to have the tools that will allow
us later to do induction on the group type by passing to subgroups.

We would also like to examine now how sets grow in solvable groups.
(We already started to look into the issue in \S \ref{subs:bore}.) 
The growth of sets in a solvable group has a much
more direct relationship to sum-product phenomena than the growth of sets
that generate $\SL_2(K)$ or $\SL_3(K)$ does.

\subsection{Lemmas on growth and subgroups} 
Let us start with two very simple lemmas.
\begin{lem}\label{lem:gorto}
Let $G$ be a group and $H$ a subgroup thereof. Let $A,B\subset G$ be finite sets.
Then
\[|A\cdot B| \geq r \cdot |B\cap H|,\]
where $r$ is the number of cosets of $H$ intersecting $A$.
\end{lem}
We will usually apply this lemma with $A=B$.
\begin{proof}
Let $S\subset A$ be a set consisting of one coset representative $g\in A$ for
every coset of $H$ intersecting $A$.  Since any two distinct cosets of a
subgroup are disjoint, we have that
(a) $|S|=r$, (b) all elements of the form $g\cdot h$ ($g\in S$, $h\in B\cap H$) are
distinct. Thus there are $|S|\cdot |B\cap H| = r\cdot |B\cap H|$ of them. 
\end{proof}

\begin{lem}\label{lem:duffy} 
Let $G$ be a group and $H$ a subgroup thereof. Let $A\subset G$ be a 
non-empty finite set. 
Then
\[|A^{-1} A \cap H| \geq \frac{|A|}{r},\] 
where $r$ is the number of cosets of $H$ intersecting $A$. In particular,
\[|A^{-1} A \cap H| \geq \frac{|A|}{\lbrack G:H\rbrack}.\]
\end{lem}
\begin{proof} 
By the pigeonhole principle, there is at least one coset $g H$ 
of $H$ containing at least $|A|/r$ elements of $A$ (and thus, in particular, 
at least one element of $A$). Choose an element $a_0 \in g H \cap A$. 
Then, for every $a\in g H \cap |A|$, the element $a_0^{-1} a$ lies both in $H$ and 
in $A^{-1} A$. As $a_0$ is fixed and $a$ varies, the elements $a_0^{-1} a$ 
are distinct.
\end{proof}

One of the reasons 
why we are interested in subgroups is that growth in subgroups
$H$ of $G$ gives us growth in the group $G$.
\begin{lem}\label{lem:koph} 
Let $G$ be a group and $H$ a subgroup thereof. Let $A\subset G$ be 
a non-empty finite set. Then, for any $k>0$,
\[|A_{2 k + 1}| \geq \frac{|(A^{-1} A \cap H)_k|}{|A^{-1} A \cap H|} |A|.\]
\end{lem}
\begin{proof} 
Let $r$ be the number of cosets of $H$ intersecting $A$. It is clear that, 
for any $E\subset H$,
\[|A\cdot E| \geq r \cdot |E|.\] 
In particular,
\[|A\cdot (A^{-1} A \cap H)_k| \geq r\cdot |(A^{-1} A\cap H)_k|\] 
and the left side is evidently $\leq |A_{2 k + 1}|$. Now, by 
Lemma \ref{lem:duffy}, $|A^{-1} A \cap H| \geq \frac{|A|}{r}$. Hence
\[|A_{2 k + 1}| \geq
|A\cdot (A^{-1} A \cap H)_k| \geq r\cdot |(A^{-1} A\cap H)_k|
\geq \frac{|(A^{-1} A\cap H)_k|}{|A^{-1} A\cap H|} |A|.\]
\end{proof}

Growth in a quotient set also gives us growth in the group.
\begin{lem}\label{lem:quotgro} 
Let $G$ be a group and $H$ a subgroup thereof. Let $G/H$ be 
the quotient set and $\pi:G\to G/H$ 
the quotient map. 
Then, for any finite non-empty subsets $A_1, A_2\subset G$,
\[|(A_1\cup A_2)_4| \geq \frac{|\pi(A_1 A_2)|}{|\pi(A_1)|} |A_1| .
\]
\end{lem} 
Actually, we will apply this lemma only for normal subgroups $H<G$, but 
it is true in general.
\begin{proof} 
By Lemma \ref{lem:duffy},
\[|A_1^{-1} A_1 \cap H| \geq \frac{|A_1|}{\pi(A_1)} .\] 
At the same time, it is clear that
\[|A_1 A_2 A_1^{-1} A_1| \geq |\pi(A_1 A_2)| \cdot |A_1^{-1} A_1 \cap H|.\] 
Hence
\[|A_1 A_2 A_1^{-1} A_1| \geq \frac{|\pi(A_1 A_2)|}{|\pi(A_1)|} |A_1| .
\]
\end{proof}

\begin{lem}\label{lem:arpad}
Let $G$ be a group and $H$ a subgroup thereof. Let $G/H$ be 
the quotient set and $\pi:G\to G/H$ 
the quotient map. 

Let $A\subset G$ be a 
finite set. Let $A'$ be a subset of $A$.
Then
\[|A' \cdot (A^{-1} A\cap H)|\geq \frac{|\pi(A')|}{|\pi(A)|} |A|.\]
\end{lem}
\begin{proof}
By Lemma \ref{lem:duffy}, $|A^{-1} A\cap H|\geq \frac{|A|}{|\pi(A)|}$.
Since any distinct cosets of $H$ are disjoint,
it follows that 
\[|A' \cdot (A^{-1} A\cap H)|\geq |\pi(A')|\cdot |A^{-1} A\cap H| \geq |\pi(A')|\cdot \frac{|A|}{|\pi(A)|}.\]
\end{proof}

Let $A$ be a finite subset of $G$ and $H$ a subset of $G$. By Lemma \ref{lem:duffy}, either the 
intersection $A^{-1} A \cap H$ is large or there are many representatives in $A$ of cosets of $H$. 
What we are about to show is that we can in effect remove the condition that $H$ 
be a subgroup.
\begin{lem}\label{lem:repre} 
Let $G$ be a group. Let $R\subset G$ be a subset with $R = R^{-1}$. Let $A\subset G$ be finite. 

Then there is a subset $A'\subset A$ with
\[|A'|\geq \frac{|A|}{|A^{-1} A \cap R|}\] such that no element of $A'^{-1} A'$ (other than possibly the
identity) lies in $R$.
\end{lem}
\begin{proof}
Let $O = A^{-1} A\cap  R$; since $R=R^{-1}$, we know that $O=O^{-1}$. 

Let $g_1$ be an arbitrary element of $A$. If $A\subset g_1 O$, let $A' = \{g_1\}$ and stop. Otherwise,
let $g_2$ be in $A$ but not in $g_1 O$. If $A\subset g_1 O \cup g_2 O$, let $A' = \{g_1,g_2\}$ and
stop. Otherwise, let $g_3$ be in $A$ but not in $g_1 O \cup g_2 O$, etc. We eventually arrive at a 
covering $A\subset g_1 O \cup g_2 O \cup \dotsb \cup g_{\ell} O$ such that $g_j\notin g_i O$ for all 
pairs $(i,j)$, $1\leq i<j\leq \ell$. As $O=O^{-1}$, it follows that we also have $g_i\notin g_j O$.
Since $O = A^{-1} A \cap R$, this implies that $g_i^{-1} g_j\notin R$ for all $1\leq i,j\leq \ell$, $i\ne j$.

Let $A' = \{g_1,g_2,\dotsc,g_{\ell}\}$. What we have just shown can be restated as follows: no
element of $A'^{-1} A'\setminus \{e\}$ lies in $R$.

Now, because $A\subset g_1 O \cup g_2 O \cup \dotsb \cup g_{\ell} O$, there is a $g_i\in A'\subset A$
such that $|A\cap g_i O|\geq \frac{|A|}{\ell}$ (by the pigeonhole principle). Hence
$|O| = |g_i O| \geq \frac{|A|}{\ell}$. By the definition of $O$, we conclude that $|A^{-1} A \cap R|\geq
\frac{|A|}{\ell}$. Since $\ell = |A'|$, we obtain that $|A'|\geq \frac{|A|}{|A^{-1} A \cap R|}$.
\end{proof}

\subsection{Lemmas for solvable groups}


We will state the following lemmas in general, but they are especially useful for solvable
groups $G$. We write $G^{(1)} := \lbrack G,G\rbrack = \{x y x^{-1} y^{-1} : x,y\in G\}$.

\begin{lem}\label{lem:cumpars}
Let $G$ be a group. Let $G^{(1)} = \lbrack G,G\rbrack$. 
Let $A\subset G$ be a finite set.

Then, for every $\delta>0$, either
\begin{enumerate}
\item\label{it:gost1} $|A A A^{-1}|\geq |A|^{1+\delta}$, or
\item\label{it:gost3} there is a $g\in A$ such that
\begin{equation}\label{eq:batsht}
|C_G(g)\cap A^{-1} A| \cdot |G^{(1)} \cap A^{-1} A| \geq |A|^{1-\delta} .
\end{equation}
\end{enumerate}
\end{lem}
\begin{proof}
By Proposition \ref{prop:ostrogoth}, there is a $g\in A$ such that the set 
$C_G(g) \cap A^{-1} A$ has
\[\frac{|A|}{|A A A^{-1}|} \cdot |\Cl_G(A)|
\] elements.

By the pigeonhole principle, there is a conjugacy class $C$ in $G$ containing
$\geq \frac{|A|}{|\Cl_G(A)|}$ elements of $A$. For any two $g_1, g_2\in C$,
the quotient $g_1^{-1} g_2$ lies in $G^{(1)}$: there is an $h\in G$ such that 
$g_2 = h g_1 h^{-1}$, and so
\[g_1^{-1} g_2 = g_1^{-1} h g_1 h^{-1} \in G^{(1)} .\]
Fixing $g_1\in C$ and letting $g_2$ vary within $C$, we obtain that
there are at least $|C|$ distinct elements in $A^{-1} A \cap G^{(1)}$.

Therefore
\[\begin{aligned}
 |C_G(g)\cap A^{-1} A| \cdot |G^{(1)} \cap A^{-1} A| 
&\geq |C_G(g)\cap A^{-1} A| \cdot |C|\\
&\geq \frac{|A|}{|A A A^{-1}|} |\Cl_G(A)| \cdot
\frac{|A|}{|\Cl_G(A)|} \geq \frac{|A|}{|A A A^{-1}|} \cdot |A| .
\end{aligned}\]
If $|A A A^{-1}|\geq |A|^{1+\delta}$, we have conclusion (\ref{it:gost1}). Otherwise,
\[|C_G(g)\cap A^{-1} A| \cdot |G^{(1)} \cap A^{-1} A| \geq |A|^{1-\delta},\]
i.e., conclusion (\ref{it:gost3}).
\end{proof}


\begin{lem}\label{lem:hastar}
Let $G$ be a group. Let $H_1,\dotsc , H_m < G$ be proper subgroups such that, if $g\in G$ does not lie in any $H_j$, $1\leq j\leq m$, then $g x g^{-1} \ne x$ for every 
$x\in G^{(1)} \setminus \{e\}$.

Let $A\subset G$ be finite.
Then, for every $\delta>0$, either 
\begin{enumerate}
\item\label{it:cuph1} $|A A A^{-1}|\gg |A|^{1+ \delta}$, where the implied
constant is absolute, 
\item\label{it:cuph2} $|A_6 \cap (H_j\cdot G^{(1)})|\geq \frac{1}{2 m} |A|^{1-2\delta}$ for some 
$1\leq j\leq m$,  or
\item\label{it:caiphas} there is a subset $Y\subset A^{-1} A$ with 
$|Y|\geq |A|^{\delta}$ such that
\[g x g^{-1} \ne x\]
for every $x\in G^{(1)}\setminus \{e\}$ and every $g\in Y^{-1} Y\setminus
\{e\}$.
\end{enumerate}
\end{lem}
\begin{proof}
If $|A\cap (H_1\cup \dotsc \cup H_m)|>\frac{1}{2} |A|$, we arrive at (a stronger
version of) conclusion (\ref{it:cuph2}). Assume otherwise. Let
$A'=A\setminus (A\cap (H_1\cup \dotsc \cup H_m))$. Apply Lemma
\ref{lem:cumpars} with $A'$ instead of $A$. 
Case (\ref{it:gost1}) of Lemma \ref{lem:cumpars} gives
us conclusion (\ref{it:cuph1}) here.
Assume, then, that we are in case (\ref{it:gost3}) of Lemma
\ref{lem:cumpars}.

Apply Lemma \ref{lem:repre} with $R=H_1\cup H_2\cup \dotsc \cup H_m$
and $C_G(g) \cap A'^{-1} A'$ instead of $A$. We 
obtain a subset $Y\subset C_G(g)\cap A'^{-1} A'$ with
$Y^{-1} Y \cap R = \{e\}$ and
\[|Y|\geq \frac{|C_G(g) \cap A'^{-1} A'|}{|(C_G(g) \cap A'^{-1} A')^{-1}
(C_G(g) \cap A'^{-1} A') \cap R|} \geq \frac{|C_G(g)\cap A'^{-1} A'|}{
|(C_G(g) \cap A_4')\cap R|} .
\]
If $|Y|\geq |A|^{\delta}$, we have obtained conclusion (\ref{it:caiphas}).
 Assume $|Y|<|A|^{\delta}$. Then
\[|(C_G(g)\cap A_4')\cap R| \geq |A|^{-\delta} \cdot |C_G(g) \cap A'^{-1} A'|,\]
and so
\[|(C_G(g)\cap A_4')\cap H_j|\geq \frac{1}{m} |A|^{-\delta} \cdot |C_G(g) \cap A'^{-1} A'|\]
for some $1\leq j \leq m$. Since $g\notin H_1 \cup \dotsc \cup H_m$, 
we have $g x g^{-1} \ne x$ for every 
$x\in G^{(1)} = \lbrack G,G\rbrack$, and thus
$C_G(g)\cap G^{(1)} = \{e\}$. It follows that
\[\begin{aligned}
|A_6\cap (H_j\cdot G^{(1)})| &\geq |(C_G(g)\cap A_4')\cap H_j| \cdot
|G^{(1)} \cap A'^{-1} A'| \\ &\geq \frac{1}{m} |A|^{-\delta} \cdot 
|C_G(g) \cap A'^{-1} A'|
\cdot |G^{(1)} \cap A'^{-1} A'|\\ &\geq
\frac{1}{m} |A|^{-\delta} \cdot |A'|^{1-\delta} \geq
 \frac{1}{2 m} |A|^{1 - 2\delta},
\end{aligned}\]
where we use (\ref{eq:batsht}). We have obtained conclusion
(\ref{it:cuph2}).
\end{proof}

It is now that our generalised sum-product techniques come in.
\begin{lem}\label{lem:gallon}
Let $G$ be a group.
Assume that there is no chain of subgroups
\begin{equation}\label{eq:hortus}
\{e\} \lneq G_1\lneq G_2 \lneq \dotsb \lneq G_r \lneq G^{(1)}\end{equation}
with $r\geq \ell$, where $\ell$ is an integer.

Let $A\subset G$ be finite. Suppose that there is a subset $Y\subset A$,
$|Y|\geq |A|^{\delta}$, $\delta>0$, such that
\[g x g^{-1} \ne x\]
for every $x\in G^{(1)}\setminus \{e\}$ and every 
$g\in Y^{-1} Y \setminus \{e\}$.

Then either
\begin{enumerate}
\item\label{it:fircas} $|A_k|\geq |A|^{1+\delta}$,
where $k$ depends only on
$\delta$ and $\ell$, or
\item\label{it:seccas} 
there is a subgroup $X<G^{(1)}\cap \langle A\rangle$ such that
(i) $X\triangleleft \langle A\rangle$, (ii) $\langle A\rangle/X$ is abelian, 
(iii)
$A_k$ contains $X$ for some $k$ depending only on $\delta$ and $\ell$.
\end{enumerate}
\end{lem}
The condition on the non-existence of long chains (\ref{eq:hortus}) can probably be
relaxed; it will have to be if results uniform over $\alpha$ on algebraic groups over
$\mathbb{F}_{p^{\alpha}}$ are to be obtained. (We will not attempt to do as much in
this paper.)
\begin{proof}
If $\langle A\rangle$ is abelian, conclusion (\ref{it:seccas}) holds with
$X=\{e\}$. Assume otherwise. Then there are two elements $g_1$, $g_2$ of $A$
(and not just two elements of $\langle A\rangle$) that do no not commute
with each other. Hence $g_1 g_2 g_1^{-1} g_2^{-1}\ne e$, and so $A_4\cap
G^{(1)}
\ne \{e\}$. 

We now apply our generalised sum-product statement, Corollary \ref{cor:ogrodo},
with $Y$ acting on $S = A_4 \cap G^{(1)}$ by conjugation. 
We
obtain that
\[\begin{aligned}
|(Y_2(S))_6| &> \frac{1}{2} \min(|Y|\cdot |S|,|G_1|) \geq \frac{1}{2} \min(2 |A|^{\delta}
\cdot |S|, |G_1|)\\ &= \min( |A|^{\delta} |S|, \frac{1}{2} |G_1|),
\end{aligned}\]
where $G_1 = \langle \langle Y\rangle (\langle S\rangle) \rangle$ is a subgroup
of $G^{(1)}$. Since $S\ne \{e\}$, the group $G_1$ is not just $\{e\}$.

We apply Corollary \ref{cor:ogrodo} again and again - a total of 
$r = \lceil \frac{1}{\delta} \rceil+1$ times -- and obtain that
\[|(Y_{2 r}(S))_{6^r}| > \min(|A|^{1 + \delta} \cdot |S|, \frac{1}{2} |G_1|).\]
If $\min(|A|^{1 + \delta} \cdot |S|, \frac{1}{2} |G_1|) = |A|^{1 + \delta} \cdot |S|$, 
we have reached conclusion (\ref{it:fircas}). Assume, then, that
$\min(|A|^{1 + \delta} \cdot |S|, \frac{1}{2} |G_1|) = \frac{1}{2} |G_1|$. By
Lemma \ref{lem:rastropor}, it follows that
\[|(Y_{2 r}(S))_{2\cdot 6^r}| = G_1.\]  
Since
\[(Y_{2 r}(S))_{2\cdot 6^r} \subset
(A_{4 r} \cdot A_2 \cdot A_{4 r})_{2\cdot 6^r} \subset
A_{2 (8 r + 2) \cdot 6^r} , \]
we have shown that $G_1\subset A_k$, where $k =  
2 (8 r + 2) \cdot 6^r$ depends only on $\delta$.

If $G_1$ is a normal subgroup of $\langle A\rangle$ and $G/G_1$ is abelian, we have obtained
conclusion (\ref{it:seccas}) (with $X=G_1$) and are done. Assume otherwise. If $G_1$ is not a normal subgroup
of $\langle A\rangle$, there is necessarily a $g$ in $A$ itself (as opposed to just in $\langle A\rangle$)
and an $h\in G_1$ such that
such that $g h g^{-1}\notin G_1$. 
If $G_1$ is a normal subgroup but $\langle A\rangle/G_1$ is not abelian, there are two elements of $A$ (and not just two elements
of $\langle A\rangle$) that do not commute $\mo G_1$, i.e., two elements $g_1,g_2\in A$ such that $g_1 g_2 g_1^{-1} g_2^{-1} \notin G_1$. It is easy to see that, 
in the former case, $g h g^{-1}$ is in $G^{(1)}$; in the latter case, $g_1 g_2 g_1^{-1} g_2^{-1}$ is in $G^{(1)}$. At any rate, there is an element $g$ of $A_4$
such that $g\in G^{(1)}\setminus G_1$.

Now we apply Cor.\ \ref{cor:ogrodo} again and again to the set $H_1 \cup \{g\}$.
After applying it a total of $r = \lceil \frac{1}{\delta}\rceil+1$ times -- say -- 
we obtain
\[|(Y_{2 r}(G_1\cup \{g\}))_{6^r}| > \min\left(|A|^{1+\delta}\cdot |G_1\cup \{g\}|, 
 \frac{1}{2} |G_2|\right).\]
If  $\min\left(|A|^{1+\delta}\cdot |G_1\cup \{g\}|, 
 \frac{1}{2} |G_2|\right) = |A|^{1+\delta}\cdot |G_1\cup \{g\}|$, 
we have reached conclusion (\ref{it:fircas}). Suppose, then, that 
$\min\left(|A|^{1+\delta}\cdot |G_1\cup \{g\}|, 
 \frac{1}{2} |G_2|\right) = \frac{1}{2}|G_2|$; 
by Lemma \ref{lem:rastropor}, it follows that
\[(Y_{2 r}(G_1\cup \{g\}))_{2\cdot 6^r} = G_2,\]
and so $G_2\subset A_{k'}$, $k'$ depending only on $\delta$.

If $G_2$ is a normal subgroup of $\langle A\rangle$ stable under the action of $J$
and $\langle A\rangle/G_2$ is abelian, we have obtained
conclusion (\ref{it:seccas}) and are done. Otherwise, we proceed as before, constructing
an element $g$ of $A_{k''}\cap G^{(1)}$ not in $G_2$, and applying Cor.\ \ref{cor:ogrodo}
again and again to $G_2\cup \{g\}$, then to $G_3\cup \{g\}$, and so on. 
As there cannot be a chain
\[\{e\} \lneq G_1 \lneq G_2 \lneq \dotsb \lneq G_r \lneq U(K)\]
with $r\geq \ell$, we reach
conclusion (\ref{it:seccas}) in at most $\ell$ steps, if we do not reach
conclusion (\ref{it:fircas}) first.
\end{proof}


\begin{cor}[to Lemmas \ref{lem:hastar} and \ref{lem:gallon}]\label{cor:liz}
Let $G$ be a group. Let $H_1,\dotsc , H_m < G$ be proper subgroups such that, if $g\in G$ does not lie in any $H_j$, $1\leq j\leq m$, then $g x g^{-1} \ne x$ for every 
$x\in G^{(1)} = \lbrack G,G\rbrack$. 
Assume that there is no chain of subgroups
\begin{equation}\label{eq:brortus}
\{e\} \lneq G_1\lneq G_2 \lneq \dotsb \lneq G_r \lneq G^{(1)}\end{equation}
with $r\geq \ell$, where $\ell$ is an integer.

Let $A\subset G$ be finite.
Then, for every $\delta>0$, either 
\begin{enumerate}
\item\label{it:cruph1} $|A_k|\gg |A|^{1+ \delta}$, where the implied
constant is absolute and $k$ depends only on $\delta$ and $\ell$,
\item\label{it:cruph2} $|A_6 \cap (H_j\cdot G^{(1)})|\geq \frac{1}{2 m} |A|^{1-2\delta}$ for some 
$1\leq j\leq m$;
moreover, $A$ is contained in the union of at most $|A|^{3\delta}$ cosets of 
$H_j\cdot G^{(1)}$ for that same index $j$;
\item\label{it:cruph3}
there is a subgroup $X<G^{(1)}\cap \langle A\rangle$ such that
(i) $X\triangleleft \langle A\rangle$, (ii) $\langle A\rangle/X$ is abelian, 
(iii)
$A_k$ contains $X$ for some $k$ depending only on $\delta$ and $\ell$.
\end{enumerate}
\end{cor}
\begin{proof}
Apply Lemma \ref{lem:hastar}. If conclusion
(\ref{it:caiphas}) of Lemma \ref{lem:hastar} holds,
 apply Lemma \ref{lem:gallon} with
$A^{-1} A$ instead of $A$.
(The comment in conclusion (\ref{it:cruph2}) on how $A$ is contained
in the union of few cosets of $H_j\cdot G^{(1)}$
follows from Lemma \ref{lem:gorto}: if there were too many cosets intersecting $A_6$, conclusion (\ref{it:cruph1}) would follow.)
\end{proof}

The following easy lemma will come in useful later.
\begin{lem}\label{lem:bach}
Let $G$ be a group.
Assume that there is no chain of subgroups
\begin{equation}\label{eq:horto}
\{e\} \lneq G_1\lneq G_2 \lneq \dotsb \lneq G_r \lneq G^{(1)}\end{equation}
with $r\geq \ell$, where $\ell$ is an integer.

Let $A\subset G$ be finite. Let $B\subset A$. Then there are
$g_1,g_2,\dotsc, g_k\in A_{k'}$ such that 
\[\langle B\cup g_1 B g_1^{-1} \cup \dotsc \cup g_k B g_k^{-1}\rangle
\triangleleft \langle A\rangle,\]
where $k$ and $k'$ depend only on $\ell$.
\end{lem}
\begin{proof}
If $\langle B \rangle \triangleleft  \langle A\rangle$, we are done.
Suppose, then, that $\langle B\rangle$ is not a normal subgroup of $A$. 
Then there is a $g\in A \cup A^{-1}$ such that $g \langle B\rangle g^{-1}
\not\subset \langle B\rangle$. Let $B_1 = B \cup g B g^{-1}$. Since
$g \langle B\rangle g^{-1} = \langle g B g^{-1} \rangle$, it follows that 
$g b g^{-1} \notin \langle B\rangle$ for some $b\in B$. Thus 
$g b g^{-1} b^{-1} \notin \langle B\rangle$, and, since $g b g^{-1} b^{-1}
\in G^{(1)}$, this shows that $\langle B \rangle \cap G^{(1)}
\subsetneq \langle B_1\rangle \cap G^{(1)}$.

If $\langle B_1 \rangle \triangleleft \langle A\rangle$, we are done. 
Otherwise, we iterate: there is a $g_1\in A \cap A^{-1}$ such that 
$g_1\langle B_1 \rangle g^{-1} \ne \langle B_1 \rangle$, we set
$B_2 = B_1 \cup g_1 B_1 g_1^{-1}$, etc. We obtain a sequence of subgroups
\[\langle B \rangle < \langle B_1 \rangle < \langle B_2 \rangle <
\langle B_3 \rangle < \dotsc\]
with $B_{i+1} = B \cup g_{i+1} B g_{i+1}^{-1}$, $g_{i+1}\in A$, and
\begin{equation}\label{eq:klodor}
\langle B\rangle \cap G^{(1)} \lneq \langle B_1\rangle \cap G^{(1)}
\lneq \langle B_2 \rangle \cap G^{(1)} \lneq \dotsb \lneq G^{(1)}.
\end{equation}
By (\ref{eq:horto}), the chain of subgroups (\ref{eq:klodor}) cannot be
of length greater than $\ell$; thus, the iteration terminates after at most
$\ell$ steps. We obtain the statement of the Lemma with $k = 2^{\ell}-1$,
$k' = \ell$.
\end{proof}

\subsection{Examples: growth in Borel subgroups of $\SL_2(\mathbb{Z}/p\mathbb{Z})$ and $\SL_3(\mathbb{Z}/p\mathbb{Z})$}\label{subs:grotes}

We with to study the growth of sets in solvable subgroups of $\SL_2(K)$
and $\SL_3(K)$, where $K = \mathbb{Z}/p\mathbb{Z}$. The main case of interest
is that of Borel subgroups $B/K$.

\begin{prop}\label{prop:stick}
Let $K=\mathbb{Z}/p\mathbb{Z}$, $p$ a prime. Let $B/K$ be a Borel subgroup of 
$\SL_2/K$.
Let $U/K$ be the maximal unipotent subgroup of $B/K$.

Let $A\subset B(K)$. 
Then, for every $\delta>0$ smaller than an absolute constant, either
\begin{enumerate}
\item\label{it:cogot1} $|A_k|\gg |A|^{1+\delta}$, where the implied
constant is absolute and $k$ depends only on $\delta$,
\item\label{it:cogot2} $|A_6\cap (\{\pm I\}\cdot U(K))|\geq \frac{1}{2} |A|^{1-2\delta}$;
moreover, $A$ is contained in the union of at most $|A|^{3\delta}$ cosets of $U(K)$;
\item\label{it:cogot3} $A$ is contained in some maximal torus $T/\overline{K}$,
\item\label{it:cogot4} $A_k$ contains $U(K)$ for some $k$ depending only on $\delta$.
\end{enumerate}
\end{prop}
\begin{proof}
We apply Corollary \ref{cor:liz} with $G=B(K)$, $H_1=\{\pm I\}\cdot U(K)$, $m=1$,
$\ell=1$. (Here $\ell=1$ because $U(K)$ has no proper subgroups.)
Cases (\ref{it:cruph1}) and (\ref{it:cruph2}) of Cor.\ \ref{cor:liz}
give us conclusions (\ref{it:cogot1}) and (\ref{it:cogot2}).
Assume, then, that we are in case (\ref{it:cruph3}) of Cor.\ \ref{cor:liz}.
If $X=\{e\}$, then either conclusion (\ref{it:cogot3}) holds or $A$ is contained in
$\{\pm I\} \cdot U(K)$; in the latter case, conclusion (\ref{it:cogot2}) holds. 
If $X = U(K)$, then conclusion (\ref{it:cogot4}) holds.


\end{proof}

\begin{prop}\label{prop:lavender}
Let $K=\mathbb{Z}/p\mathbb{Z}$, $p$ a prime. Let $B/K$ be a Borel subgroup of 
$\SL_3/K$. Let $U/K$ be the maximal unipotent subgroup of $B/K$.

Let $A\subset B(K)$.
Then, for every $\epsilon>0$, one of the following conclusions holds:
\begin{enumerate}
\item\label{it:gana} $|A_k|\gg |A|^{1+\delta}$, where the implied constant
is absolute and $k$ and $\delta>0$ depend only on $\epsilon$;
\item\label{it:ganac}
there are subgroups $X\triangleleft Y\triangleleft \langle A\rangle$ 
such that
(a) $X<U(K)$, (b) $Y/X$ is nilpotent, (c)
$A_k$ contains $X$ for some $k$ depending only on $\epsilon$,
(d) $A$ is contained in the union $\leq |A|^{\epsilon}$ cosets of $Y$.
\end{enumerate}
\end{prop}
 
\begin{proof}
Let $\delta=\epsilon/3$.
Let $\rho_{i,j}:B(K)\to K^*$ taking an element of $B(K)$ with diagonal entries $r_1$, $r_2$, and $r_3$ to $r_i r_j^{-1}\in K^*$. (In other words, $\rho_{i,j}$ is a root map.)
Apply Cor.\ \ref{cor:liz} with $G=B(K)$ and $H_1$, $H_2$, $H_3$
equal to the kernels of $\rho_{1,2}$, $\rho_{2,3}$ and $\rho_{1,3}$, respectively.
(Condition (\ref{eq:hortus}) holds with $\ell=3$.) Cases (\ref{it:cruph1})
and (\ref{it:cruph3})
in Cor.\ \ref{cor:liz} give us conclusions (\ref{it:gana}) and
(\ref{it:ganac}) here.
Assume, then, that we are in case (\ref{it:cruph2}) of Cor.\ 
\ref{cor:liz} for some $j=1,2,3$. Let $H=H_j$. From the definition of our $H_j$, we have $H\cdot G^{(1)} = H$
for every $j=1,2,3$, and thus $|A_6\cap H|\geq \frac{1}{6} |A|^{1 - 2\delta}$.

We apply Lemma \ref{lem:bach} with $B = A_6 \cap H$, and obtain a set
$A' = B \cup g_1 B g_1^{-1} \cup \dotsc \cup g_k B g_k^{-1}\subset H$
such that $A_6 \cap H \subset A' \subset A_{k'}$ and $\langle A'\rangle
\triangleleft \langle A \rangle$, where $k'$ is an absolute constant.  Since
$|A'|\geq |A_6\cap H| \geq \frac{1}{6} |A|^{1-2\delta}$, Lemma \ref{lem:gorto}
(applied with $A'$ instead of $B$ and $\langle A'\rangle$ 
instead of $H$), either
$A$ is contained in the union of at most $|A|^{3\delta}$ cosets of 
$\langle A' \rangle$,
or conclusion (\ref{it:gana}) holds. 
 Let us assume
conclusion (\ref{it:gana}) does not hold.

{\em Case 1: $H=\ker(\rho_{1,2})$ or $H=\ker(\rho_{2,3})$.}
Apply Cor.\ \ref{cor:liz} once again, this time with $G=H$, $m=1$,
$H_1 = U(K)$ and $A'$ instead of $A$.
Cases (\ref{it:cruph1}), (\ref{it:cruph2}) and (\ref{it:cruph3})
give us conclusions (\ref{it:gana}), (\ref{it:ganac}) (with
$Y=U(K)$, $X=\{e\}$) and again (\ref{it:ganac}) (with $Y=\langle A'\rangle$ 
and $X$ as in the statement of conclusion (\ref{it:ganac})), respectively.

{\em Case 2: $H=\ker(\rho_{1,3})$.} Let $G= B(K)$, where $B/K$ is the
Borel subgroup of $\SL_3/K$ we are studying.
 Write $G^{(2)} = \lbrack G^{(1)},G^{(1)}\rbrack$.
If $g\in H$ is not contained in $U(K)$, then $g$ acts without fixed points 
by conjugation on $U(K)/G^{(2)}$.
Apply Cor.\ \ref{cor:liz} with $G=H/G^{(2)}$, $m=1$, $H_1 = U(K)/G^{(2)}$ and
 $A'' = \{h \cdot G^{(2)} : h\in A'\}$ instead of $A$. If case (\ref{it:cruph1}) of Cor.\ \ref{cor:liz}
holds, then Lemma \ref{lem:quotgro} gives us conclusion (\ref{it:gana}), unless
$A''$ is much smaller than $A'$ ($|A''|<|A'|^{\delta}$), in which case
$A'^{-1} A' \cap G^{(2)}$ must be very large ($\gg |A|^{1-2\delta}$), giving
us conclusion (\ref{it:ganac}) with $Y=U(K) \cap \langle A\rangle$, 
$X=\{e\}$.
Case (\ref{it:cruph2}) gives us conclusion (\ref{it:ganac}) with
$Y=U(K) \cap \langle A\rangle$, $X=\{e\}$.
It remains to examine
case (\ref{it:cruph3}) of Cor.\ \ref{cor:liz}. 

Suppose first that $X = U(K)/G^{(2)}$.
A quick calculation suffices to show that, for any set 
$C\subset U(K)$ such that $\{c G^{(2)}: c\in C\}$ is all of $U(K)/G^{(2)}$,
the set of commutators $\lbrack C,C\rbrack$ is all of $G^{(2)}$, and thus
$C_5 = U(K)$.  We conclude that $A_k$ contains $U(K)$
for some $k$ depending only on $\delta$; we have obtained conclusion
(\ref{it:ganac}) with $Y=\langle A\rangle$, $X=U(K)$.

Suppose now that $X = \{e\}$. Then $\langle A'\rangle/G^{(2)}$ is abelian,
and, since $G^{(2)}$ lies in the centre of $H=H_{1,3}$, the group
$\langle A'\rangle$ must itself be abelian. We have obtained
conclusion (\ref{it:ganac}) with $Y=\langle A'\rangle$ and $X=\{e\}$.

Suppose, lastly, that $X\ne U(K)/G^{(2)}$ and $X\ne \{e\}$.
We know that $\langle A'\rangle/H'$
is abelian. This implies that either
\[X = \left\{\left(\begin{matrix}1 & a & 0\\ 0 & 1 & 0\\0 & 0 & 1
\end{matrix}\right)\cdot G^{(2)} : a\in \mathbb{Z}/p\mathbb{Z}
\right\}
\]
and all elements of $A'$ are contained in the group
\[R = \left\{\left(\begin{matrix}r & a &b\\0 &r^{-2} & 0\\ 0 & 0 & r\end{matrix}\right):
r\in (\mathbb{Z}/p\mathbb{Z})^*, a,b\in (\mathbb{Z}/p\mathbb{Z})\right\},\]
or
\[X = \left\{\left(\begin{matrix}1 & 0 & 0\\ 0 & 1 & a\\0 & 0 & 1
\end{matrix}\right)\cdot G^{(2)} : a\in \mathbb{Z}/p\mathbb{Z}
\right\}
\]
and all elements of $A'$ are contained in the group
\[R = \left\{
\left(\begin{matrix}r & 0 &b\\0 &r^{-2} & a\\ 0 & 0 & r\end{matrix}\right)
: r\in (\mathbb{Z}/p\mathbb{Z})^*, a,b\in (\mathbb{Z}/p\mathbb{Z})\right\}.\]

We apply Cor.\ \ref{cor:liz} with $G=R$, $m=1$, $H_1 = U(K)\cap R$ and
$A'$ instead of $A$.
(We can do this because all elements of $G$ not in $H_1$ act on $G^{(1)}$ without fixed points: $G^{(1)}$ is now smaller than it was when $G$
was $B(K)$ or $H$.) 
Case (\ref{it:cruph1}) of Cor.\ \ref{cor:liz} gives us conclusion
(\ref{it:gana}) here, case (\ref{it:cruph2}) gives us conclusion
(\ref{it:ganac}) with $Y=U(K)\cap \langle A\rangle$, 
$X=\{e\}$, and case (\ref{it:cruph3})
gives us conclusion (\ref{it:ganac}) with $Y=\langle A'\rangle$ and
$X$ as in the statement of conclusion (\ref{it:ganac}).
\end{proof}


\subsection{Robustness under passage to subgroups}

We will need the fact that results such as Theorem \ref{thm:qartay} are robust
under passage to subgroups. We state the lemmas below only for
$H<G$ with $\lbrack G:H\rbrack=2$, 
since that is the only case we will actually use.
The arguments could probably be adapted to any $H<G$ with
$\lbrack G:H\rbrack$ bounded by a constant.

\begin{lem}\label{lem:kawins}
Let $G$ be a group. Let $H<G$ be a subgroup with $\lbrack G:H\rbrack = 2$.
Let $A'\subset G$ not be contained in $H$. Write $A' = C\cup g C'$,
where $C$ and $C'$ are subsets of $H$,
$g$ is not contained in $H$ and $C'$ contains the identity. 

 Then there is a subset $A\subset A'_3 \cap H$ such
that $\langle A'\rangle = \langle A\rangle \cup g \langle A\rangle$. Moreover,
$\frac{1}{2} |A'| \leq |A| \leq 4 |A'|$, $C \cup C'\subset A$, $g^2\in A$
 and $g^{-1} A g \subset A_3$.
\end{lem}
\begin{proof}
Define 
\begin{equation}\label{eq:janacek}
A = C\cup C' \cup g C g^{-1} \cup g C' g^{-1} \cup g^2 \subset
A'_3.\end{equation}
Since $\lbrack G:H\rbrack = 2$, $H$ is normal in $G$, and thus
$A\subset H$. Clearly $\frac{1}{2} 
|A'|\leq |A| \leq 4 |A'|$. It is also clear that
$g^{-1} A g \subset A_3$.

It remains to prove that 
$\langle A'\rangle = \langle A \rangle \cup g \langle A\rangle$,
where $g$ is as above. Clearly $\langle A\rangle \cup g \langle A\rangle$
is contained in $\langle A'\rangle$. To show that $\langle A'\rangle =
\langle C \cup g C'\rangle$ is contained in $\langle A\rangle \cup
g \langle A\rangle$, it is enough to show that, if $x\in C \cup g C'$
and $y\in \langle A\rangle \cup g \langle A\rangle$, then 
$x y \in \langle A\rangle \cup g \langle A\rangle$. Let us see:
\begin{enumerate}
\item if $c\in C \subset A$ and $y\in \langle A\rangle$, then
$c\cdot y \in \langle A\rangle$;
\item if $c'\in C'\subset A$ and $y\in \langle A\rangle$, then
$g c' y \in g \langle A\rangle$; 
\item if $c\in C$ and $y\in g\langle A\rangle$, then
$c\cdot y = g \cdot g^{-2} \cdot g c g^{-1} \cdot g y$, and, since $g^{-2}
\in A^{-1}$, $g c g^{-1} \in A$ and $g y \in g g \langle A\rangle = \langle
A\rangle$, we obtain that $c y \in g \langle A\rangle$;
\item if $c'\in C'$ and $y\in g \langle A\rangle$, then
$g c' y = g c' g^{-1} \cdot g y \in A \cdot g^2 \langle A\rangle =
\langle A\rangle$.
\end{enumerate}
Thus $\langle A'\rangle = \langle C \cup g C'\rangle \subset
\langle A\rangle \cup g \langle A\rangle$, and so $\langle A'\rangle
= \langle A \rangle \cup g \langle A\rangle$.
\end{proof}

\begin{lem}\label{lem:vangeli}
Let $H$ be a group. Let $H_1\triangleleft H$, $H'<H$. Then
$H_1 \cap H' \triangleleft H'$. Moreover, 
$H'/(H_1 \cap H')$ is isomorphic to a subgroup of
$H/H_1$.
\end{lem}
\begin{proof}
For any $g\in H'$ and any $h\in H_1\cap H'$, we have $g h g^{-1}\in H_1$
(because $H_1$ is normal) and $g h g^{-1} \in H'$ (because $g$ and $h$
are in $H'$). Thus, $H_1 \cap H' \triangleleft H'$.

We define a map $\iota:H'/(H_1\cap H')\to H/H_1$ as follows:
$\iota(g (H_1\cap H')) = g H_1$. It is easy to see that the map is
a well-defined homomorphism. Since its kernel is $\{e\}$, it is also
injective. 
\end{proof}

\begin{lem}\label{lem:cocot}
Let $M$ be a group. Let $N_1, N_2 \triangleleft M$. Let $A\subset M$.
Suppose that $A$ is contained in the union of $\leq n_1$ cosets of $N_1$;
suppose also that $A$ is contained in the union of $\leq n_2$ cosets of $N_2$.
Then $A$ is contained in the union of $\leq n_1 n_2$ cosets of $N_1\cap N_2$.
\end{lem}
\begin{proof}
The map $\iota:M/(N_1 \cap N_2) \to M/N_1 \times M/N_2$ given by
$\iota(g (N_1 \cap N_2)) = (g N_1, g N_2)$ is a well-defined homomorphism;
since its kernel is $\{e\}$, it is also injective. The image of
$\iota(A\cdot (N_1 \cap N_2))$ is of size at most $n_1 \cdot n_2$;
hence $A\cdot (N_1\cap N_2) \subset M/(N_1 \cap N_2)$ is of size at most
$n_1 \cdot n_2$.
\end{proof}

\begin{prop}\label{prop:amery}
Let $G$ be a group. Let $H<G$ be a subgroup with $\lbrack G:H\rbrack = 2$.

Suppose that, for every finite subset $A\subset H$ and every $\epsilon>0$,
either
\begin{equation}\label{eq:koko1}|A_k| \gg |A|^{1 + \delta},\end{equation}
where $k$ and $\delta$ depend only on $\epsilon$, or there are
subgroups $H_1 \triangleleft H_2 \triangleleft \langle A\rangle$ such that
\begin{enumerate}
\item $H_2/H_1$ is nilpotent,
\item $A_k$ contains $H_1$, where $k$ depends only on $\epsilon$, and
\item $A$ is contained in the union of $\leq |A|^{\epsilon}$ cosets of $H_2$.
\end{enumerate}

Then, for every finite subset $A'\subset G$ and every $\epsilon'>0$, either
either
\begin{equation}\label{eq:koko2}|A'_k| \gg |A'|^{1 + \delta'},\end{equation}
where $k$ and $\delta$ depend only on $\epsilon'$, or there are
subgroups $H_1' \triangleleft H_2' \triangleleft \langle A\rangle$ such that
\begin{enumerate}
\item $H_2'/H_1'$ is nilpotent,
\item $A'_k$ contains $H_1'$, where $k$ depends only on $\epsilon'$, and
\item $A'$ is contained in the union of $\leq |A'|^{\epsilon'}$ cosets of $H_2'$.
\end{enumerate}
\end{prop}
\begin{proof}
Let $A'$ and $\epsilon'>0$ be given. 
If $A'$ is contained in $H$, we are done. Assume $A'\not\subset H$.
Write $A' = C \cup g C$, $g\in G\setminus H$, as in the statement of
Lemma \ref{lem:kawins}.
By Lemma \ref{lem:kawins}, there is a subset $A\subset A'_3 \cap H$ such
that $\langle A'\rangle = \langle A\rangle \cup g \langle A\rangle$,
$\frac{1}{2} |A'|\leq |A|\leq 4 |A'|$ and $C \cup C' \subset A$. 
We apply our assumptions to $A$ with $\epsilon = \epsilon'/5$.
If (\ref{eq:koko1}) holds, (\ref{eq:koko2})
follows immediately and we are done. Assume (\ref{eq:koko1}) does not hold.
We obtain subgroups $H_1 \triangleleft H_2\triangleleft \langle A\rangle$
as in the statement.

Let $H_2' = H_2 \cap g H_2 g^{-1}$, $H_1' = H_1 \cap H_2'$.
By Lemma \ref{lem:vangeli} with $H= H_2$, $H_1 = H_1$ and
$H'=H_2'$, we have that
$H_1'\triangleleft H_2'$ and $H_2'/H_1'$ is isomorphic to a subgroup of
$H_2/H_1$.
Since $H_2/H_1$ is nilpotent, so is $H_2'/H_1'$. 
We now want to show that $H_2' \triangleleft \langle A'\rangle$. Recall
that $H_2\triangleleft \langle A\rangle$, $\langle A'\rangle = \langle
A\rangle
\cup g \langle A\rangle$ and $g^{-1} A g \subset A_3$. If $a\in A$, then
\[\begin{aligned}
a H_2' a^{-1} &= a H_2 a^{-1} \cap a g H_2 g^{-1} a^{-1} =
H_2 \cap (g \cdot g^{-1} a g \cdot H_2 \cdot (g^{-1} a g)^{-1} \cdot g^{-1})
\\ &= H_2 \cap (g \cdot a' H_2 (a')^{-1} \cdot g^{-1} = H_2 \cap g H_2 g^{-1})
= H_2',\end{aligned}\]
where $a' = g^{-1} a g \in A_3\subset \langle A\rangle$. It remains to
check that $g H_2' g^{-1} = H_2'$. Indeed,
\[g H_2' g^{-1} = g H_2 g^{-1} \cap g^2 H_2 g^{-2} = g H_2 g^{-1} \cap
H_2 = H_2',\]
where we use the facts that $H_2 \triangleleft \langle A\rangle$ and, by
Lemma \ref{lem:kawins}, $g^2\in A$.


Since $A\subset A'_3$
and $A_k$ contains $H_1$, we see that $A'_{3 k}$ contains $H_1'$. It remains
only to bound the number of cosets of $H_2'$ on which $A'$ lies. Since $A' = C
\cup g C'$ and $C,C'\subset A$, this is no greater than twice
the number of cosets of
$H_2'$ on which $A$ lies. We know that $A$ lies in $\leq |A|^{\epsilon}$
cosets of $H_2$. Since $g^{-1} A g \subset A_3$ and $H_2\triangleleft \langle
A\rangle$,  we deduce that $g^{-1} A g$ lies in $\leq
|A|^{3 \epsilon}$ cosets of $H_2$, and thus $A$ lies in $\leq |A|^{3 \epsilon}$
cosets of $g H_2 g^{-1}$. Lemma \ref{lem:cocot} now implies that $A$ lies on
$\leq |A|^{4 \epsilon}$ cosets of $H'_2 = H_2 \cap g H_2 g^{-1}$. Thus, $A'$
lies
on $\leq 2 |A|^{4 \epsilon} \leq 8 |A'|^{4 \epsilon} \leq |A|^{\epsilon'}$ cosets of
$H_2$. (We may assume $|A'|^\epsilon\geq 8$, as otherwise $|A'|$ is less than a
constant and
(\ref{eq:koko2}) holds trivially.)
\end{proof}

\section{Growth in proper subgroups of $\SL_3(\mathbb{Z}/p\mathbb{Z})$}\label{sec:pogor}
Let $K = \mathbb{Z}/p\mathbb{Z}$ and $G = \SL_3$. Suppose $A\subset
G(K)$ does not generate $G(K)$. Then $A$ generates a proper subgroup $\langle
A\rangle$ of $G$. Does $A$ grow? That is: does 
$|A \cdot A \cdot A|> |A|^{1 + \delta}$ hold?

The answer depends on which subgroup of $G$ the group $\langle A\rangle$
happens to be.
The subgroups of $G=\SL_3(\mathbb{Z}/p\mathbb{Z})$ 
are not particularly hard to
classify.

\begin{prop}[Mitchell \cite{Mi}]\label{prop:mitch}
Let $G=\PSL_3(\mathbb{Z}/p\mathbb{Z})$, $p$ odd.
 The maximal subgroups of $G$ are
\begin{enumerate}
\item\label{it:yor1} the stabiliser of a point in $\mathbb{P}^3(\mathbb{Z}/p\mathbb{Z})$,
\item\label{it:yor1b} the stabiliser of a line in $\mathbb{P}^3$ defined over $\mathbb{Z}/p\mathbb{Z}$,
\item\label{it:yor2} the stabiliser of a set of three points in 
$\mathbb{P}^3(\overline{\mathbb{Z}/p\mathbb{Z}})$,
\item\label{it:yor3} the stabiliser of a conic in
$\mathbb{P}^3(\mathbb{Z}/p\mathbb{Z})$,
\item\label{it:yor4} groups of order $\leq 360$.
\end{enumerate}
\end{prop}
\begin{proof}
This is Theorem 2.4 for $q$ prime in the survey paper
\cite{Ki}. Cases (a) and (b) in \cite[Thm.\ 2.4]{Ki} correspond to 
cases (\ref{it:yor1}) and (\ref{it:yor1b}) here; cases (c) and (d) correspond to case
(\ref{it:yor2}) here; case (e) is (\ref{it:yor3}) here; cases (f)--(i)
do not happen; finally, cases (j) and (k) in \cite[Thm.\ 2.4]{Ki} go into case (\ref{it:yor4}) here.
\end{proof}

From this, we get the following classification.
\begin{cor}\label{cor:odious}
Let $G=\SL_3(\mathbb{Z}/p\mathbb{Z})$, $p$ odd. Let $H$ be a proper
subgroup of $G$. Then at least one of the following statements holds:
\begin{enumerate}
\item\label{it:aleg1} $H$ is contained in the stabiliser of a point in $\mathbb{P}^3(\mathbb{Z}/p\mathbb{Z})$,
\item\label{it:aleg2} $H$ is contained in the stabiliser of a line in $\mathbb{P}^3$ defined over $\mathbb{Z}/p\mathbb{Z}$,
\item\label{it:aleg4} $H$ has an abelian subgroup of index $\leq 6$,
\item\label{it:aleg3} $H$ is contained in a subgroup of $G$ isomorphic to $\SO_3(\mathbb{Z}/p\mathbb{Z})$,
\item\label{it:aleg5} $H$ is of order $\leq 1080$.
\end{enumerate}
\end{cor}
\begin{proof}
Let $M$ be a maximal subgroup of $G=\SL_3(\mathbb{Z}/p\mathbb{Z})$ containing $H$.
Let $\overline{M}$ be the image of $M$ under the natural map $\pi:\SL_3(\mathbb{Z}/p\mathbb{Z})
\to \PSL_3(\mathbb{Z}/p\mathbb{Z})$. If $\overline{M}$ were not a proper subgroup of 
$\PSL_3(\mathbb{Z}/p\mathbb{Z})$, then $M$ would have index $3$ in $G$.
The action of $G$ on cosets of $M$ would induce a non-trivial homomorphism $\phi$ from
$G$ to the symmetric group $S_3$. The kernel $\ker(\phi)$ of that homomorphism would be a 
proper normal subgroup
of $G$ of index at most $6$. Now, $G/Z(G) = \SL_3(\mathbb{Z}/p\mathbb{Z})/
Z(\SL_3(\mathbb{Z}/p\mathbb{Z}))$ is simple, and so $\ker(\phi)$ would have to be contained in 
$Z(G)$. Since $Z(G)$ has at most $3$ elements, it would follow that $G$ has at most
$6\cdot 3=18$ elements. This is clearly false. Thus, $\overline{M}$ is a proper subgroup of
$\PSL_3(\mathbb{Z}/p\mathbb{Z})$.

Moreover, $\overline{M}$ is a maximal subgroup of
$\PSL_3(\mathbb{Z}/p\mathbb{Z})$, as otherwise $M$ would not be maximal in
$G=SL_3(\mathbb{Z}/p\mathbb{Z})$.
Now apply Prop.\ \ref{prop:mitch}. 

If $\overline{M}$ is the stabiliser of a line, then $M$ is contained in the
stabiliser in $G=SL_3(\mathbb{Z}/p\mathbb{Z})$ of a line. (The action of $G$ on 
$\mathbb{P}^3$ factors through $PSL_3(\mathbb{Z}/p\mathbb{Z})$.) 
If $\overline{M}$ is the stabiliser of a point, then $M$ is contained in the
stabiliser of a point. This takes care of cases (\ref{it:yor1}) and (\ref{it:yor1b}) of Prop.\ 
\ref{prop:mitch}.

Suppose now that we are in case (\ref{it:yor2}) of Prop.\ \ref{prop:mitch}.
Since $\overline{M}$ is the stabiliser of a set of three points, $M$ is contained in
the stabiliser of a set of three points.
The stabiliser in $G$ of a set of three points in 
$\mathbb{P}^3(\overline{\mathbb{Z}/p\mathbb{Z}})$ is equal to the semidirect product of the points
over $\mathbb{Z}/p\mathbb{Z}$ of a torus $T$ in $G$ (defined over $\overline{\mathbb{Z}/p\mathbb{Z}}$) and
the elements of $G$ that induce elements of the Weyl group of the torus. Since the Weyl group
of a torus in $\SL_3$ has index $6$, we see that
the group $M$ must have an abelian subgroup of index $\leq 6$, and thus $H$ itself has an abelian subgroup
of index $\leq 6$. We have obtained conclusion (\ref{it:aleg4}).

Suppose that we are in case (\ref{it:yor3}) of Prop.\ \ref{prop:mitch}. The conic in 
question is given by an equation $Q(v)=0$, where $Q$ is some non-degenerate quadratic form. 
The group $G_Q$ of all elements $g\in G$ such that $Q(gv)=Q(v)$ is isomorphic to 
$\SO_3(\mathbb{Z}/p\mathbb{Z})$ (\cite[Prop.\ 2.5.4]{KL}). The group $G_Q$ is
certainly contained in the stabiliser of $Q(v)=0$. Comparing orders (where
the order of the stabiliser of a conic $Q(v)=0$ is given by \cite[Thm.\ 2.4]{Ki}) we see that $G_Q$ is actually equal to the stabiliser of $Q(v)=0$.

Finally, case (\ref{it:yor4}) of Prop.\ \ref{prop:mitch} corresponds to case (\ref{it:aleg5}) here, and so
we are done.
\end{proof}

Let us see what we can say about each of the cases of Cor.\ \ref{cor:odious}. 

For groups of bounded order, the statement  $|A\cdot A\cdot A|\gg |A|^{1+\delta}$ is trivially true
(as one may adjust $\delta$ and the implied constant if needed). Thus, we may ignore case
(\ref{it:aleg5}). As for case (\ref{it:aleg1}), it reduces to case (\ref{it:aleg2}):
the stabiliser in $G$
of a point in $\mathbb{P}^3(\mathbb{Z}/p\mathbb{Z})$ is always conjugate (and hence isomorphic) to the subgroup
\begin{equation}\label{eq:naug}\left\{g=\left(\begin{matrix}* & * & *\\ 0 & * & *\\ 0 & * & * \end{matrix}\right) : \det(g)=1 \right\}\end{equation}
of $G$, whereas the stabiliser in $G$ of a line in $\mathbb{P}^3$ defined over 
$\mathbb{Z}/p \mathbb{Z}$
is always conjugate (and hence isomorphic) to the subgroup
\begin{equation}\label{eq:ghty}\left\{g=\left(\begin{matrix}* & * & *\\ * & * & *\\ 0 & 0 & * \end{matrix}\right) : \det(g)=1 \right\}\end{equation}
of $G$. The subgroups (\ref{eq:naug}) and (\ref{eq:ghty})
 are isomorphic as groups. (They and their conjugates are called the
{\em maximal parabolic subgroups} of $G$.)
Thus, case
(\ref{it:aleg1}) and case (\ref{it:aleg2}) are essentially the same.

We hence have three cases to study: 
(1) subgroups
of $\SO_3(\mathbb{Z}/p\mathbb{Z})$ (case (\ref{it:aleg3}) in Cor.\ \ref{cor:odious}); (2) 
subgroups of $G$ having abelian subgroups of small index (case (\ref{it:aleg4}) in Cor.\
\ref{cor:odious});
(3) subgroups of maximal parabolic subgroups of 
$G=\SL_3(\mathbb{Z}/p\mathbb{Z})$ (that is, subgroups
of stabilisers of points and lines, i.e., cases (\ref{it:aleg1}) and 
(\ref{it:aleg2}) in Cor.\ \ref{cor:odious}).
Let us consider them in order.

(1) {\em The group $\SO_3(\mathbb{Z}/p\mathbb{Z})\sim \PGL_2(\mathbb{Z}/p\mathbb{Z})$.}\\
As it happens, $\SO_3(\mathbb{Z}/p\mathbb{Z})$ is isomorphic as a group to 
$\PGL_2(\mathbb{Z}/p\mathbb{Z})$ (\cite[Thm.\ 11.6]{Ta}). 
We will conclude our study of growth in $\SL_2(\mathbb{Z}/p\mathbb{Z})$,
and then use the fact that 
$\PGL_2(\mathbb{Z}/p\mathbb{Z})$ has a subgroup of index $2$ isomorphic
to 
$\SL_2(\mathbb{Z}/p\mathbb{Z})/Z(\SL_2(\mathbb{Z}/p\mathbb{Z}))$.


(2) {\em Subgroups of $G=\SL_3(\mathbb{Z}/p\mathbb{Z})$ 
having abelian subgroups of small index.}

This is a different kettle of fish. Some subsets of abelian groups grow and others do not. (This matter
is the classical object of study of additive combinatorics.) A great deal has been said on this general subject,
but very little is known on the question of which subsets of abelian groups grow truly rapidly ($|A\cdot A\cdot A|\gg |A|^{1+\delta}$). All we know is which sets grow very slowly 
($|A\cdot A\cdot A|\ll (\log |A|)^{1/3} |A|$, say); this is Freiman's theorem,
generalised to arbitrary abelian groups by Green and Ruzsa \cite{GR}).

We will not attempt to improve on this; we will do no more than set aside the
abelian case whenever we come across it.


(3) {\em Subgroups of maximal parabolic subgroups of 
$G=SL_3(\mathbb{Z}/p\mathbb{Z})$.}

These are the groups isomorphic to (\ref{eq:naug}) and (\ref{eq:ghty}).
They are the main subject of this section (\S \ref{subs:playdirt} --
\S \ref{subs:thalion}). A subset $A$ of a maximal parabolic subgroup of $G$
may or may not be contained in a Borel subgroup of $G$. Growth in Borel
subgroups is closely related to Prop.\ \ref{prop:guggen},
 i.e., to generalised sum-product phenomena. If a subset $A$ of a parabolic
subgroup is not contained in a Borel subgroup, the study of its growth amounts
more or less to the study of growth in $\SL_2$ plus a little cohomology
(\ref{sec:indtep}). 

\subsection{Growth in subgroups of $\SL_2(\mathbb{F}_p)$ and
  $\SO_3(\mathbb{F}_p) \sim \PGL_2(\mathbb{F}_p)$}\label{subs:rapanui}

The classification of the proper subgroups of $\SL_2(\mathbb{Z}/p\mathbb{Z})$
is classical.

\begin{prop}\label{prop:dick}
 Let $K = \mathbb{Z}/p\mathbb{Z}$.
 Let $G = \SL_2(K)$.
 Let $H$ be a proper subgroup of $G$ with more than $120$ elements.
Then either
\begin{enumerate}
\item\label{it:frodo} $H$ is contained in a Borel subgroup $B$ of $G$ defined over $K$, or
\item there is a maximal torus $T/\overline{K}$ such that 
$H\leq N_{G(K)}(T(K))$.
\end{enumerate}
\end{prop}
If $T$ is defined over $K$, then the normaliser $N_{G(K)}(T(K))$ is a
dihedral group containing $T(K)$ as a subgroup of index $\leq 2$.
If $T$ is not defined over $K$, then $N_{G(K)}(T(K)) = T(K)$.
\begin{proof}
See \cite{Di}, p.\ 286.
\end{proof}

We now need to do very little work given what we already did in \S \ref{subs:grotes}.
\begin{thm}\label{thm:agora}
Let $G = \SL_2$.
Let $K=\mathbb{Z}/p\mathbb{Z}$, $p$ a prime. 
Let $A\subset G(K)$.

Then, for every $\epsilon>0$, either
\begin{equation}\label{eq:frond}
|A \cdot A\cdot A|\gg |A|^{1+\delta},\end{equation} where $\delta>0$ and
the implied constant depend only on $\epsilon$, or one of the following cases holds:
\begin{enumerate}
\item\label{it:caraj0} $A$ generates $G(K)$ and $|A|> |G(K)|^{1-\epsilon}$, or
\item\label{it:caraj1} there is a maximal torus $T/\overline{K}$ such that 
$H\leq N_{G(K)}(T(K))$, or
\item\label{it:caraj2} there is a Borel subgroup $B/K$ such that $A\subset B(K)$, such
that either
\begin{enumerate}
\item\label{it:caraj2a} $|A_6\cap (\{\pm I\}\cdot  U(K))|\geq |A|^{1-\epsilon}$ (
where $U/K$ is the maximal unipotent subgroup of $B$) and $A$ intersects
at most $|A|^{2\epsilon}$ cosets of $U(K)$, or
\item\label{it:caraj2b} $A_k$ contains $U(K)$ for some $k$ depending only on $\epsilon$.
\end{enumerate}
\end{enumerate}
\end{thm}
\begin{proof}
If $A$ generates $G(K)$, then, by Proposition \ref{prop:baggage}, either
(\ref{eq:frond}) or conclusion (\ref{it:caraj0}) holds. Assume, then, that
$A$ does not generate $G(K)$.

Thanks to the classification of the proper subgroups of $G(K)$ (Prop.\ \ref{prop:dick}),
either conclusion (\ref{it:caraj1}) holds or $A$ is contained in $B(K)$, where $B/K$ is a Borel subgroup of
$G$. In the latter case, we apply Prop.\ \ref{prop:stick}.  If case (\ref{it:cogot4}) in Prop.\ \ref{prop:stick} holds, then (\ref{eq:frond}) follows by the tripling lemma
(Lem.\ \ref{lem:furcht}).

(We use the fact that we can assume that $|A|$
is larger than an absolute constant, as otherwise (\ref{eq:frond}) holds trivially;
this allows us, for example, to do without a factor of $\frac{1}{4}$ in front of
$|A|^{1-\epsilon}$ when deriving conclusion \ref{it:caraj2a} from case (\ref{it:cogot2})
of Prop.\ \ref{prop:stick}.)
\end{proof}

One may ask how tight Thm.\ \ref{thm:agora} is. There are examples of sets $A$ falling
into one of the cases \ref{it:caraj1},
\ref{it:caraj2a}, \ref{it:caraj2b} in Thm.\ \ref{thm:agora} and failing to grow
(i.e., failing to satisfy (\ref{eq:frond})). To wit --\\

Case \ref{it:caraj1}, example 1: 
 Let
\[A = \left\{\left(\begin{matrix}x^n & 0\\ 0 & x^{-n}\end{matrix}\right) : 1\leq n\leq N
 \right\},\]
where $x$ is a generator of $(\mathbb{Z}/p\mathbb{Z})^*$ and $N\leq p-1$. Then 
$|A| = N$ and $|A\cdot A\cdot A| < 3 N = 3 |A|$.\\

Case \ref{it:caraj1}, example 2:
 Let
\[
A = \left\{\left(\begin{matrix}x^n & 0\\ 0 & x^{-n}\end{matrix}\right) : -N\leq n\leq N
 \right\} \cup \left\{\left(\begin{matrix}0 & x^n\\ x^{-n} & 0\end{matrix}\right) : -N\leq n\leq N \right\},\]
where $x$ is a generator of $(\mathbb{Z}/p\mathbb{Z})^*$ and $N\leq (p-1)/2$. Then 
$|A| = 4 N+2$ and $|A\cdot A\cdot A| < 2\cdot (6 N + 1) < 3 |A|$.\\

Case \ref{it:caraj2a}:
Let
\[A = \left\{\left(\begin{matrix}n & m\\ 0 & n^{-1}\end{matrix}\right) : 1\leq n\leq N^{\epsilon},\;
1\leq m\leq N \right\}.\]
(Here $n^{-1}$ stands for inverse of $n\mo p$.)
Then $|A| \sim N^{1+\epsilon}$ and $|A\cdot A\cdot A| \ll N^{1+9\epsilon}$.\\

Case \ref{it:caraj2b}: 
Let
\[A = \left\{\left(\begin{matrix}x^n & m\\ 0 & x^{-n}\end{matrix}\right) : 1\leq n\leq N,\; m\in \mathbb{Z}/p\mathbb{Z} \right\},\]
where $x$ is a generator of $(\mathbb{Z}/p\mathbb{Z})^*$ and $N\leq p-1$.
Then $|A| = p N$ and $|A\cdot A\cdot A| < 3 p N = 3 |A|$.\\

We can rewrite the conclusions of Thm.\ \ref{thm:agora} so that
it looks more like what a general statement on all groups would be likely to look like.
(See the remarks after Thm.\ \ref{thm:qartay}.)
\begin{cor}[to Theorem \ref{thm:agora}]\label{cor:ostor}
Let $G = \SL_2$.
Let $K=\mathbb{Z}/p\mathbb{Z}$, $p$ a prime. 
Let $A\subset G(K)$.

Then, for every $\epsilon>0$, either
\begin{equation}\label{eq:car}
|A \cdot A\cdot A|\gg |A|^{1+\delta},\end{equation} where $\delta>0$ and
the implied constant depend only on $\epsilon$, or 
there are normal subgroups $H_1, H_2\triangleleft \langle A\rangle$,
$H_1<H_2$ such that 
\begin{enumerate}
\item $H_2/H_1$ is abelian,
\item $A_k$ contains $H_1$, where $k$ depends only on $\epsilon$, and
\item $A$ is contained in the union of $\leq |A|^{\epsilon}$ cosets of $H_2$.
\end{enumerate}
\end{cor}
In other groups,``abelian'' would be replaced by ``nilpotent'' (as in
the statement of Theorem \ref{thm:qartay}). We have ``abelian'' here simply because
there is not much room for non-abelian nilpotent groups in $\SL_2$.
\begin{proof}
Apply Thm.\ \ref{thm:agora} (with $\epsilon/2$ instead of $\epsilon$).
Equation (\ref{eq:frond}) in Thm.\ \ref{thm:agora}
is equation (\ref{eq:car}) here. By the Key Proposition (part (b)) in
\cite[\S 1]{He}, case (\ref{it:caraj0}) in Thm.\ \ref{thm:agora}
implies that $A_k$ contains $H_1$
and is contained in $H_2$, where $H_1=H_2=G(K)$.
Case (\ref{it:caraj1}) in Thm.\ \ref{thm:agora} gives us that $A$
is contained in the union of $\leq 2$ cosets of
$H_2 = T(K) \triangleleft \langle A\rangle$; since $H_2$ is abelian, we can set
$H_1=\{e\}$. Case \ref{it:caraj2a} in Thm. 
\ref{thm:agora} gives us that $A$ is contained in few subsets of
the abelian group $H_2 = (\{\pm I\}\cdot  U(K)) \triangleleft B(K)$; again, we set
$H_1=\{e\}$.  Finally, case \ref{it:caraj2b} tells us that $A$ contains $H_1=U(K)$ and
is contained in $H_2=B(K)$; $H_1$ is a normal subgroup of $H_2$, $H_2$
is a normal subgroup of $\langle A\rangle = H_2$, and $H_2/H_1$ is abelian.
\end{proof}



\begin{cor}[to Corollary \ref{cor:ostor}]\label{cor:mororga}
Let $G = \PGL_2$, $G=\SO_3$ or $G = \PSL_2$.
Let $K=\mathbb{Z}/p\mathbb{Z}$, $p$ a prime. 
Let $A\subset G(K)$.

Then, for every $\epsilon>0$, either
\begin{equation}\label{eq:carto}
|A \cdot A\cdot A|\gg |A|^{1+\delta},\end{equation} where $\delta>0$ and
the implied constant depend only on $\epsilon$, or 
there are subgroups $H_1\triangleleft H_2\triangleleft \langle A\rangle$ such that 
\begin{enumerate}
\item $H_2/H_1$ is abelian,
\item $A_k$ contains $H_1$, where $k$ depends only on $\epsilon$, and
\item $A$ is contained in the union of $\leq |A|^{\epsilon}$ cosets of $H_2$.
\end{enumerate}
\end{cor}
\begin{proof}
Since $\SO_3(\mathbb{Z}/p\mathbb{Z})$ and $\PGL_2(\mathbb{Z}/p\mathbb{Z})$
are isomorphic as groups, it is enough to prove the statement for
$G = \PGL_2$ or $G = \PSL_2$. Since $\PSL_2(K)<\PGL_2(K)$ and
$\lbrack \PGL_2(K):\PSL_2(K)\rbrack = 2$, Prop.\ \ref{prop:amery} implies it
is enough to prove the statement for $G = \PSL_2$. Let, then, $G = \PSL_2$ and
$A \subset \PSL_2(K)$.


Let $\pi:\SL_2(K)\to \PSL_2(K)$ be the natural projection map,
and let $A' = \pi^{-1}(A) \subset \SL_2(K)$. 
Apply Corollary \ref{cor:ostor} to $A'$. Clearly (\ref{eq:car})
implies (\ref{eq:carto}) (with the implied constant changing by a factor
of at most $2$). If (\ref{eq:car}) does not hold, then
Corollary \ref{cor:ostor} provides subgroups
 $H_1'\triangleleft H_2'\triangleleft \langle A'\rangle$; we use them
to define subgroups $H_1 = \pi(H_1')$,
$H_2 = \pi(H_2')$ satisfying the properties in the statement of the 
corollary we are proving. 
\end{proof}

\subsection{Parabolic subgroups of $\SL_3(\mathbb{Z}/p\mathbb{Z})$:
general setup} \label{subs:playdirt}
Let $K = \mathbb{Z}/p\mathbb{Z}$,
$G = \SL_3(K)$. Let $e_1, e_2, e_3\in K^3$ be a basis of $K^3$.
Let $P\subset G$
 be the stabiliser of the subspace $K e_1 + K e_2$ of $K^3$ under the natural
action of $G$ in $K^3$. (This is the same as the stabiliser of $K e_1 + K e_2$
seen as a line in $\mathbb{P}^3(K)$; we prefer to use affine rather than
projective language.)
Let $H_0$ be the group consisting
of the elements $g\in P(K)$ sending $e_3$ to elements of the form
$a_1 e_1 + a_2 e_2 + a_3 e_3$, with $a_1,a_2\in K$
 and $a_3\in K^*$ a square in $K^*$.
Let $M$ be the subgroup of $H_0$ consisting of the elements $g\in P(K)$ sending $e_3$ to elements $a_1 e_1 + a_2 e_2 + e_3$ with $a_1, a_2\in K$.
 Let
\begin{itemize}
\item $G_+$ be the subgroup of $H_0$
consisting of all $g\in H_0$ fixing the space $K e_3$, 
\item $G_-$ be the subgroup of $M$
consisting of all $g\in M$ fixing $e_3$,
\item $A_0$ be the subgroup of $M$
consisting of all $g\in M$ fixing both $e_1$ and $e_2$, and
\item $Z(G_+)$ be the center of $G_+$,
\item $\pi_+:H_0\to G_+$, $\pi_-:M\to G_-$ be the natural projections.
\end{itemize}
More legibly, in matrix form (with $e_1$, $e_2$, $e_3$ as the basis),
\begin{equation}\label{eq:onestar}\begin{aligned}
H_0 &= \left\{ g = \left(\begin{matrix} \ast & \ast & \ast\\
\ast & \ast &\ast\\0 & 0 &s^2 \end{matrix}\right) : \det(g)=1, s\in K^*
\right\},\\
G_+ &= \left\{ g = \left(\begin{matrix} \ast & \ast & 0\\
\ast & \ast &0\\0 & 0 &s^2 \end{matrix}\right) : \det(g)=1, s\in K^*
\right\},\\
M &= \left\{ g = \left(\begin{matrix} a & b & \ast\\
c & d &\ast\\0 & 0 &1 \end{matrix}\right) :
\det\left(\begin{matrix}a &b\\c &d\end{matrix}\right)=1\right\},\\
G_- &= \left\{ g = \left(\begin{matrix} a & b & 0\\
c & d &0\\0 & 0 &1 \end{matrix}\right) :
\det\left(\begin{matrix}a &b\\c &d\end{matrix}\right)=1\right\},\\
Z(G_+) &= \left\{ g = \left(\begin{matrix} s^{-1} & 0 & 0\\
0 & s^{-1} &
0\\0 & 0 &s^{2} \end{matrix}\right) : s \in (\mathbb{Z}/p\mathbb{Z})^*
 \right\},\\
A_0 &= \left\{\left(\begin{matrix} 1 & 0 & e\\
0 & 1 &f\\0 & 0 &1\end{matrix}\right)\right\},\end{aligned}\end{equation}
where all entries are understood to lie in $K$. The projections
$\pi_+:H_0\to G_+$, $\pi_-:M\to G_-$ are given by
\begin{equation}\label{eq:twostar}
\pi_+\left(\begin{matrix} a & b &e\\ c & d & f\\
0 & 0 &s^2\end{matrix}\right) = \left(\begin{matrix} a&b &0\\
c & d &0\\0 &0 &s^2\end{matrix}\right),\;\;\;\;\text{and}\;\;\;\;
\pi_-\left(\begin{matrix} a & b &e\\ c & d & f\\
0 & 0 &1\end{matrix}\right) = \left(\begin{matrix} a&b &0\\
c & d &0\\0 &0 &1\end{matrix}\right).
\end{equation}
It is clear that $G_- \simeq \SL_2(K)$ and $A_0\simeq K^2$. Moreover, $A_0$
is a normal subgroup of $M$; the projection $\pi_-:M\to G_-$ can be identified
with the quotient homomorphism $M\to A_0\backslash M \simeq G_-$.
(Here we write $A_0\backslash M$ for the group of right cosets of $A_0$ in $M$.)
 We can thus see $M$
as a semidirect product $A_0\rtimes G_-$ of $A_0$ and $G_-$. The action
of $G_-$ on $A_0$ in the semidirect product $M = A_0\rtimes G_-$
is the natural one, as is shown
by the identity
\[
\left(\begin{matrix} a & b & 0 \\ c & d & 0\\ 0& 0& 1\end{matrix}\right) \cdot
\left(\begin{matrix} 1 & 0 & e \\ 0 & 1 & f\\ 0 &0& 1 \end{matrix}\right) \cdot
\left(\begin{matrix} a & b & 0 \\ c& d& 0 \\ 0 & 0 &1 \end{matrix}\right)^{-1} =
\left(\begin{matrix} 1 & 0 & a e + b f\\
0 & 1 & c e + d f\\ 0 & 0 & 1\end{matrix}\right) =
\left(\begin{matrix} I & \left(\begin{matrix} a & b\\c & d\end{matrix}\right)
\cdot \left(\begin{matrix} e\\ f\end{matrix}\right) \\0 & 1\end{matrix}\right).
\]
In other words, we may write the elements of $M$ as pairs $(a,g)$,
$a\in A_0$, $g\in G_-$, and then the group law of $M$ looks as follows:
\[(a_1,g_1)\cdot(a_2,g_2) = (a_1 + g_1 \cdot a_2, g_1 g_2),
\]
where $g_1\in G_- \simeq \SL_2(K)$ acts on $a_2\in A_0\simeq K^2$ by the natural
action of $\SL_2(K)$ on $K^2$.

We can also decompose $G_+$ as a product, namely, $G_+ \simeq \SL_2(K)
\times \{x^2 : x\in K^*\}$.
We let the projection maps $\pi_1:G_+\to \SL_2(K)$,
$\pi_2:G_+\to K^*$ be given by
\begin{equation}\label{eq:threestar}
\pi_1\left(\begin{matrix} a & b &0\\c & d &0\\0 &0 & s^2\end{matrix}\right) =
\left(\begin{matrix} s a &s b\\ s c & s d\end{matrix}\right) \in SL_2(K),\;\;\;
\pi_2\left(\begin{matrix} a & b &0\\c & d &0\\0 &0 & s^2\end{matrix}\right)
= s^2.\end{equation}

The above setup will be somewhat familiar to some readers from the theory of
automorphic forms. (The group $M$ is of a kind called
 {\em mirabolic} by some; the decomposition $M = A_0\rtimes G_-$ treated
above is
well-known in general.)

In the following, we shall examine a subset $E$ of $H_0$, and determine
its growth. (We call our set $E$ rather than $A$ so as to avoid confusion with the group $A_0$,
which is usually called $A$ in the literature.) Since
$\pi_1(\pi_+(E))$ is a subset of $\SL_2(K)$, it generates a subgroup
of $\SL_2(K)$. This subgroup can be all of $\SL_2(K)$, or
it can lie inside one of the maximal subgroups of
$\SL_2(K)$ (which were classified in Prop.\ \ref{prop:dick}). We treat these two cases individually.

\subsection{Parabolic subgroups: passage to $\SL_2$, or
the case $\langle \pi_1(\pi_+(E))\rangle = \SL_2(K)$}\label{sec:indtep}

\begin{prop}\label{prop:shri}
Let $K = \mathbb{Z}/p\mathbb{Z}$, $p$ a prime. Let $G = \SL_3$.
Let $G_+$, $G_-$ and $H_0$
be as in (\ref{eq:onestar}); let $\pi_+:H_0\to G_+$ and $\pi_1:G_+\to
\SL_2(K)$ be as in (\ref{eq:twostar}) and (\ref{eq:threestar}).

Let $E$ be a subset of $H_0$ such that $\pi_1(\pi_+(E))$ generates $\SL_2(K)$.
Then either
\begin{equation}\label{eq:ogrolo}
|E\cdot E\cdot E| > |\pi_1(\pi_+(E))|^{\epsilon} \cdot |E|\end{equation}
or
\begin{equation}\label{eq:nurdo}\pi_1(\pi_+(E_k)) = \SL_2(K),\end{equation}
where $\epsilon>0$ and $k$ are absolute constants.
\end{prop}
\begin{proof}
It is here that the inductive step happens; we will use what we know
on $\SL_2$.
By the Key Proposition in \cite[\S 1]{He}, there are absolute constants
$\delta, \epsilon, k>0$ such that, for $A\subset \SL_2(K)$
generating $\SL_2(K)$, we have two cases:
\begin{itemize}
\item If $|A|\leq |\SL_2(K)|^{1-\delta}$, then
$|A\cdot A\cdot A|>|A|^{1+\epsilon}$.
\item If $|A|> |\SL_2(K)|^{1-\delta}$, then $A_k = \SL_2(K)$.
\end{itemize}

Now define $\pi = \pi_1 \circ \pi_+$. By the statement of the lemma,
$\pi(E)$ generates $\SL_2(K)$. If $|\pi(E)|> |\SL_2(K)|^{1-\delta}$, then
\[\pi(E_k) = \pi(E)_k = \SL_2(K),\]
and we are done. Suppose $|\pi(E)|\leq |\SL_2(K)|^{1-\delta}$.
Then
\[|\pi(E \cdot E \cdot E)| = |\pi(E)\cdot \pi(E) \cdot \pi(E)|>
|\pi(E)|^{1 + \epsilon},\]
and so, by Lemma \ref{lem:quotgro} (applied with $A_1 = E$, $A_2 = E\cdot E$),
\[|E_8| > |\pi(E)|^{\epsilon} \cdot |E|.\]
Statement (\ref{eq:ogrolo}) then follows by the tripling lemma (Lemma \ref{lem:furcht}).
\end{proof}

If we have (\ref{eq:ogrolo})
and $|\pi_1(\pi_+(E))|>|E|^{\delta}$ for some fixed $\delta>0$, the problem
is solved. We will leave the case of
$|\pi_1(\pi_+(E))|\leq |E|^{\delta}$ for later. We focus for now
on (\ref{eq:nurdo}), i.e., on the case of sets
$E$ with $\pi_1(\pi_+(E)) = \SL_2(K)$. (The case
of sets $E$ with $\pi_1(\pi_+(E_k)) = \SL_2(K)$ reduces to this after
we multiply $E$ with itself and its inverse a few times.)

We will need the following result, credited by Dickson to Galois.
\begin{prop}[Galois]\label{prop:galois}
Let $p>11$ be a prime. Let $G$ be $\SL_2(\mathbb{Z}/p\mathbb{Z})$
or $\PSL_2(\mathbb{Z}/p\mathbb{Z})$. Let $H$ be a proper subgroup of
$G$. Then $\lbrack G:H\rbrack \geq p+1$.
\end{prop}
This can be derived quickly from Prop.\ \ref{prop:dick}.
\begin{proof}
See, e.g.,  \cite{Di}, e.g., Ch.\ XII, Theorem 261.
\end{proof}

We can now proceed.
\begin{lem}\label{lem:adadar}
Let $K = \mathbb{Z}/p\mathbb{Z}$, $p>11$ a prime. Let $G$, $H_0$ and $M$
be as in (\ref{eq:onestar});
let $\pi_+:H_0\to G_+$, $\pi_-: M\to G_-$ and $\pi_1:G_+\to
\SL_2(K)$ be as in (\ref{eq:twostar}) and (\ref{eq:threestar}).

Let $E$ be a subset of $H_0$ such that $\pi_1(\pi_+(E)) = \SL_2(K)$.
Then
\[
\pi_-(E_k\cap M) = G_-,
\]
where $k$ is an absolute constant.
\end{lem}
\begin{proof}
Since $\pi_1(\pi_+(E)) = \SL_2(K)$, we know that $|\pi_+(E)| \geq
|\SL_2(K)|$. Let \[R = (\pi_+(E)^{-1} \pi_+(E)) \cap G_- .\] Then, by Lemma
\ref{lem:duffy},
\begin{equation}\label{eq:audo}
|R| \geq \frac{|\pi_+(E)|}{|\lbrack G_+ : G_-\rbrack|} \geq
\frac{|\SL_2(K)|}{(p-1)/2} = \frac{|G_-|}{(p-1)/2} .\end{equation}
By Prop.\ \ref{prop:galois}, $G_- \simeq \SL_2(K)$ has no
proper subgroups of index
$\leq \frac{p-1}{2}$; hence $R$ generates $G_-$. Now, (\ref{eq:audo}) also gives us that
$|R| > |G_-|^{2/3}$. We apply the Key Proposition in \cite{He} (part (b)) and
obtain that
\[ R_k = G_-,\]
where $k$ is an absolute constant.

Since $R = \pi_+(E^{-1} E)\cap G_- \subset \pi_+(E_2) \cap G_-$, it follows
that
\[\pi_-(E_{2 k} \cap M) \supseteq \pi_-((E_2 \cap M)_k) =
(\pi_-(E_2 \cap M))_k = (\pi_+(E_2) \cap G_-)_k \supseteq R_k = G_-,\]
as we desired.
\end{proof}

We need a little lemma on cohomology.
\begin{lem}\label{lem:cohom}
Let $G$ be a group acting on an abelian group $R$. Suppose the centre $Z$
of $G$ contains an element $z\in Z$ with $z^2 = e$ such that
$z v = - v$ for all $v\in R$.
Suppose furthermore that every
element of $R$ is uniquely $2$-divisible, i.e., 
suppose that, for every $r\in R$, there is
a unique $r'\in R$ such that $r = 2 r'$.

Then
\begin{equation}\label{eq:gotoro}H^1(G,R)=0.\end{equation}
\end{lem}
We can restate (\ref{eq:gotoro}) in non-cohomological language as follows:
given any map $s:G\to R$ satisfying
$s(g_1 g_2) = s(g_1) + g_1 s(g_2)$ for all $g_1, g_2\in G$, there is
a $v$ such that $s(g) = g v - v$ for all $g$.

One can show \cite{Hi} that $H^n(G,R)=0$, $n\geq 1$, under the same conditions
we have given; 
we shall need only the case $n=1$. The conditions of the lemma
are clearly satisfied when $G = \SL_2(K)$, $R = K^2$, $K$ a finite field
of odd order: set $z= -I$.
\begin{proof}
Let $g\in G$. Because $z$ is in the centre, $z$ and $g$ commute; we
also have that $z$ is an involution, i.e., $z^2 = e$. Thus
\[\begin{aligned}
s(g) &= s(g \cdot z^2) = s(z \cdot g \cdot z) = 
s(z) + z\cdot s(g\cdot z) = 
s(z) - s(g \cdot z)\\
&= s(z) - (s(g) + g\cdot s(z)) = - s(g) + s(z) - g\cdot s(z).
\end{aligned}\]
Thus
\[s(g) = \frac{1}{2}(s(z) - g\cdot s(z)).\]
So
\[s(g) = g v - v\]
for $v = -\frac{1}{2} s(z)$.
\end{proof}

\begin{lem}\label{lem:achajar}
Let $K = \mathbb{Z}/p\mathbb{Z}$, where $p$ is an odd prime. Let
$M$, $G_-$ and $A_0$ be as in (\ref{eq:onestar}). Let $\pi_-:M\to G_-$ be as in
(\ref{eq:twostar}).

Let $E\subset M$ be such that $\pi_-(E) = G_-$. Then either
\[E_k = M,\]
where $k$ is an absolute constant, or
\[E = g G_- g^{-1}\;\;\;\;\;\;\;\;\;\;\; \text{for some $g\in M$.}\]
\end{lem}
\begin{proof}
Suppose first that there are two distinct $g_1,g_2\in E$ such that
$\pi_-(g_1) = \pi_-(g_2)$. Then $g_1^{-1} g_2$ is an element of $A_0$ other than
$I$. Since $\SL_2(K)$ acts transitively on the set of non-zero elements of $K^2$, we have that $G_-$ acts
transitively on $A_0$ by conjugation. (Recall that $M\simeq A_0\rtimes G_-$,
$A_0 \simeq K^2$, $G_-\simeq \SL_2(K)$, and that the action of $G_-$ on $A_0$ by
conjugation is described by the action of $\SL_2(K)$ on $K^2$.)
 Since $\pi_-(E) = G_-$, it
follows that
$E g_1^{-1} g_2 E^{-1} \subset M$ is all of $A_0$, and so
\[E g_1^{-1} g_2 E^{-1} E \subset E_5\]
is equal to all of $M$.

Now suppose that there are no two distinct $g_1,g_2 \in E$ with
$\pi_-(g_1) = \pi_-(g_2)$. Then $E$ is of the form
$\{(s(h),h): h\in G_-\}$, where $s$ is a map $s:G_-\to A_0$.
If there are $h_1, h_2 \in G_-$ such that
\[(s(h_1),h_1) \cdot (s(h_2),h_2) \ne (s(h_1 h_2), h_1 h_2),\]
then the argument is as before: there are two distinct elements
(namely, $(s(h_1),h_1) \cdot (s(h_2),h_2)$ and
$(s(h_1 h_2), h_1 h_2)$) whose image $h_1 h_2$ under $\pi_-$ is the same,
and so
\[E_{10} = M.\]

Suppose, then, that
\[(s(h_1),h_1) \cdot (s(h_2),h_2) = (s(h_1 h_2), h_1 h_2)\]
for all $h_1, h_2\in G_-$, or what is the same,
\[s(h_1 h_2) = s(h_1) + h_1 s(h_2)\]
for all $h_1, h_2\in G_-$. We now use Lemma \ref{lem:cohom}, and conclude
that $s(h) = h v - v$ for some $v\in A_0$. Hence
\[\begin{aligned}
E = \{(s(h),h):h\in G_-\} &= (-v,1) \cdot \{(0,h):h\in G_-\} \cdot (v,1)\\
&= g \cdot G_-\cdot g^{-1},\end{aligned}\]
where $g$ is the element of $M$ corresponding to $(-v,1)$ under the
isomorphism $M\simeq A_0 \rtimes G_-$.
\end{proof}

We can now draw certain conclusions.
\begin{prop}\label{prop:ralder}
Let $K = \mathbb{Z}/p\mathbb{Z}$, $p$ a prime. Let $G = \SL_3$.
Let $H_0<G(K)$ and $G_+<G(K)$ be as in (\ref{eq:onestar}); let $\pi_+$ and $\pi_1$ be as in
(\ref{eq:twostar}) and (\ref{eq:threestar}).

Let $E\subset H_0$ be such that
$\pi_1(\pi_+(E))$ generates $\SL_2(K)$. Then, for every $\epsilon>0$, either
\begin{enumerate}
\item\label{it:kat1} $|E\cdot E\cdot E| > |E|^{1 + \delta}$, where
$\delta>0$ depends only on $\epsilon$, or
\item\label{it:kat2} $E_k$ contains $M$, where $k$ is an absolute constant, or
\item\label{it:kat3} $E_k$ contains $g G_- g^{-1}$ and is contained in 
$g G_+ g^{-1}$, where $g\in M$ and $k$ is an absolute constant,
\item\label{it:keato} $E^{-1} E$ has $\geq |E|^{1-\epsilon}$ elements
in the subgroup $H' = \pi_+^{-1}(Z(G_+))$ of $H_0$;
moreover, $E$ intersects at most $|E|^{2\epsilon}$ cosets of $H_0$.
\end{enumerate}
\end{prop}
Here (\ref{it:keato}) is in effect a reduction to one of the cases to be
treated in the next subsection (Lem.\ \ref{lem:spolga}). We will treat it further there.
\begin{proof}
Suppose $|\pi_1(\pi_+(E))|\leq |E|^{\epsilon}$. Then, by Lemma
\ref{lem:duffy}, there are are $\geq |E|^{1-\epsilon}$ elements of $E^{-1} E$
lying in the kernel of $\pi_1\circ \pi_+$. The kernel of $\pi_1\circ
\pi_+$ is precisely $H' = \pi_+^{-1}(Z(G_+))$, and so we obtain (\ref{it:keato}).
(If the statement on the number of cosets of $H_0$ that $E$ intersects did not hold,
conclusion (\ref{it:kat1}) would follow by Lemma \ref{lem:gorto}.)

Suppose now that
 $|\pi_1(\pi_+(E))| > |E|^{\epsilon}$. If we have (\ref{eq:ogrolo}),
we are done. It remains to consider what happens if we have (\ref{eq:nurdo}).
Applying Lemma \ref{lem:adadar} (with $E = E_k$, $k$ an absolute constant), we see that
$\pi_-(E_{k k'} \cap M) = G_-$ for $k'$ an absolute constant. 
(We may assume that $p>11$ (as is required by Lemma \ref{lem:adadar})
because (\ref{it:kat1}) is trivially true otherwise.)
We now
apply Lemma \ref{lem:achajar}, and obtain that $(E_{k k'}\cap M)_{k''}$ ($k''$
an absolute constant) equals either $M$ or a conjugate
$g G_- g^{-1}$, $g\in M$, of $G_-$. If $(E_{k k'}\cap M)_{k''} = M$, we have obtained (\ref{it:kat2}).

Suppose, then, that $(E_{k k'}\cap M)_{k''} = g G_- g^{-1}$.  If $E$ is contained in the group
$g G_+ g^{-1}$, we have (\ref{it:kat3}) and are done. Assume, then, that there is a $g_1\in E$ such
that $g_1\notin g G_+ g^{-1}$. We can write $g_1 = g a z g^{-1} g_2$, where $a\in A$, $z\in Z(G_+)$,
$g_2\in g G_- g^{-1}$, $a\neq I$. The orbit of $g a z g^{-1}$ under the action of $g G_- g^{-1}$ by
conjugation is all of $g A_0 z g^{-1}$. (This is so  because the action of $G_-$ on $A_0$ by conjugation
can be identified with the action of $\SL_2(K)$ on $K^2$ by left multiplication; since the latter action
is transitive, the former action is transitive too.) Thus $g G_- g^{-1}\cdot  g_1\cdot g G_- g^{-1}$ 
contains $g A_0 z g^{-1}$, and hence 
\[g G_- g^{-1}\cdot  g_1\cdot g G_- g^{-1} = 
g G_- g^{-1} \cdot (g G_- g^{-1}\cdot  g_1\cdot g G_- g^{-1})
\supset g G_- A_0 z g^{-1} = g M z g^{-1} = M z.
\]
Therefore, $E_{2 k k' + 1}$ contains $M z$, and so $E_{4 k k' + 2}$ contains $M$.
We have obtained (\ref{it:kat2}).
\end{proof}

We have spent enough time for now studying subsets $E\subset H_0(K)$ such that
$\pi_1(\pi_+(E))$ generates $\SL_2(K)$; let us now pass to the other cases.

\subsection{Parabolic subgroups: solvable groups}\label{subs:cgarde}

Let $K=\mathbb{Z}/p\mathbb{Z}$, $G = \SL_3(K)$. Let $H_0<G(K)$ be as in 
(\ref{eq:onestar}); let $\pi_+$, $\pi_1$ be as in (\ref{eq:twostar}) and
(\ref{eq:threestar}). Consider a subset $E\subset H_0$ such that 
$\pi_1(\pi_+(E))$ does not generate $\SL_2(K)$.

By Prop.\ \ref{prop:dick}, either (a) $\pi_1(\pi_+(E))$ is contained in a Borel subgroup $B/K$ of
$\SL_2(K)$ or (b) $\pi_1(\pi_+(E))$ is contained in a subgroup $H<\SL_2(K)$ having
a subgroup of index $\leq 2$ lying within a maximal torus $T_0/\overline{K}$ of
$\SL_2$. 

In case (a), $E$ must be contained in a Borel subgroup $B'/K$ of
$SL_3(K)$. We have already examined this situation in 
\S \ref{subs:grotes}
(Prop.\ \ref{prop:lavender}).

In case (b), we can use Prop.\ \ref{prop:amery} to assume without loss of generality that
$H\subset T_0(\overline{K})$.
We will examine this in detail;
the solution will be a simple application of Cor.\ \ref{cor:liz}.

\begin{prop}\label{prop:esmeralda}
Let $K = \mathbb{Z}/p\mathbb{Z}$, $p$ a prime. Let $H_0<G(K)$ and
$A_0<G(K)$ be as in (\ref{eq:onestar}); let $\pi_+$ and $\pi_-$ be as in
(\ref{eq:twostar}) and (\ref{eq:threestar}). Let $T_0/\overline{K}$ be a maximal torus
of $\SL_2$ not defined over $K$. Let $H<H_0$ be the preimage
$(\pi_1\circ \pi_+)^{-1}(T_0(K))$.

Let $E\subset H$. Then, for every $\epsilon>0$, either
\begin{equation}\label{eq:monmar}
 |E_k| \geq |E|^{1 + \delta},\end{equation}
where $k$ and $\delta>0$ depend only on $\epsilon$, or one of the following
cases holds:
\begin{enumerate}
\item\label{it:raph2} $E$ is contained in at most $|E|^{\epsilon}$ cosets of $A_0$,
\item\label{it:raph4} $E_k$ contains a subgroup $X$ of $Z(G(K))\cdot A_0$ for some $k$ depending
only on $\epsilon$; moreover, $X$ is a normal subgroup of $\langle E\rangle$
and $\langle E\rangle/X$ is abelian.
\end{enumerate}
\end{prop}
\begin{proof}
Consider any $g\in H$ with $\pi_1(\pi_+(g))\ne \pm I$. We wish to show that
$g$ has three distinct eigenvalues. (This will simplify matters when we apply 
Cor.\ 
\ref{cor:liz}.)  Since $\pi_1(\pi_+(g)) \ne \pm 1$ belongs to a torus, it must have
two distinct eigenvalues $\lambda_1, \lambda_2\in \overline{K}$. If either were
in $K$,  the other one would be in $K$ as well (by $\lambda_1 \lambda_2 = 1$),
and then $\pi_1(\pi_+(g))$ would be diagonalisable over $K$, i.e.,
$T_0/\overline{K}$ would be defined over $K$.
 Since we are assuming that that cannot
happen, it follows that $\lambda_1, \lambda_2 \notin K$. Thus $g$ has one
rational (that is, $\in K$) eigenvalues $s^2$, and two irrational and distinct
eigenvalues $s^{-1} \lambda_1$, $s^{-1} \lambda_2$. In particular, $g$ has three
distinct eigenvalues.

Consider now any $g\in H$. If $g$ has three distinct eigenvalues, then it clearly has no
fixed points when acting on $H^{(1)}\subset A_0$ by conjugation. If, instead,
$\pi_1(\pi_+(g)) = \pm I$, then, unless $g$ is actually in $Z(G(K))\cdot A_0$, it is also easy to see
that $g$ acts without fixed points on $H^{(1)}$. 

We can thus apply Cor.\ \ref{cor:liz} 
with $G = H$, $m=1$, $H_1 = Z(G(K))\cdot A_0$, $\ell=2$. 
Case (\ref{it:cruph1}) of Cor.\ \ref{cor:liz} gives us (\ref{eq:monmar});
cases (\ref{it:cruph2}) and (\ref{it:cruph3}) give us conclusions
us conclusions (\ref{it:raph2}) and (\ref{it:raph4}).
\end{proof}

We can now study the case of Prop.\ \ref{prop:ralder} that we left for later.
\begin{lem}\label{lem:spolga}
Let $K = \mathbb{Z}/p\mathbb{Z}$, $p$ a prime. Let $H_0, G_+, A_0< \SL_3(K)$ be
as in (\ref{eq:onestar}); let $\pi_+$ and $\pi_1$ be as in (\ref{eq:twostar}) and
(\ref{eq:threestar}).

Let $E\subset H_0$ be such that $\pi_1(\pi_+(E))$ generates $\SL_2(K)$. Suppose that
$E^{-1} E$ lies in the union of at most $|E|^{\delta}$ ($\delta>0$) 
cosets of the subgroup $H' = \pi_+^{-1}(Z(G_+))$ of
$H_0$ .

Then either
\begin{equation}\label{eq:monmaros}
 |E_k| \geq |E|^{1 + \delta},\end{equation}
where $k$ and $\delta>0$ depend only on $\epsilon$, or one of the following
cases holds:
\begin{enumerate}
\item\label{it:craph2} $E$ is contained in at most $|E|^{4 \delta}$
  cosets of $A_0$,
\item\label{it:craph3} $E$ is contained in the union of at most
$|E|^{\delta}$ cosets of $g Z(G_+) g^{-1}$, where $g\in M$; moreover,
  $E\subset g G_+ g^{-1}$;
\item\label{it:craph4} $E_k$ contains $A_0$ for some $k$ depending
only on $\delta$.
\end{enumerate}
\end{lem}
\begin{proof}
Apply Cor.\ \ref{cor:liz} with $G = H'$, $m=1$, $H_1 = A_0\cdot Z(G)$,
$\ell=2$,
$A= H'\cap E^{-1} E$ and $2\delta$ instead of $\delta$. Case (\ref{it:cruph1}) in Cor.\ \ref{cor:liz} gives us
(\ref{eq:monmaros}). Case (\ref{it:cruph2}) (together with Lem.\ 
\ref{lem:gorto}) gives us conclusion (\ref{it:craph2}). Assume, then, that
case (\ref{it:cruph3}) holds.

Suppose first that $X\ne \{e\}$. If $X=A_0$, we have obtained conclusion
(\ref{it:craph4}). Suppose $X$ is neither $\{e\}$ nor $A_0$.
Since $\pi_1(\pi_+(E))$ generates $\SL_2(K)$, $X$ is not stabilised by
the action of $\langle E\rangle$ by conjugation. Thus, there is an $h\in E$
such that $h X h^{-1}$, while in $A_0$, is not equal to $X$. Hence
$h X h^{-1} X$ is all of $A_0$, and thus we have reached conclusion
(\ref{it:craph4}) again.

Suppose now that  $X = \{e\}$ . Then $\langle E\rangle$ is abelian. Unless conclusion
(\ref{it:craph2}) holds, this means that $\langle E\rangle$ lies in
a conjugate $g Z(G_+) g^{-1}$ of $Z(G_+)$ ($g\in M$). If every element
of $E$ lies in $g G_+ g^{-1}$, we have obtained conclusion (\ref{it:craph3}).
If there is an element $h$ of $E$ not
in $g G_+ g^{-1}$, then $h \langle E\rangle
h^{-1}$, while certainly 
in $H'$, is a different torus from $\langle E\rangle$, and so 
$h E h^{-1} E^{-1}$ contains an element of $A_0$ other than the identity.
We apply Cor.\ \ref{cor:liz} with $G = H$, $m=1$, $H_1 = A_0\cdot Z(G)$,
$\ell=2$ and $h E h^{-1} E^{-1}$ instead of $E$; each case works out as
before,
except that $X=\{e\}$ is no longer a possibility.
\end{proof}
\subsection{Conclusions}\label{subs:thalion}
We can now give a detailed account of what happens inside a parabolic subgroup.
\begin{prop}\label{prop:macbeth}
Let $K=\mathbb{Z}/p\mathbb{Z}$, $p$ a prime. Let $G = \SL_3$. Let $G_+$,
$G_-$ and $H_0$ be as in (\ref{eq:onestar}).

Let $A$ be a subset of $H_0$. Then, for every $\epsilon>0$, either
\begin{equation}\label{eq:nickel}|A A A|\gg |A|^{1+\delta},\end{equation}
where $\delta$ and the implied constant depend only on $\epsilon$, or
there are subgroups $H_1\triangleleft H_2\triangleleft \langle A\rangle$ such that
\begin{enumerate}
\item $H_2/H_1$ is nilpotent,
\item $A_k$ contains $H_1$ for some $k$ depending only on $\epsilon$,
\item $A$ is contained in the union of $\leq |A|^{\epsilon}$ cosets of $H_2$.
\end{enumerate}

Moreover, either $H_1$ is trivial ($=\{e\}$) or it contains a non-trivial
subgroup of $U(K)$ for some maximal unipotent subgroup $U/K$ of $G(K)$.
\end{prop}
\begin{proof}
If $A$ is contained in a Borel subgroup $B/K$ of 
$G=\SL_3$, we apply Proposition \ref{prop:lavender}
with $E=A$; we set $H_1 = X$, $H_2 =\langle E\rangle$
and are done.
Suppose, then, that $A$ is not contained in any Borel subgroup $B/K$ of
$\SL_3$. Let $\pi_+$ and $\pi_1$ be as in (\ref{eq:twostar}) and
(\ref{eq:threestar}). Assume first that $\pi_1(\pi_+(A))$ generates
$\SL_2(K)$.  We apply Proposition \ref{prop:ralder} (with $A$ instead of $E$
and $\epsilon/8$
instead of $\epsilon$). Case (\ref{it:kat1}) in Prop.\ \ref{prop:ralder}
gives us (\ref{eq:nickel}), case (\ref{it:kat2}) gives us the statement of the
present
proposition with $H_1 = M$, $H_2 = H_0$, and case (\ref{it:kat3}) gives us the
statement with
$H_1 = g G_- g^{-1}$, $H_2 = g G_+ g^{-1}$.
If case (\ref{it:keato})
in Prop.\ \ref{prop:ralder} holds, apply Lemma \ref{lem:spolga} (with $E=A$).
Equation (\ref{eq:monmaros}) gives us (\ref{eq:nickel});
cases (\ref{it:craph2}), (\ref{it:craph3}) and (\ref{it:craph4})
of Lemma \ref{lem:spolga} give us (a) $H_1 = \{e\}$, $H_2 = A_0$, (b)
$H_1 = \{e\}$, $H_2 = g Z(G_+) g^{-1}$, and (c) $H_1 = A_0$, 
$H_2 = \pi_+^{-1}(Z(G_+))$, respectively.

Assume now, lastly, that (a) $A$ is not contained in any Borel subgroup
of $G$, and (b) $\pi_1(\pi_+(A))$ does not generate $\SL_2(K)$. Then,
as we discussed at the beginning of \S \ref{subs:cgarde},
Prop.\ \ref{prop:amery}
allows us to reduce the situation
to that of Prop.\ \ref{prop:esmeralda} (by passage to a subgroup of
$\langle A\rangle$ of index at most $2$). Equation (\ref{eq:monmar})
gives us (\ref{eq:nickel}); case (\ref{it:raph2}) 
of Prop.\ \ref{prop:esmeralda}) gives us $H_1 = \{e\}$, $H_2 = A_0$,
and case (\ref{it:raph4}) of Prop.\ \ref{prop:esmeralda} gives us
$H_1 = X\cap A_0$, $H_2 = \langle E\rangle$.
\end{proof}

We can now prove Thm.\ \ref{thm:qartay} in the case where $A$ does not
generate $G = \SL_3(K)$.
\begin{prop}\label{prop:wefr}
Let $G = \SL_3$. Let $K = \mathbb{Z}/p\mathbb{Z}$, $p$ a prime.
Let $A\subset G(K)$ be a set that does not generate $G(K)$.

Then, for every $\epsilon>0$, either
\begin{equation}\label{eq:catsup}
|A \cdot A\cdot A|\gg |A|^{1+\delta},\end{equation} where $\delta>0$ and
the implied constant depend only on $\epsilon$, or 
there are subgroups $H_1\triangleleft H_2 \triangleleft \langle A\rangle$
such that 
\begin{enumerate}
\item $H_2/H_1$ is nilpotent,
\item $A_k$ contains $H_1$, where $k$ depends only on $\epsilon$, and
\item $A$ is contained in the union of $\leq |A|^{\epsilon}$ cosets of $H_2$.
\end{enumerate}
\end{prop}
\begin{proof}
Let $H = \langle A\rangle$.
Then $H$ satisfies one of the descriptions
in Cor.\ \ref{cor:odious}, cases
(\ref{it:aleg1})--(\ref{it:aleg5}).

If case (\ref{it:aleg1}) of Cor.\ \ref{cor:odious} holds, 
then $H$ is contained in a conjugate of the maximal parabolic group $P(K)$ having $H_0$ as a subgroup 
of index $2$.
($H_0$ is as defined in the beginning of \S \ref{subs:rapanui}.)
We then apply Prop.\ \ref{prop:macbeth}, follow it by Prop.\ \ref{prop:amery},
and are done.

Since stabilisers of lines in $\mathbb{P}^3$ defined over
$\mathbb{Z}/p\mathbb{Z}$ are isomorphic as groups to stabilisers
of points in $\mathbb{P}^3$ defined over $\mathbb{Z}/p\mathbb{Z}$
(see (\ref{eq:naug}) and (\ref{eq:ghty})), case (\ref{it:aleg2})
of Cor.\ \ref{cor:odious} reduces to case (\ref{it:aleg1}) of
Cor.\ \ref{cor:odious}. Case (\ref{it:aleg3}) of Cor.\ \ref{cor:odious}
gives us desired conclusion immediately (with $H_1 = \{e\}$).
If case (\ref{it:aleg4}) of Cor.\ \ref{cor:odious} holds, apply
Cor.\ \ref{cor:mororga}. Finally, if case (\ref{it:aleg5}) holds, then
$|A|$ is bounded by an absolute constant and so (\ref{eq:catsup}) holds
trivially.
\end{proof}

This is as good a place as any to note that the conclusion ``$H_2/H_1$
is nilpotent'' in Prop.\ \ref{prop:wefr} (and like results) cannot
be strengthened to ``$H_2/H_1$ is abelian''. Indeed, there are non-abelian
nilpotent groups where some sets of generators fail to grow
even though they are not too large to grow. Take $N<\sqrt{p}$. Let
\[A = \left\{\left(\begin{matrix} 1 & a & b\\ 0 & 1 & c\\ 0 & 0 & 1
\end{matrix}\right) : |a|,|c|\leq N, |b|\leq N^2\right\}.\]
Then $|A|= (2 N + 1)^2 (2 N^2 + 1) \geq 8 N^4$ and $|A\cdot A \cdot A|\ll N^4$; in other words, $A$
does not grow, and yet it is neither too large to grow nor 
a subset of an abelian group.

\section{Growth of medium-sized and large sets}\label{sec:grome}
Let $G=\SL_3$, $K=\mathbb{Z}/p\mathbb{Z}$.
 Let
$A$ be a set of generators of $G(K)$.
We must show that, if
$p^{4-\delta} \leq |A|\leq p^{4+\delta}$, $\delta>0$, then $A$ grows. We will, in fact,
be able to show something stronger: if $|A|\geq p^{3.2 + \delta'}$, $\delta'>0$,
then $A$ grows.


The key here will be to pass to a subgroup. We let $H_0$ be as in \S
\ref{subs:playdirt}. The group $H_0$ is then a subgroup of index $2$ in
a maximal parabolic subgroup of $G(K)$, and so $\lbrack G(K) : H_0\rbrack
= 2 (p^2 + p + 1)$.
Then $A$ is a great deal larger than $\lbrack G(K):H_0\rbrack$, and thus,
by Lemma \ref{lem:duffy},
the intersection
$A^{-1} A\cap H_0$
 will be large. We devoted most of
\S \ref{sec:pogor} to the question of which subsets of $H_0$ grow. If
$A^{-1} A\cap H_0\subset H_0$ grows, then, by Lemma \ref{lem:koph}, $A$ itself grows.
If, instead, there are subgroups $H_1$, $H_2$ as in Prop.\ \ref{prop:macbeth}
-- so that $A^{-1} A \cap H_0$ essentially contains $H_1$ and is essentially
contained in $H_2$ --
we can multiply conjugates of $H_1$ or $H_2$
(``sticking subgroups in different directions'') to obtain that $A$ grows.

\subsection{Sticking subgroups of $\SL_3$ in different directions}
Let us begin by considering abelian subgroups $H$ of $\SL_3(K)$ that 
(a) are not contained in tori and (b) do not have subgroups of index $\leq 3$
lying on unipotent subgroups.
It is easy to show that
every such abelian subgroup $H$ is conjugate over $\SL_3(\overline{K})$
to a subgroup of one of the following groups:
\begin{equation}\label{eq:ah12}H_{1,2} = \left\{
\left(\begin{matrix}r & x & 0\\ 0 & r & 0\\ 0 & 0 & r^{-2}\end{matrix}
\right) : r\in K^*, x\in K\right\},
\end{equation}
\begin{equation}\label{eq:ah23}H_{2,3} = \left\{
\left(\begin{matrix}r^{-2} & 0 & 0\\ 0 & r & z\\ 0 & 0 &r\end{matrix}
\right) : r\in K^*, z\in K\right\},
\end{equation}
\begin{equation}\label{eq:ah13}
H_{1,3} = \left\{
\left(\begin{matrix}r & 0 & y\\ 0 & r^{-1} & 0\\ 0 & 0 & r\end{matrix}
\right) : r\in K^*, y\in K\right\}.
\end{equation}
(If $H$ contained at least one element with three distinct eigenvalues,
then $H$ would lie on a torus. If $H$ contains at least one element
with two distinct eigenvalues, then $H$ is contained in a conjugate of
one of the groups $H_{1,2}$, $H_{2,3}$, $H_{1,3}$. If no element of $H$
contains at least two distinct eigenvalues, then $H$ has a subgroup $H'$
of index $\leq 3$ such that every element of $H'$ has $1$ as its only
eigenvalue, and then $H'$ is, by definition, a unipotent group.)
\begin{lem}\label{lem:bipip}
Let $G = \SL_3$, seen as a group defined
 over a field $K$ of characteristic $\ne 3$. 
Let $H$ be one of the subgroups $H_{i,j}$ listed above. Let $\mathfrak{g}$ be the Lie
algebra of $G$ and $\mathfrak{h}$ the Lie algebra of $H$.
 Then there
are $\vec{g}_1, \vec{g}_2\in \mathfrak{g}$ such that
\[
\mathfrak{h}, \lbrack \vec{g}_1, \mathfrak{h}\rbrack, \lbrack \vec{g}_2,
\mathfrak{h}\rbrack
\]
are linearly independent and of dimension $\dim(\mathfrak{h})$.
\end{lem}
\begin{proof}
Since the three subgroups $H_{i,j}$ listed above
are conjugate over $G(K)$,
we can assume without loss of generality that
we have $H = H_{1,2}$. Then $\mathfrak{h}$ is spanned by
$e_{1,1} + e_{2,2} - 2 e_{3,3}$ and $e_{1,2}$, where $e_{i,j}$ is 
the $3$-by-$3$ matrix having a $1$ at the $(i,j)$th entry and $0$s elsewhere.
Set $\vec{g}_1 = e_{3,1}$, $\vec{g}_2=e_{2,3}$.
\end{proof}

\begin{prop}\label{prop:hanbel}
Let $G = \SL_3$, seen as a group defined
 over a field $K$ of characteristic $\charac(K)=0$ or $\charac(K)>3$. 
Let $H\subset G$ be conjugate over $G(\overline{K})$ to
 one of the subgroups $H_{i,j}\subset H$ listed above.

Let $A$ be a set of generators of $G(K)$, and $E$ a non-empty subset of 
$H(K)$. Then there are $g_0,g_1,g_2\in A_k$, $k\ll 1$, such that
\[
|g_0 E g_0^{-1} \cdot g_1 E g_1^{-1} \cdot g_2 E g_2^{-1}| \gg |E|^3,
\]
where the implied constants are absolute.
\end{prop}
\begin{proof}
By Lemma \ref{lem:bipip}, the assumptions of Prop.\ \ref{prop:elysium}
are fulfilled. The conclusions of Prop.\ \ref{prop:elysium} provide the
linear-independence assumption of Prop.\ \ref{prop:galoshes}; we apply
Prop.\ \ref{prop:galoshes}, and are done. The implied constants 
are absolute because they depend only on $n$, which is fixed ($n=3$).
\end{proof}

Let us now look at algebraic subgroups of a unipotent subgroup of $\SL_3$.
\begin{lem}\label{lem:anchar}
Let $G = \SL_3$, defined over a field $K$. Let $H$ be one of the algebraic
subgroups of $G$ listed in Lemma \ref{lem:betson} (equations (\ref{eq:dotor}) 
-- (\ref{eq:dogar})) other than $\{I\}$. 
Let $T\subset G$ be the subgroup of diagonal matrices
of $G$. Write $\mathfrak{g}$ for the Lie
algebra of $G$, $\mathfrak{h}\subset \mathfrak{g}$ 
for the Lie algebra of $H$, and $\mathfrak{t}\subset \mathfrak{g}$ for
the Lie algebra of $T$.

Then there are $g_0, g_1, g_2,\dotsc, g_{\ell}\in G(\overline{K})$
such that
\[\Ad_{g_0}(\mathfrak{t}), \Ad_{g_1}(\mathfrak{h}), \Ad_{g_2}(\mathfrak{h}),
\dotsc, \Ad_{g_{\ell}}(\mathfrak{h})\]
are linearly independent. Here 
$\ell = (\dim(\mathfrak{g}) - \dim(\mathfrak{t}))/\dim(\mathfrak{h})$.
\end{lem}
Strictly speaking, 
Lemma \ref{lem:betson} actually lists the sets of points $H(K)$; it should
be clear which algebraic groups $H$ are thereby listed.
\begin{proof}
Set $g_0 = g_1 = I$. In every case, we will set $g_2,\dotsc,g_{\ell}$
equal to permutation matrices. If $H$ is the whole group $U$ of 
upper-triangular unipotent matrices, then $\ell = 2$; set
$g_{2}$ equal to the permutation 
matrix corresponding to the permutation $(1\; 3)$ -- that is,
\[g_2 = \left(\begin{matrix} 0 & 0 & 1\\ 0 & 1 & 0 \\ 1 & 0 & 0
\end{matrix}\right).\] If $H$ is
as in (\ref{eq:mata1}), let $g_2$, $g_3$ be the matrices corresponding
to the permutations $(1\; 2)$ and $(1\; 3)$. For (\ref{eq:mata2}), choose
$g_2$, $g_3$ corresponding to $(1\; 3)$ and $(2\; 3)$. For (\ref{eq:mata2}),
(\ref{eq:matb1}) or (\ref{eq:matb2}),
we use the entire permutation group $S_3$, i.e.,
we let $g_1,\dotsc,g_6$ be the permutation matrices corresponding to each element of
$S_3$ in turn. For (\ref{eq:doloro}), we
let  $g_2$ and $g_3$ be the permutation matrices corresponding to the
$3$-cycles in $S_3$. Finally, for (\ref{eq:dogar}), we use the
entire permutation group $S_3$.
\end{proof}

\begin{prop}\label{prop:bettyboo}
Let $G = \SL_3$, defined over a field $K$. Let $H$ be conjugate to
one of the algebraic
subgroups of $G$ listed in Lemma \ref{lem:betson} (equations (\ref{eq:dotor}) 
-- (\ref{eq:dogar})) other than $\{I\}$. 
Let $T/\overline{K}$ be a maximal torus of $G$. Let $\ell =
(\dim(G) - \dim(T))/\dim(H)$, where $\dim(G)$, $\dim(T)$ and $\dim(H)$
are the dimensions
of $G$, $T$ and $H$ as varieties.

Let $A$ be a set of generators of $G(K)$, $D$ a non-empty subset
of $T(K)$, and $E$ a non-empty subset of 
$H(K)$. Then there are $g_0,g_1,\dotsc,g_{\ell}\in A_k$, $k\ll 1$, such that
\[
|g_0 D g_0^{-1} \cdot g_1 E g_1^{-1} \dotsb g_{\ell} E g_{\ell}^{-1}| 
\gg |D|\cdot |E|^{\ell}, 
\]
where the implied constants are absolute.
\end{prop}
\begin{proof}
Lemma \ref{lem:anchar} states that the
linear-independence assumption of Prop.\ \ref{prop:galoshes} is true.
We apply Prop.\ \ref{prop:galoshes}, and we are done. The implied constants 
are absolute because they depend only on $n$, which is fixed ($n=3$).
\end{proof}

\subsection{Growth}
\begin{lem}\label{lem:tarugo}
Let $G = \SL_3$.
Let $K = \mathbb{Z}/p\mathbb{Z}$, $p$ a prime. 
Let $H_0<G(K)$ be as in (\ref{eq:onestar}). Let $A\subset G(K)$
be a set of generators of $G(K)$, and let $E$
be a non-empty subset of $H_0$.

Then, for every $\epsilon>0$, either
\begin{enumerate}
\item\label{it:tric}
$|A_k| \gg |A|^{1 + \epsilon}$, 
where $k$ and the implied constant are absolute, or
\item\label{it:treble}
$|E\cdot E\cdot E| \gg |E|^{1+\delta}$,
where $\delta$ and the implied constant depend only on $\epsilon$, or
\item\label{it:aga1} 
$|(A\cup E)_k| \gg |A|^{\frac{1}{4}-\epsilon} \cdot |E|^{2 - 2\epsilon}$,
where the implied constant is absolute and $k$ depends only on $\epsilon$,
\item\label{it:aga2}
$|(A\cup E)_k| \gg |E|^{3-3\epsilon}$, where $k$ and the implied constant are absolute,
\item\label{it:aga3}  
$|(A\cup E)_k| \gg p^6 \cdot |A|^{\frac{1}{4}-\epsilon}$, where $k$ and
the implied constant are absolute, or  
\end{enumerate}
\end{lem}
\begin{proof}
Apply Prop.\ \ref{prop:macbeth} with $E$ instead of $A$. If 
(\ref{eq:nickel}) holds, we have conclusion (\ref{it:treble}). Assume
(\ref{eq:nickel}) does not hold. Then there are subgroups 
$H_1\triangleleft H_2\triangleleft \langle A\rangle$ as in
Prop.\ \ref{prop:macbeth}.

Suppose first that $H_1 = \{e\}$. Since $H_2$ is then a nilpotent subgroup
of $\SL_3(K)$, either $H_2$ is an abelian group
containing at least one element with at least two distinct eigenvalues or $H_2$
has a subgroup of index $\leq 3$ contained in $U(K)$ for some maximal
unipotent subgroup $U/K$ of $G=\SL_3$. Consider the latter case first.
Since $E$ is contained in $\leq |E|^{\epsilon}$ cosets of
$H_2$, Lemma \ref{lem:duffy} implies 
$|E_{2k}\cap U(K)|\geq 3 |E|^{1-\epsilon}$.  By Cor.\ \ref{cor:dophus},
either conclusion (\ref{it:tric}) holds (with $k+2$ instead of $k$)
or $|A_k \cap T(K)| \gg |A|^{\frac{1}{4} - \epsilon}$
for some maximal torus $T/\overline{K}$ of $G$, where $k$ and the
implied constant are absolute. Suppose
$|A_k \cap T(K)| \gg |A|^{\frac{1}{4} - \epsilon}$. By Prop.\
\ref{prop:bettyboo}, it follows that 
\[|(A\cup E)_{6 k' +5 k}|\gg |A_k\cap T(K)| |E_{2k} \cap U(K)|^2 \gg
|A|^{\frac{1}{4}-\epsilon} |E|^{2 - 2\epsilon},\]
where $k'$ and the implied constant are absolute 
($k'$ is the constant $k$ from Prop.\ \ref{prop:bettyboo}).
Conclusion (\ref{it:aga1}) follows (with $6 k' + 5k$ instead of $k$).

Suppose now that $H_1=\{e\}$ and $H_2$ is an abelian group
containing at least one element with at least two distinct eigenvalues.
Then either $H_2$ is one of the groups $H_{1,2}$, $H_{2,3}$, $H_{1,3}$
in (\ref{eq:ah12})--(\ref{eq:ah13}) or $H_2$ lies in a maximal
torus. Suppose first that $H_2$ lies in a maximal torus.
Then, by Proposition 
\ref{prop:grati}, \[|(A\cup E)_{8 k + 8}| \gg |E_2 \cap H_2|^4 \geq
 |E|^{4 - 4\epsilon},\]
where $k$ and the implied constant are absolute. Conclusion (\ref{it:aga2})
follows (with $8 k +8$ instead of $k$; we may assume $\epsilon<1$,
and so $4 - 4\epsilon>3-3\epsilon$).
Now suppose $H_2$ is as in (\ref{eq:ah12}), (\ref{eq:ah23}) or (\ref{eq:ah13}).
Then, by Proposition \ref{prop:hanbel},
\[|(A\cup E)_{6 k + 3}| \gg |E_2 \cap H_2|^3 \geq |E|^{3 - 3\epsilon},\]
where $k$ and the implied constant are absolute. Conclusion (\ref{it:aga2})
follows again (with $6 k + 3$ instead of $k$).

Suppose now that $H_1\ne \{e\}$. We know from Prop.\ \ref{prop:macbeth}
that $H_1$ contains a non-trivial subgroup $H$ of $U(K)$, where
$U/K$ is a maximal unipotent subgroup of $G$.
Then $H$ is conjugate to one of the subgroups listed in Lem.\ \ref{lem:betson},
(\ref{eq:mata1})--(\ref{eq:dogar}). By Cor.\ \ref{cor:dophus},
either conclusion (\ref{it:tric}) holds (with $k+2$ instead of $k$)
or $|A_k \cap T(K)| \gg |A|^{\frac{1}{4} - \epsilon}$, where the
implied constant is absolute. Suppose
$|A_k \cap T(K)| \gg |A|^{\frac{1}{4} - \epsilon}$. Then, by Prop.\
\ref{prop:bettyboo}, 
\[|(A\cup E)_{k + 12 k' + 6 k''}| \gg |A|^{\frac{1}{4}-\epsilon} \cdot p^6,  
\]
where $k$, $k'$, $k''$ and the implied constants are absolute.
Conclusion (\ref{it:aga3}) follows (with $k+12 k'+6k''$ instead of $k$).
\end{proof}

\begin{prop}\label{prop:vicru}
Let $G = \SL_3$.
Let $K = \mathbb{Z}/p\mathbb{Z}$, $p$ a prime. 
Let $A\subset G(K)$ be a set of generators of $G(K)$.
Suppose $p^{3.2 + \eta}\leq |A| \leq p^{8 - \eta}$, where $\eta>0$.
Then
\begin{equation}\label{eq:urud}
|A\cdot A\cdot A| \gg |A|^{1 + \delta},\end{equation}
where $\delta>0$ and the implied constant depend only on $\eta$.
\end{prop}
\begin{proof}
Let $E = A^{-1} A \cap H_0$, where $H_0<G(K)$ is as in (\ref{eq:onestar}). 
By Lemma \ref{lem:duffy}, \[|E| \geq 
\frac{|A|}{\lbrack G(K):H_0\rbrack} = \frac{|A|}{2 (p^2 + p + 1)}
> \frac{|A|}{3 p^2}.\]
Apply Lem.\ \ref{lem:tarugo} with $\epsilon = \min\left(\frac{1}{32},
\frac{\eta}{3}\right)$.

If case (\ref{it:tric}) of Lem.\ \ref{lem:tarugo} holds, we obtain
(\ref{eq:urud}) by the tripling lemma (Lemma \ref{lem:furcht}).
If case (\ref{it:treble}) of Lem.\ \ref{lem:tarugo} holds,
we obtain (\ref{eq:urud}) by Lemma \ref{lem:koph}.

Suppose case (\ref{it:aga1}) of Lemma \ref{lem:tarugo} holds.
Then
\[|A_{2 k}| \geq |(A\cup E)_k| \gg |A|^{\frac{1}{4}-\epsilon} 
|E|^{2 - 2 \epsilon} > \frac{1}{9 p^4} |A|^{2 + \frac{1}{4} - 3 \epsilon},\]
where the implied constant is absolute. Since $|A|\geq p^{3.2+\eta}$,
we see that $|A|^{1 + \frac{1}{4}} \geq p^{4 + \frac{5}{4} \eta}$, and so
\[|A_{2 k}| \gg |A|^{1 + \frac{5}{4} \eta - 3 \epsilon} 
= |A|^{1 + \frac{1}{4} \eta},\]
where the implied constant is absolute.
We then obtain (\ref{eq:urud}) by the tripling lemma.

Suppose case (\ref{it:aga2}) of Lemma \ref{lem:tarugo} holds. Then
\[|A_{2 k}| \geq |(A\cup E)_k| \gg |E|^{3 -3 \epsilon} >
\frac{1}{27 p^6} |A|^{3 - 3\epsilon},\]
where the implied constant is absolute. Since $|A|\geq p^{3.2+\eta}$,
we have $|A|^{2\cdot \frac{3}{3.2}}\gg p^6$, and thus
\[|A_{2 k}| \gg |A|^{1 + (2 - 2\cdot \frac{3}{3.2}) - 3\epsilon} =
|A|^{1 + \frac{1}{8} - 3\epsilon} \geq |A|^{1 + \frac{1}{32}},\]
where the implied constant is absolute.
We obtain (\ref{eq:urud}) by the tripling lemma.

Suppose, finally, that (\ref{it:aga3}) of Lem.\ \ref{lem:tarugo} holds.
Then
\[|A_{2 k}| \geq |(A\cup E)_k| \gg |E|^{3 -3 \epsilon} >
p^6 \cdot |A|^{\frac{1}{4}-\epsilon}.\]
Since $|A| \leq p^{8 - \eta}$, where $\eta>0$, we have
$p^6 \geq |A|^{\frac{6}{8-\eta}} \geq |A|^{\frac{3}{4} + \eta}$ (as $\eta$ is
certainly $<7$) and so 
\[|A_{2 k} \geq |A|^{1 + \eta - \epsilon}.\]
We obtain (\ref{eq:urud}) by the tripling lemma.
\end{proof}

\section{General conclusions and final remarks}\label{sec:genfin}

\subsection{The main theorem, related results and their consequences}\label{subs:broker}

We must now simply put together our work on small and large sets
(\S \ref{sec:grosp} -- \S \ref{sec:armon})
 with our work on medium-sized and large sets (\S \ref{sec:grome}). (As should
be clear from the wording, there is an overlap; we have no use for it.)

\begin{main}
Let $G = \SL_3$. Let $K = \mathbb{Z}/p\mathbb{Z}$, $p$ a prime.
Let $A\subset G(K)$ be a set of generators of $G(K)$.

Suppose $|A|\leq |G(K)|^{1-\delta}$, $\delta>0$. Then
\begin{equation}\label{eq:durod}
|A\cdot A\cdot A|\gg |A|^{1 + \epsilon},\end{equation}
where $\epsilon>0$ and the implied constant depend only on $\delta$. 
\end{main}
\begin{proof}
The condition $|A|\leq |G(K)|^{1-\delta}$ implies
$|A| \leq p^{8 - 8\delta} < p^{8 - \delta}$. If $|A|\leq p^{3.5}$ (say)
or $p^{4.5}\leq |A| \leq p^{8 - \delta}$, use Prop.\ \ref{prop:ogoth}.
If $p^{3.5}< |A| < p^{4.5}$, use Prop.\ \ref{prop:vicru}
(with $\eta = 3.5-3.2 = 0.3$, say).
\end{proof}

In the remainder, we shall need the following extremely simple lemma.
\begin{lem}\label{lem:rodo}
Let $G$ be a group. Let $A$ be a finite set of generators of $G$. Then, for
every $\ell\geq 1$, either
\[|A_{\ell+1}| \geq |A_{\ell}|+1\;\;\;\; \text{or}\;\;\;\;
A_{\ell} = G.
\]
\end{lem}
\begin{proof}
Since $A_{\ell} \subset A_{\ell+1}$, either $|A_{\ell+1}|\geq |A_{\ell}|+1$
or $A_{\ell+1} = A_{\ell}$ holds. If $A_{\ell+1} = A_{\ell}$, then
$A_{\ell}$
is closed under multiplication by elements of $A\cup A^{-1}$. By the
definition of $A_{\ell}$, this implies that $A_{\ell}$ is closed under
the group operation. We already know that $A_{\ell}$ is closed under inversion
by the definition of $A_{\ell}$ (see (\ref{eq:defsubl})). Hence $A_{\ell}$ is a subgroup of $G$. Since $A$ generates
$G$, this means that $A_{\ell} = G$.
\end{proof}

The main theorem has the following alternative statement. It looks stronger,
but it isn't really; it is merely simpler to use sometimes.
\begin{prop}\label{prop:alt}
Let $G = \SL_3$. Let $K = \mathbb{Z}/p\mathbb{Z}$, $p$ a prime.
Let $A\subset G(K)$ be a set of generators of $G(K)$.

Suppose $|A|\leq |G(K)|^{1-\epsilon}$, $\epsilon>0$. Then
\[
|A\cdot A\cdot A|\geq |A|^{1 + \delta},\]
where $\delta>0$ and the implied constant depend only on $\epsilon$. 
\end{prop}
This amounts simply to the following: the $\gg$ in (\ref{eq:durod}) has been replaced by a $\geq$.
\begin{proof}
First, notice that (\ref{eq:durod}) 
(that is, $|A\cdot A\cdot A|\gg |A|^{1+\delta}$) implies 
$|A\cdot A \cdot A|\geq |A|^{1 + \delta/2}$ for $|A|$ larger than a constant
$C$ depending only on $\delta$ and on the implied constant in 
(\ref{eq:durod}). Since $\delta$ and the implied constant in
(\ref{eq:durod})
depend only on $\epsilon$, which is fixed, we conclude that the main
theorem implies that
\begin{equation}\label{eq:erwo}
|A\cdot A\cdot A|\geq |A|^{1 + \delta/2}\end{equation}
whenever $|A|\leq |G|^{1-\epsilon}$ 
and $|A|\geq C$, where $C$ is an absolute
constant. 

If $|A|<C$, then (\ref{eq:erwo}) and Lemma
\ref{lem:rodo} imply that $|A_3|\geq |A|+2\geq |A|^{1+2/C}$. By the
tripling lemma (Lemma \ref{lem:furcht}), it follows that $|A\cdot A\cdot A|
\geq |A|^{1+\delta}$, $\delta$ depending only on $C$, which is an absolute
constant.
\end{proof}


For most applications, it is necessary to supplement the main theorem with a 
result on very large sets. The result we need was 
proven by Gowers \cite{Gow} and (in great generality) by Babai,
Nikolov and Pyber (\cite{NP}, \cite{BNP}). 

\begin{lem}\label{lem:twins}
Let $G=\SL_3$.
Let $K = \mathbb{Z}/p\mathbb{Z}$, $p$ a prime. 
Let $A\subset G(K)$ be a set of generators of $G(K)$.

There is an absolute constant $\epsilon>0$ such that, if $|A|>|G|^{1-
\epsilon}$, then
\begin{equation}\label{eq:toyo}A\cdot A\cdot A = G(K).\end{equation}
\end{lem}
\begin{proof}
By \cite[Cor.\ 1 and Prop.\ 2]{NP},
\begin{equation}\label{eq:godoro}
A\cdot A\cdot A = G(K)
\end{equation}
provided that $|A|>2 |G|^{1-\frac{1}{3 (n+1)}} = 2 |G|^{1 - 1/12}$.

Let $\epsilon = \frac{1}{13}$. Now $|A|>|G(K)|^{1-\frac{1}{13}}$ implies
$|A|> 2 |G(K)|^{1 - \frac{1}{12}}$, provided that $p$ is larger
than an absolute constant, and so (\ref{eq:godoro}) follows. (If $p$ is not larger than an absolute constant,
Lemma \ref{lem:rodo} gives us (\ref{eq:toyo}) easily.)
\end{proof}

\begin{proof}[Proof of Thm.\ \ref{thm:qartay}]
If $A$ does not generate $G(K)$, apply Proposition \ref{prop:wefr}.
Assume $A$ generates $G(K)$. If $|A|>|G|^{1- \epsilon}$, 
where $\epsilon$ is as in Lemma \ref{lem:twins}, then, by Lemma \ref{lem:twins},
$A\cdot A\cdot A = G(K)$. If $|A|\leq |G|^{1-\epsilon}$, the main theorem shows
that $|A\cdot A\cdot A|\gg |A|^{1+\delta}$, where $\delta$ is absolute.
\end{proof}

We recall that Corollary \ref{cor:gorot} states that, for any
set of generators $A$ of $G = \SL_3(\mathbb{Z}/p\mathbb{Z})$,
\begin{equation}\label{eq:miex}
\diam(\Gamma(G,A)) \ll (\log |G|)^c,
\end{equation}
where $c$ and the implied constant are absolute.

\begin{proof}[Proof of Corollary \ref{cor:gorot}]
Let $\epsilon$ be as in Lem.\ \ref{lem:twins}.
Apply Prop.\ \ref{prop:alt} to
$A$, then to $A' = A\cdot A\cdot A$, then to 
$A'' = A'\cdot A' \cdot A'$, etc. After at most
\[k = \log_{(1 + \delta)} \frac{\log |G|}{\log |A|}
= \frac{\log((\log |G|)/(\log |A|))}{\log(1 + \delta)}
\leq \frac{\log((\log |G|)/(\log 2))}{\log(1 + \delta)}
 \ll
\frac{1}{\delta} \cdot \log \log |G|\]
steps, we shall have obtained a set $A^{(k)}$ with 
$|A^{(k)}|>|G|^{1-\epsilon}$ elements, where $\epsilon$ is as in
Lemma \ref{lem:twins}.
 We now apply  Lemma \ref{lem:twins} to $A^{(k)}$.

We conclude that
\[\mathop{\underbrace{A \cdot A \cdot A \dotsb A}}_{\text{$\ell$ times}}
= G,\]
where $\ell = 3^{k + 1} = 3\cdot e^{O\left(
\frac{1}{\delta} \cdot \log \log |G|\right)} = 
3 \cdot (\log |G|)^{O(1/\delta)}$.
Thus, the statement (\ref{eq:miex}) holds with $c = O(1/\delta)$,
where the implied constant is absolute. Since $\delta$ depends only
on $\epsilon$, and $\epsilon$ is as in Lemma \ref{lem:twins}, i.e.,
an absolute constant, we see that $\delta$ itself is an absolute constant.
\end{proof}

For the sake of making matters self-contained, we could replace Lemma
\ref{lem:twins} with a weaker result that we can prove ``by hand'',
namely, Lemma \ref{lem:twins} with
\begin{equation}\label{eq:yotor}A_k = G(K)\end{equation}
instead of (\ref{eq:toyo}). (Here $k$ is an absolute constant.)
This weaker version of Lemma \ref{lem:twins}
can be proven as follows.
\begin{proof}[Sketch of proof of (\ref{eq:yotor})]
Let $U_1$, $U_2$ and $T$ be the algebraic subgroups of $G$ consisting
of unipotent upper-triangular, unipotent lower-triangular and diagonal
matrices, respectively. By Lemma \ref{lem:duffy}, there are many 
($\geq p^{3 - \epsilon}$) elements of $A^{-1} A$\, in $U_1(K)$, many 
($\geq p^{3 - \epsilon}$) elements of $A^{-1} A$ in $U_2(K)$, and many
($\geq p^{2 - \epsilon}$) elements of $A^{-1} A$ in $T(K)$. Assume $\epsilon<1$.
 Then the set
$A^{-1} A \cap U_1(K)$ is too large not to generate $U_1(K)$ (Lemma
\ref{lem:betson}), and thus indeed generates $U_1(K)$. For the same reason,
$A^{-1} A \cap U_2(K)$ generates $U_2(K)$. 

Let $D_0 = A^{-1} A \cap T(K)$. There are fewer than $3 p$ elements
of $T(K)$ with repeated eigenvalues; hence, for every $g\in G$,
there are fewer than $3 p$ elements $g'\in T(K)$ such that
$g^{-1} g'$ has repeated eigenvalues. We choose an element $g_1\in 
D_0$, then an element $g_2\in D_0$ such that $g_1^{-1} g_2$ does not
have repeated eigenvalues, then a $g_3\in D_0$ such that neither
$g_1^{-1} g_3$ nor $g_2^{-1} g_3$ has repeated eigenvalues, etc. We stop
when we cannot find a $g_{k+1}\in D_0$ such that each of
$g_1^{-1} g_{k+1}$, $g_2^{-1} g_{k+1}$,\dots , $g_k^{-1} g_{k+1}$ has
distinct eigenvalues. Now, for each $1\leq j\leq k$, the condition that
$g_j^{-1} g_{k+1}$ have distinct eigenvalues rules out fewer than 
$3 p$ possible elements $g_{k+1}$ of $D_0$. Thus
\[k> \frac{|D_0|}{3 p} \gg p^{1 - \epsilon},\]
where the constant is absolute. Let $D = \{g_1,g_2,\dotsc,g_k\}$.
Then every element of $D^{-1} D$ has distinct eigenvalues.

Therefore, every element of $D^{-1} D$ acts on $U_1(K)$ and $U_2(K)$ without
fixed points (condition (\ref{eq:conocon})).
We now apply Cor.\ \ref{cor:ogrodo}, once to the action of $T(K)$ on
$U_1(K)$ and once to the action of $T(K)$ on $U_2(K)$. We obtain
\begin{equation}\label{eq:litstar}
U_1(K)\subset A_k\;\;\;\;\; \text{and}\;\;\;\;\; U_2(K)\subset A_k 
\end{equation}
(where $k$ is absolute provided that $\epsilon$ is less than some fixed
constant less than $1$). 

By direct
computation, one can verify that every matrix $g\in G(K)$ that is not
upper triangular can be written in the form $g = u_1\cdot u_2\cdot u_1'$,
where $u_1, u_1' \in U_1(K)$, $u_2\in U_2(K)$, and every matrix $g\in G(K)$
that is not lower triangular can be written in the form 
$g = u_2 \cdot u_1 \cdot u_2'$, where $u_1\in U_1(K)$ and 
$u_2, u_2'\in U_2(K)$. It can then be easily shown that every $g\in G(K)$
can be written in the form $g = u_1\cdot u_2 \cdot u_1' \cdot u_2'$, 
where $u_1, u_1' \in U_1(K)$ and $u_2,u_2'\in U_2(K)$. Thus, by
(\ref{eq:litstar}), we conclude that
\[G(K) \subset A_{4 k},\]
as was desired.
\end{proof}

If, for the sake of making the paper relatively self-contained, one
were to use (\ref{eq:yotor}) instead of Lem.\ \ref{lem:twins}
in the proof of Cor.\ \ref{cor:gorot},
one would obtain $\diam(\Gamma(G,A \cup A^{-1}))\ll (\log |G|)^c$
instead of (\ref{eq:miex}). (One could then deduce (\ref{eq:miex})
by \cite[Thm.\ 1.4]{Ba}.) Neither the proof nor the statement of
Thm.\ \ref{thm:qartay} would require any changes.

\subsection{Work to do: other groups}\label{subs:gorno}
It is natural to hope for a broad generalisation. The methods in
\S \ref{sec:torcon} are very likely to carry over to all semisimple groups
of Lie type over arbitrary fields. One can thus arguably hope for
a proof of the main theorem with $\SL_3(\mathbb{Z}/p\mathbb{Z})$ replaced
by $G(K)$, $G$ semisimple of Lie type, $K$ a finite field. (The methods
in \S \ref{sec:torcon} are such that $\epsilon$ would have to depend on
the Lie type of $G$. Some very recent results of Pyber \cite{P} show
that this is a reality, and not just a limitation of the method --
a statement such as the main theorem with $\epsilon$ independent of
the rank of $G$ would be false.)

Results such as those in \S \ref{subs:schw} can probably be strengthened
at least to the extent needed for $\SL_n$. The main difficulties reside in
generalising \S \ref{sec:armon} and \S \ref{sec:grome}. As it stands,
\S \ref{sec:armon} uses the fact that, for $n = 2,3$, the conjugacy
class of an element $g\in \SL_3(K)$ is given by the values $\chi(g)$
of characters $\chi$ of dimension $n$. This is no longer the case for
$n>3$. There do seem to be somewhat involved
ways to avoid this problem by the use of a single character of dimension $n$
(such as the trace).

One of the problems in generalising \S \ref{sec:pogor} -- which is used in 
\S \ref{sec:grome} --  lies in the fact that $\SL_{n-1}$ can have a rather complicated
subgroup structure for $n>3$. (We are speaking of $\SL_{n-1}$ because it is
the more interesting part of any maximal parabolic subgroup of $\SL_n$.)
It does seem that, if one's goal is simply to prove
results on medium-sized sets as in \S \ref{sec:grome} -- rather than to study growth in
subgroups for its own sake -- there are ways to
limit oneself to the consideration of {\em algebraic} subgroups (of
bounded degree) of $\SL_{n-1}$, as opposed to {\em all}
 subgroups of $\SL_{n-1}(K)$. This does simplify matters. However, the
 problem remains that growth in some algebraic subgroups of $\SL_{n-1}$ may be
 harder to study than in $\SL_n$ itself. For example, right now, we
 are
farther away from understanding growth in $\SO_{n-1}$ ($n>5$) than in $\SL_n$.

Thus, one must either study all groups of Lie type together (since
they
are all isomorphic to some algebraic subgroup of $\SL_{n-1}$ for some $n$) or
find
a way to do things so that one needs to examine only those subgroups
of $\SL_{n-1}$ that are more or less isomorphic to products of copies
of
$\SL_m$, $m\leq n-1$ (times something they act on).
There seems  to be a way to carry out the latter plan -- and thus arrive at
results for $\SL_n$ before the available techniques can be
successfully modified to work for $\SO_n$ -- but substantial technical
difficulties remain.

Needless to say, what we have just discussed makes sense only if we
aim at a statement like the main theorem in the present paper -- that
is,
a statement valid for sets $A\subset G(K)$ that generate $G(K)$. If we
do
not require that $A$ generate $G(K)$, then, by definition, proving
growth in $G(K)$ involves proving growth in all subgroups of $G(K)$,
algebraic or not. (There will be some subgroups where growth does not
actually happen, namely, solvable groups; they are not the real
difficulty.) Thus, for example, proving an analogue of Theorem
\ref{thm:qartay} for
$G(K) = \SL_n(\mathbb{F}_q)$
would involve proving growth in all finite groups (with bounds allowed
to depend only on $n$,
where $n$
is the dimension of the smallest faithful representation over 
$\mathbb{F}_q$ of the finite group in question). This seems to be far
away, and will
probably be rather cumbersome once it becomes possible: it would have to
involve
the classification of finite simple groups.

\begin{center}
* * *
\end{center}

The main theorem is still true if $\mathbb{Z}/p\mathbb{Z}$ is replaced by
$\mathbb{R}$ or
$\mathbb{C}$; it is easy to modify the proof slightly to show as much.
(Sum-product results over $\mathbb{R}$ and $\mathbb{C}$ are older and stronger
than those over finite fields.) However, this would arguably not be
the right generalisation to $\mathbb{R}$ or $\mathbb{C}$. What is needed
for results on expansion is a statement on convolutions of measures;
in the case of $\mathbb{Z}/p\mathbb{Z}$, such statements follow from
results such as the main theorem -- namely, results on multiplication of
sets -- but, for infinite fields such as $\mathbb{C}$,
one needs to start from stronger results. Over $\mathbb{C}$, one should
show that $A\cdot A\cdot A$ not only has more elements than $A$, but,
furthermore, has more elements that are at a certain distance 
from each other. Bourgain and Gamburd \cite{BG2} showed how
to strengthen the proof in \cite{He} accordingly in the case of $\SU(2)$.
It remains to be seen how difficult it will be to do the same to the
proof in the present paper.

We finish by remarking that a recent result of Breuillard \cite{Br}
gives some hope that certain results that depend on the assumption of large girth
(such as those of Bourgain and Gamburd) may some day be proven without that 
assumption.


\begin{thebibliography}{EMO}
\bibitem[Ba]{Ba}
{L.\ Babai}, On the diameter of Eulerian orientations of a graph,
{\em Proc.\ 17th Ann.\ Symp.\ on Discr.\ Alg.\ (SODA '06)}, ACM-SIAM 2006,
pp.\ 822--831.
\bibitem[BNP]{BNP}
{L.\ Babai}, {N.\ Nikolov} and {L.\ Pyber}, Product growth and mixing in
finite groups, {\em Proceedings of the nineteenth annual ACM-SIAM symposium on
discrete algorithms}, SIAM, Philadelphia, PA, USA, pp.\ 248--257.
\bibitem[BS]{BS}
{L.\ Babai} and {\'A. Seress}, On the diameter of permutation groups,
{\em European J.\ Combin.\/} {\bf 13} (1992), pp.\  231--243.
\bibitem[BoG]{BoG}
{E.\ Bombieri and W.\ Gubler}, {\em Heights in Diophantine geometry},
Cambridge University Press, 2007.
\bibitem[Bor]{Bor}
{A.\ Borel}, {\em Linear algebraic groups}, 2nd.\ ed., Springer, New York, 1991.
\bibitem[B]{B2}
{J.\ Bourgain}, Mordell's exponential sum estimate revisited,
{\em J.\ Amer.\ Math.\ Soc.} {\bf 18} (2005), pp.\ 477--493.
\bibitem[B2]{B}
{J.\ Bourgain}, Multilinear exponential sums in prime fields under optimal
entropy condition on the sources, preprint.
\bibitem[B3]{Bo}
{J.\ Bourgain}, Sum-product theorems and exponential sum bounds in residue
classes for general modulus, {\em C. R. Math.\ Acad.\ Sci.\ Paris} {\bf 344}
(2007), no.\ 6, pp.\ 349--352.
\bibitem[BG]{BG}
{J.\ Bourgain}  and {A.\ Gamburd}, Uniform expansion bounds for Cayley graphs
of $\SL_2(\mathbb{F}_p)$, {\em Ann.\ of Math.} {\bf 167} (2008), no.\ 2,
pp.\ 625--642.
\bibitem[BG2]{BG2}
{J.\ Bourgain}  and {A.\ Gamburd}, On the spectral cap for finitely generated
subgroups of $\SU(2)$, {\em Invent. Math.} {\bf 171} (2008), no.\ 1, 
pp.\ 83--121.
\bibitem[BKT]{BKT}
{J.\ Bourgain, N.\ Katz},  and {T.\ Tao},
A sum-product estimate in finite fields, and applications,
{\em Geom.\ Funct.\ Anal.\/} {\bf 14} (2004), pp.\ 27--57.
\bibitem[Br]{Br}
{E.\ Breuillard}, A strong Tits alternative, preprint, arXiv:0804.1395.
\bibitem[Da]{Da}
{V.\ I.\ Danilov}, Algebraic varieties and schemes,
in V. I. Danilov and V. V. Shokurov, {\em 
Algebraic curves, algebraic manifolds and schemes}, 
{\em Encyclopaedia of mathematical sciences} {\bf 23}, Springer
(1998), pp.\ 167--297.
\bibitem[Di]{Di}
Dickson, L. E., {\em Linear groups, with an exposition of the Galois field
theory,} Teubner, Leipzig, 1901.
\bibitem[Din]{Din}
{O.\ Dinai},  Poly-log diameter bounds for some families of finite groups,
{\it Proc.\ Amer.\ Math. Soc.} {\bf 134} (2006), pp.\ 3137--3142
(electronic).
\bibitem[E]{E}
{G.\ Elekes}, On the number of sums and products, {\em Acta Arith.\! }
{\bf 81} (1997), pp.\ 365--367.
\bibitem[EMO]{EMO}
{A.\ Eskin, S.\ Mozes},  and {H.\ Oh}, 
On uniform exponential growth for linear groups,
{\em Invent.\ Math.\/} {\bf 160} (2005), pp.\  1--30.
\bibitem[ES]{ES}
{P.\ Erd\H{o}s and E.\ Szemer\'edi}, On sums and products of integers,
in {\em Studies in Pure Mathematics; To the memory of Paul Tur\'an},
P.\ Erd\H{o}s, L.\ Alpar, and G.\ Halasz (eds.), Akademiai Kiado --
Birkhauser Verlag, Budapest, 1983, pp.\ 213--218.
\bibitem[Ga]{Ga}
{M.\ Z.\ Garaev}, An explicit sum-product estimate in
$\mathbb{F}_p$, preprint, arxiv:math/0702780.
\bibitem[GK]{GK}
{A.\ A.\ Glibichuk} and {S.\ V.\ Konyagin},
Additive properties of product sets in fields of prime order,
preprint.
\bibitem[Gow]{Gow}
{W.\ T.\ Gowers}, Quasirandom groups, preprint, arXiv:0710.3877.
\bibitem[GR]{GR}
{B.\ Green and I. Z. Ruzsa}, Freiman's theorem in an arbitrary abelian group,
{\em J.\ London Math.\ Soc.\ } {\bf 75} (2007), pp.\ 163--175.       
\bibitem[He]{He}
{H.\ Helfgott}, Growth and generation in $\SL_2(\mathbb{Z}/p\mathbb{Z})$,
{\em Ann. of Math.} {\bf 167} (2008), no.\ 2, pp.\ 601--623.
\bibitem[Hi]{Hi}
{R.\ Hill}, personal communication.
\bibitem[Hum]{Hum}
{J.\ Humphreys}, {\em Linear algebraic groups}, Springer, New York, 1975.
\bibitem[Ki]{Ki}
{O. H. King}, The subgroup structure of finite classical groups in terms of geometric configurations, {\em Surveys in combinatorics 2005},  
29--56, London Math. Soc.\ Lecture Note Ser. {\bf 327}, Cambridge Univ. Press, Cambridge, 2005. 
\bibitem[KL]{KL}
{P. Kleidman} and {M. Liebeck}, {\em The subgroup structure of the finite classical groups,}
 London Mathematical Society Lecture Note Series {\bf 129}, Cambridge University Press, Cambridge, 1990.
\bibitem[Ko]{Ko}
{S.\ V.\ Konyagin}, A sum-product estimate in fields of prime order, 
arXiv:math.NT/03042147. 
\bibitem[LW]{LW}
{S.\ Lang} and {A.\ Weil}, Number of points of varieties in finite fields,
{\em Amer.\ J.\ Math.} {\bf 76} (1954), 819--827.
\bibitem[Mi]{Mi}
{H. Mitchell}, Determination of the ordinary and modular ternary linear 
groups, {\em Trans. Amer. Math. Soc.} {\bf 12} (1911), no. 2, 207--242. 
\bibitem[NP]{NP}
{N.\ Nikolov} and {L.\ Pyber}, Product decompositions of quasirandom groups
and a Jordan-type theorem, preprint, arXiv:math/0703.5343.
\bibitem[P]{P}
{L.\ Pyber}, personal communication.
\bibitem[SX]{SX}
{P.\ Sarnak} and {X. Xue},    Bounds for multiplicities of automorphic
representations, \emph{Duke Math.\ J.} {\bf 64} (1991), pp.\ 207--227.
\bibitem[ST]{ST}
{E.\ Szemer\'edi and W.\ T.\ Trotter Jr.}, Extremal problems in discrete
geometry, {\em Combinatorica} {\bf 3} (1983), pp.\ 381--392.
\bibitem[T]{T}
{T.\ Tao}, Product set estimates for non-commutative groups, 
preprint, arxiv:math/0601431.
\bibitem[T2]{T2}
{T.\ Tao}, The sum-product phenomenon in arbitrary rings, preprint,
arxiv:math/0806.2497.
\bibitem[TV]{TV}
{T.\ Tao} and {V.\ Vu}, {\em Additive Combinatorics}, {\em
Cambridge Studies in Adv.\ Math.} {\bf 105}, Cambridge Univ.\ 
Press, Cambridge, 2006.
\bibitem[Ta]{Ta}
{D. Taylor}, {\em The geometry of the classical groups}, Heldermann Verlag, Berlin, 1992.
\end{thebibliography}
\end{document}